\documentclass[12pt]{amsart}%
\usepackage{amsfonts}
\usepackage{amsmath}
\usepackage{amssymb}
\usepackage{graphicx}
\usepackage{comment}%
\includecomment{comment1}
\includecomment{newrigid}
\newtheorem{theorem}{Theorem}
\theoremstyle{plain}
\newtheorem{acknowledgement}{Acknowledgement}

\newtheorem{corollary}{Corollary}

\newtheorem{lemma}{Lemma}

\newtheorem{proposition}{Proposition}

\newtheorem{thm}{Theorem}
\theoremstyle{definition}
\newtheorem{definition}{Definition}
\newtheorem{example}{Example}
\newtheorem{remark}{Remark}

\numberwithin{equation}{section}
\begin{document}
\title{Tiling Iterated Function Systems}
\author[L. F. Barnsley]{Louisa F. Barnsley}
\address{Australian National University\\
Canberra, ACT, Australia }
\author[M. F. Barnsley]{Michael F. Barnsley}
\address{Australian National University\\
Canberra, ACT, Australia }
\author[A. Vince]{Andrew Vince}
\address{University of Florida\\
Gainesville, FL, USA \\
 }
\date{26 Nov 2020}

\begin{abstract}
This paper presents a detailed symbolic approach to the study of self-similar
tilings. It uses properties of addresses associated with graph-directed
iterated function systems to establish conjugacy properties of tiling spaces.
Tiles may be fractals and the tiled set maybe a complicated unbounded subset
of $\mathbb{R}^{M}$.

\end{abstract}
\maketitle

\section{Introduction \label{sec:intro}}

This paper presents a symbolic approach to the study of self-similar tilings.
It uses graph-directed iterated function systems to produce and analyze both
classical tilings of $\mathbb{R}^{M}$ and also other generalized tilings of
what may be unbounded fractal subsets of $\mathbb{R}^{M}$. Our primary goal is
to understand conjugacy properties of these tilings.

See \cite{hutchinson} for formal background on iterated function systems (IFS)
and \cite{fatehi} for a recent review. We are concerned with graph directed
IFSs as defined here, but see also \cite{arnoux, bandtTILE, bedfordGraph, das,
dekking, frank2, mauldin, werner}. Terms in this introduction are defined
formally elsewhere in the text.

\subsection{Two examples}

Here we illustrate informally two simple examples of the construction and
properties of what we call \textit{rigid} fractal tilings. We use these
examples to illustrate Theorem \ref{basictheorem}.

Let $A\subset\mathbb{R}^{2}$ be either the filled hexagonal polygon
illustrated in Figure \ref{fig1-intro}(i) or the fractal illustrated in Figure
\ref{fig1-intro}(iii). $A$ in Figure \ref{fig1-intro}(i) satisfies the
equation
\[
A=E_{1}(sA)\cup E_{2}(s^{2}A)
\]
where $0<s$ solves $s^{2}+s-1=0$ and $E_{1}$, $E_{2}$ are the isometries
implied by Figure \ref{fig1-intro}(ii). Likewise, $A$ in Figure
\ref{fig1-intro}(iii) satisfies the same equation, but here $0<s$ solves
$s^{4}+s-1=0$ and $E_{1}$, $E_{2}$ are the isometries implied by Figure
\ref{fig1-intro}(iv). In both cases we say that $A$ is \textit{tiled} by
\textit{copies }of the two \textit{prototiles} $sA$ and $s^{2}A$.

\begin{figure}[ptb]
\centering
\includegraphics[
height=3.6322in,
width=3.8813in
]{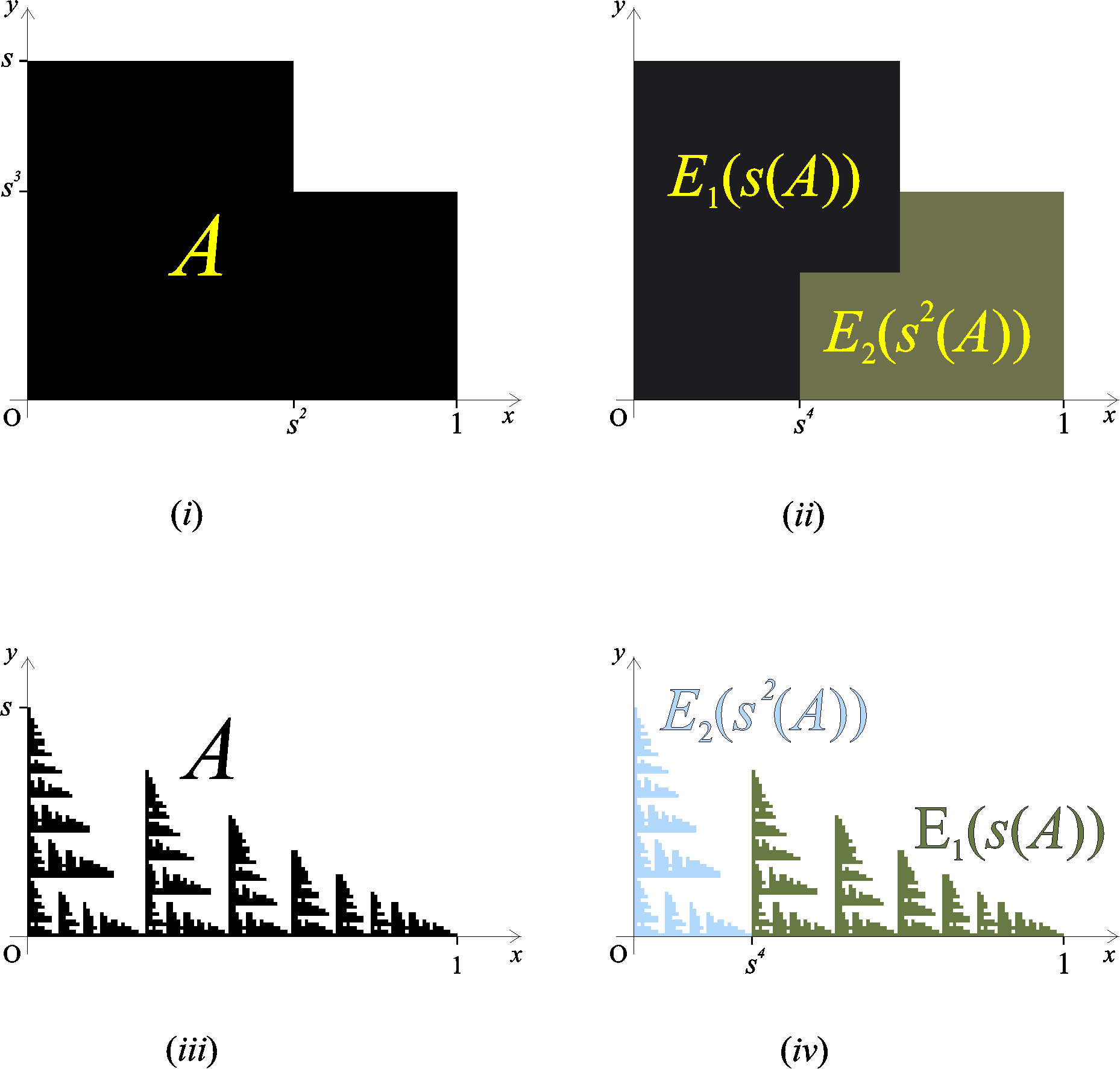}\caption{See text.}%
\label{fig1-intro}%
\end{figure}

The little tilings, in Figures \ref{fig1-intro}(ii) and \ref{fig1-intro}(iv),
share three interesting properties. If we scale up either tiling in Figure
\ref{fig1-intro} by $s^{-1}$, the tile $E_{1}(sA)$ becomes a copy of $A$ and
the tile $E_{2}(s^{2}A)$ becomes a copy of $sA$. The former can be split into
a copy of $sA$ and a copy of $s^{2}A.$ We can repeat this scaling up and
splitting to form, in each case, a sequence of successively strictly larger
tilings as illustrated in Figure \ref{fig2-intro}. We call the tilings in
these sequences \textit{canonical }tilings, $\{T_{n}\}.$

\begin{figure}[ptb]
\centering
\includegraphics[
height=1.9593in,
width=3.8829in
]{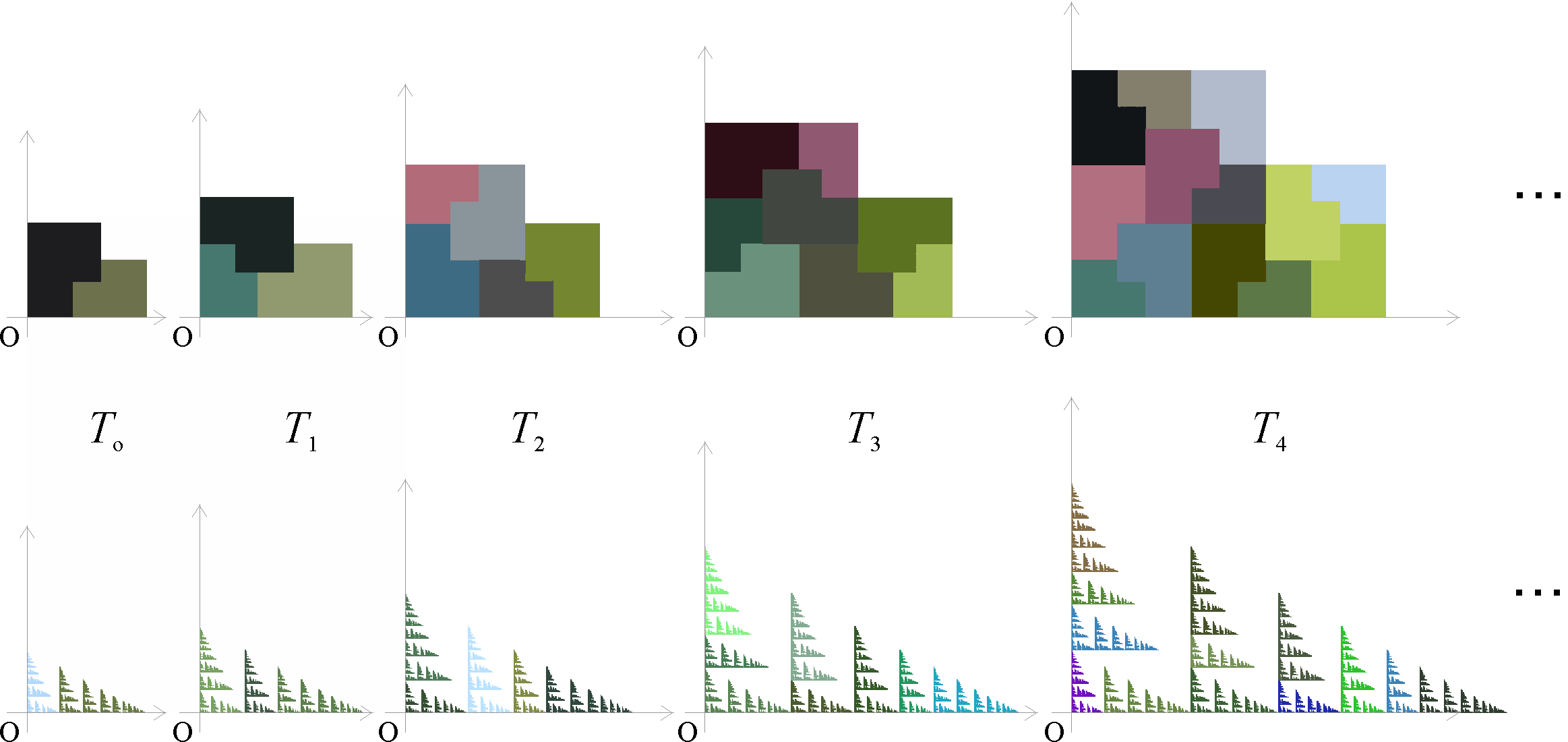}\caption{This illustrates the sequences of canonical tilings
$\{T_{n}\}$ associated with two examples. See text.}%
\label{fig2-intro}%
\end{figure}

The second interesting property is this. Consider the sequence of canonical
tilings $\{T_{n}\}$ in the first example. Let $E_{1}$ and $E_{2}$ be Euclidean
transformations on $\mathbb{R}^{2}$. Suppose $E_{1}T_{k}\cap$ $E_{2}T_{l}$ is
nonempty and tiles $E_{1}s^{-k}A\cap E_{2}s^{-l}A$. Then either $E_{1}%
T_{k}\subset E_{2}T_{l}$ or $E_{2}T_{l}\subset E_{1}T_{k}$. This is a
consequence of the observation that if $s^{k}T_{0}\cap$ $ET_{0}$ is nonempty
for some integer $k$ and some isometry $E,$ then $E=I$ and $k=0$. We say that
the little tiling in Figure \ref{fig1-intro}(ii) is \textit{rigid} (with
respect to Euclidean transformations). Similarly, we say that the tiling in
Figure \ref{fig1-intro}(iv) is rigid (with respect to non-flip Euclidean transformations).

The third interesting property is this. There are non-denumerably many
different infinite sequences of isometries $\{E_{k_{n}}\},$ where $\{k_{n}\}$
is a subsequence of the positive integers, such that $E_{k_{n}}T_{k_{n}%
}\subset E_{k_{n+1}}T_{k_{n+1}}$. This enables us to define an unbounded
tiling $T(\{E_{k_{n}}\}):=%
{\textstyle\bigcup}
E_{k_{n}}T_{k_{n}}$ corresponding to the sequence $\{E_{k_{n}}\}.$

This illustrates informally a generalization of a standard construction
\cite{anderson} of self-similar tilings that applies both to classical
tilings, as defined by Grunbaum and Sheppard \cite{grun}, and to certain
purely fractal tilings. It also illustrates the notion of rigid tilings.

As in self-similar tiling theory, a key question is: When does $T(\{E_{k_{n}%
}\})$ $=ET(\{E_{k_{n}^{\prime}}^{\prime}\})$ for some isometry $E$? Theorem
\ref{basictheorem} (below) answers this question for the case of rigid tilings.

To informally explain Theorem \ref{basictheorem}, we redefine the tilings
$T(\{E_{k_{n}}\})$ using the language of iterated function systems (IFS). Each
of the above examples is associated with a pair of contractive similitudes
that comprise an IFS $\{f_{1}:\mathbb{R}^{2}\rightarrow\mathbb{R}^{2}%
,f_{2}:\mathbb{R}^{2}\rightarrow\mathbb{R}^{2}\}$ such that there is a fixed
$0<s<1$ so that, for all $k\in\mathbb{N=}\{1,2,...\},$ for all $x\in
\mathbb{R}^{2}$,
\[
f_{\theta_{1}}^{-1}\circ f_{\theta_{2}}^{-1}\circ...\circ f_{\theta_{k}}%
^{-1}x=s^{-\xi}U(\theta_{1},\theta_{2},...,\theta_{k})x+t(\theta_{1}%
,\theta_{2},...\theta_{k})
\]
where $\xi=\theta_{1}+\theta_{2}+...+\theta_{k},$ $U$ is a unitary
transformation on $\mathbb{R}^{2}$ and $t$ is a translation, both dependent
only on $(\theta_{1}\theta_{2}...\theta_{k})\in\{1,2\}^{k}$.

We define a family of partial tilings in terms of canonical tilings by%
\[
\Pi(\theta_{1}\theta_{2}...\theta_{k})=f_{\theta_{1}}^{-1}\circ f_{\theta_{2}%
}^{-1}\circ...\circ f_{\theta_{k}}^{-1}s^{\xi}T_{\xi}{}_{{}}%
\]
It is a remarkable and beautiful fact that
\[
\Pi(\theta_{1})\subset\Pi(\theta_{1}\theta_{2})\subset....\subset\Pi
(\theta_{1}\theta_{2}...\theta_{k})
\]
so that for all $\theta_{1}\theta_{2}\theta_{3}...$
\[
\Pi(\theta_{1}\theta_{2}\theta_{3}...):=\cup\Pi(\theta_{1}\theta_{2}%
...\theta_{k})
\]
is a well-defined unbounded tiling of (possibly a subset of) $\mathbb{R}^{2}$.

In Section \ref{relationsec} we establish an equivalence between
representations of tilings in the form $T(\{E_{k_{n}}\})$ with representations
in the form $\Pi(\theta_{1}\theta_{2}\theta_{3}...)$.

Our question "When does $T(\{E_{k_{n}}\})$ $=ET(\{E_{k_{n}^{\prime}}^{\prime
}\})?"$ becomes: "When does $\Pi(\theta_{1}\theta_{2}\theta_{3}...)=E\Pi
(\psi_{1}\psi_{2}\psi_{3}...)?$".

\textbf{Theorem }\ref{basictheorem}: \textit{Let }$\left(  \mathcal{F}%
,\mathcal{G}\right)  $\textit{ be a tiling IFS.}

\textit{(i) If }$\theta,\psi\in\Sigma_{\infty}^{\dag}$\textit{, }$S^{p}%
\theta=S^{q}\psi,$\textit{ }$E=f_{-\theta|p}(f_{-\psi|q})^{-1},$\textit{
}$\left(  \theta|p\right)  ^{+}=\left(  \psi|q\right)  ^{+}$\textit{, and
}$\xi\left(  \theta|p\right)  =\xi\left(  \psi|q\right)  ,$\textit{ then }%
$\Pi(\theta)=E\Pi(\psi)$\textit{ where }$E$\textit{ is an isometry.}

\textit{(ii) Let }$\left(  \mathcal{F},\mathcal{G}\right)  $\textit{ be rigid,
and let }$\Pi(\theta)=E\Pi(\psi)$\textit{ where }$E\in\mathcal{U}$\textit{ is
an isometry, for some pair of addresses }$\theta,\psi\in\Sigma_{\infty}^{\dag
}.$\textit{ Then there are }$p,q\in N$\textit{ such that }$S^{p}\theta
=S^{q}\psi,$\textit{ }$E=f_{-\left(  \theta|p\right)  }(f_{-\left(
\psi|q\right)  })^{-1},\left(  \theta|p\right)  ^{+}=\left(  \psi|q\right)
^{+}$\textit{, and }$\xi\left(  \theta|p\right)  =\xi\left(  \psi|q\right)  .$

The statement of Theorem \ref{basictheorem} involves terms that are defined
precisely in Sections 2 and 3. In the present context: $\left(  \mathcal{F}%
,\mathcal{G}\right)  $ is the IFS $\mathcal{F=}\{f_{1},f_{2}\}$ with a
directed graph $\mathcal{G}$ with two edges and one vertex, $\Sigma_{\infty
}^{\dag}=\{1,2\}^{\infty}$, $S^{p}\theta=\theta_{p+1}\theta_{p+2}...$,
$f_{-\theta|p}=f_{\theta_{1}}^{-1}\circ f_{\theta_{2}}^{-1}\circ...\circ
f_{\theta_{p}}^{-1},$ $\left(  \theta|p\right)  ^{+}=\left(  \psi|q\right)
^{+}=A$, $\xi\left(  \theta|p\right)  =\theta_{1}+\theta_{2}+...+\theta_{p},$
and $\xi\left(  \theta|p\right)  =\psi_{1}+\psi_{2}+...+\psi_{q}$. For the
first example, $\mathcal{U}$ is a set of Euclidean transformations. For the
second example $\mathcal{U}$ is the set of non-flip Euclidean transformations
together with the set of transformations described in part (i) of Theorem
\ref{keythm}.

Part (i) of Theorem \ref{basictheorem}, in the present context, asserts that
if there are positive integers $p$ and $q$ so that $\theta_{1}+\theta
_{2}+...+\theta_{p}=\psi_{1}+\psi_{2}+...+\psi_{q}$ and $\theta_{p+i}=$
$\psi_{q+i}$ for all $i\in\mathbb{N}$, then $\Pi(\theta_{1}\theta_{2}%
\theta_{3}...)=E\Pi(\psi_{1}\psi_{2}\psi_{3}...)$ with $E=f_{\theta_{1}}%
^{-1}\circ f_{\theta_{2}}^{-1}\circ...\circ f_{\theta_{p}}^{-1}\circ
f_{\psi_{q}}^{{}}\circ f_{\psi_{q-1}}^{{}}...\circ f_{\psi_{1}}^{{}}$. Part
(ii) of Theorem \ref{keythm} asserts how, for rigid systems, these conditions
are also necessary.

This completes our informal introduction to the central result in this paper.

\subsection{Main result and related work}

In fact, the tilings just described are associated with inverse limits as in
\cite{sadun}$.$ In the body of this paper a tiling $\Pi(\theta)$ is associated
with a path $\theta$ of a directed graph and an IFS. The general question is:
when does $\Pi(\theta)=E\Pi(\psi)$ for some pair of paths $\theta$ and $\psi$
and some isometry $E$? This question lies at the back of many ideas related to
homology, spectral theory of operators defined on tiling spaces, and
non-commutative geometry. See for example the brief overview in \cite{frank}%
$.$ Our main result is Theorem \ref{basictheorem} which explains exactly when
$\Pi(\theta)=E\Pi(\psi)$ for rigid systems. Much of the work in this paper is
to set up the framework, to define the tilings $\Pi(\theta)$ and describe some
of their basic properties.

There are relationships between this work and Solomyak \cite{solomyak,
solomyak2}, and Anderson and Putnam \cite{anderson}, and many other works on
tiling theory. However our approach to the construction of tilings is more
general because we include purely fractal tilings, where tiles may have empty
interiors, as well as more standard self-similar tilings. Our methods are
based on addresses associated with graph directed IFS and mappings from these
addresses into tilings and tiling spaces.

We mention the work of Pearse \cite{pearse1} and Pearse and Winter
\cite{pearse2} concerning tilings of the convex hull of attractors of IFSs. We
do not discuss their construction here. But we note that it is relevant
because their canonical tilings may be extended to tilings of $\mathbb{R}^{2}$
by taking inverse limits. It appears that such tilings may cover the
complements of the supports of some of the tilings we discuss.

We mention the recent work of Smilansky and Solomon \cite{smilansky}
concerning non commensurable tilings of $\mathbb{R}^{2}$. While we do not
discuss non commensurable fractal tilings in this paper, we note that such
tilings may be described symbolically via a natural extension of the present
framework, along the lines of \cite{barnvince2019}.

This paper extensively develops \cite{barnsleyvince} which concerns tilings
derived from attractors of IFSs with trivial graphs. Here we generalize to
graph-directed IFSs and show that a certain property, rigidity, implies a
specific equivalence class structure on the tiling space. In the context of
standard self-similar tiling theory, as considered for example in
\cite{anderson, solomyak}, rigidity is largely equivalent to the unique
composition property\ and to recognizability. But it is a more general
geometrical notion and it also applies to purely fractal structures. It is
also related to, but distinct from, the notion of measure rigidity in IFS
theory \cite{hoch}. Our main results are Theorems \ref{basictheorem} and
\ref{lastthm}. They describe the possible conjugacy classes of isometries
applied to rigid fractal tilings.

This paper is the completion of \cite{barnvince2018}, which initiated our
study of graph-directed fractal tiling theory. It has relationships with
\cite{barnvince2019}, which uses graph-theoretic language and what we call
tiling hierarchies. Here the point of view is that of iterated function
systems, addresses\ and certain supertiles called canonical tilings. This
paper goes much further than \cite{barnvince2019}. For example it considers
purely fractal tilings, continuity properties of the map from addresses\ to
tilings, the formal description of canonical tilings in terms of addresses,
and the relationship between rigidity and recognizability.

\subsection{Outline}

Section \ref{sec:one} introduces notation and concepts needed throughout. We
define a graph IFS $(\mathcal{F},\mathcal{G)}$ where $\mathcal{F}$ is an IFS
on $\mathbb{R}^{M}$, $\mathcal{G}$ is a directed graph, and paths in
$\mathcal{G}$ correspond to allowed compositions of functions in $\mathcal{F}%
$. We present our notation for paths $\Sigma$ in $\mathcal{G}$ and paths
$\Sigma^{\dag}$ in $\mathcal{G}^{\dag},$ the reversed graph. Definition
\ref{attractordef} specifies the attractor $A$ of $(\mathcal{F},\mathcal{G)}$
and the address map $\pi$ which takes paths and vertices of $\mathcal{G}$ to
subsets of $A.$ Theorem \ref{thm:two} states the continuity properties of
$\pi$. Definition \ref{disjunctivedef} defines disjunctive paths. These are
infinite paths that visit all vertices in every allowed order. Insight into
the relationship between disjunctive points and $A$ is provided by Theorem
\ref{theoremthree} and underlies a simple chaos game algorithm, generalizing
\cite{barnsleychaos}$,$ for calculating both $A$ and the tilings discussed in
this paper. Properties of the shift map $S,$ Definition \ref{shiftdef}, acting
on paths and vertices are stated in Theorem \ref{theoremshift}. Subsection
\ref{subsec:disj} introduces a part of dynamical systems theory relevant later
to describing intersections of fractal tiles. In Theorem \ref{ergodictheorem}
the pointwise ergodic theorem is applied to establish that the image under
$\pi$ of the disjunctive points in $\Sigma$ have full measure, for many
natural stationary measures on $A$.

Section \ref{tilingsec} establishes tiling IFSs and their tilings. A
generalized notion of a tiling, to accommodate fractal supports, is described
in Subsection \ref{tilingdefsec}. A tiling IFS is a graph IFS $(\mathcal{F}%
$,$\mathcal{G)}$ with the special conditions in Definition \ref{defONE}. In
particular, it is required that $(\mathcal{F}$,$\mathcal{G)}$ obeys the open
set condition (OSC) in Definition \ref{oscdef}. According to Theorem
\ref{thm:three} a tiling map $\Pi,$ tilings $\Pi(\theta)$ and sets of tiles
associated with paths $\theta\in\Sigma^{\dag}$, are well-defined by Definition
\ref{def10}. Definitions \ref{def6}, \ref{def7}, \ref{def8}, \ref{def9}
describe the critical set, the dynamical boundary, and the inner boundaries of
the attractor of a tiling IFS. These objects, and their relationship with
disjunctive points (they don't contain any), play a key role in describing the
nonempty intersections of the tiles in $\Pi(\theta)$. Some of their properties
are the subject of Theorem \ref{theorem5}, which also provides the Hausdorff
dimensions of the attractor of $(\mathcal{F}$,$\mathcal{G)}$ and the tiles in
$\Pi(\theta)$. Theorem \ref{theorem5} underpins Theorem \ref{thm:three}.

Section \ref{continuitysec} studies continuity properties of the tiling
function $\Pi(\theta).$ A convenient metric $d_{\mathbb{T}}$ on the space of
tilings $\mathbb{T}=$ $\Pi({\Sigma}^{\dag})$ is introduced. Theorem
\ref{compactthm} says that $(\mathbb{T},d_{\mathbb{T}})$ is a compact metric
space, and Theorem \ref{ctytheorem} says that $\Pi:{\Sigma}^{\dag}%
\rightarrow\mathbb{T}$ is upper semi-continuous, but continuous when
restricted to reversible points$,$ a generalization of a notion in
\cite{tilings}. A proof of Theorem \ref{ctytheorem}, using a natural
generalization of central open sets as\ defined by Bandt \cite{bandtneighbor2}%
, is presented.

Section \ref{symbosection} examines the combinatorics of the addresses of
finite tilings in $\Pi(\theta).$ Theorem \ref{lem:struct} relates the
addresses of tiles in $\Pi(\theta),$ where $\left\vert \theta\right\vert $ is
finite, to addresses of tiles in copies of tilings contained in $\Pi(\theta)$.

Section \ref{canonical} introduces canonical tilings. Definition
\ref{canondef} defines the canonical tilings $T_{k}^{v}$ indexed by a vertex
$v\in\mathcal{G}$ and $k\in\mathbb{N}$. All tilings $\Pi(\theta)$ comprise
what we call isometric combinations of canonical tilings. Theorem
\ref{piusingts} gives identities between isometric combinations, and follows
from Theorem \ref{lem:struct}. Canonical tilings, their notation, and related
identities, play a key role in establishing our main results.

Section \ref{quasisec} considers general properties of tilings $\Pi(\theta).$
The notion of a coprime graph and standard properties of tilings such as
quasiperiodic, local isomorphism, and self-similarity, are defined. Theorem
\ref{theoremTHREE} states that if $\mathcal{G}$ is coprime then all tilings
$\Pi(\theta)$ with $\left\vert \theta\right\vert =\infty$ are quasiperiodic,
and that any pair are locally isomorphic; also if $\theta$ is eventually
periodic, then $\Pi(\theta)$ is self-similar. The proof uses earlier
identities involving canonical tilings.

Section \ref{relsec} introduces relative and absolute addresses of canonical
tilings and uses them to establish deflation $\alpha$ and inflation
$\alpha^{-1},$ operators that act on the graph of $\Pi(\theta)$ to produce new
objects. Relative addresses are associated with copies of canonical tilings
$T_{k}^{v}$ and are defined in Definition \ref{relativedef}. Lemma
\ref{bijlemma} notes that the relative addresses of $T_{k}^{v}$ are in
bijection with the subset $\Omega_{k}^{v}$ of ${\Sigma}$. Theorem
\ref{hierarchythm} explains how a relative address is associated with a
hierarchy of canonical tilings. Absolute addresses are also defined and in
Theorem \ref{relabsthm} a relationship between absolute and relative addresses
is exhibited. In Definition \ref{infldef1}, inflation and deflation of
canonical tilings are defined according to $\alpha T_{k}^{v}=T_{k-1}^{v}$.
Finally, Definition \ref{uniondef} supported by Theorem \ref{keythm2}
establishes how the domains of $\alpha$ and $\alpha^{-1}$ can be extended to
include the graph of $\Pi(\theta)$, and how their actions relate to
$\Pi(S\theta)$. This is a key result.

In Section \ref{secrigid} it is pointed out that $\alpha$ may not act
consistently on isometric combinations of canonical tilings. We define rigid
tilings and rigid tiling IFSs, and extend the definitions of $\alpha$ and
$\alpha^{-1}$ so they act consistently on isometric combinations of rigid
canonical tilings. Definition \ref{meetdef} specifies what it means for two
tilings to \textit{meet} and in Definition \ref{localdef} we define what is a
\textit{rigid} tiling. The notion of a rigid tiling is with respect to a set
of isometries $\mathcal{U}$ that act on tiles and tilings. In Lemma
\ref{keylemma} it is explained that for rigid tilings, if certain scaled
copies of canonical tilings meet, then one is contained in the other. Theorem
\ref{equivalence} provides some properties of rigid tilings and gives an
alternative test for rigidity.

Also in Section \ref{secrigid}, the definitions of $\alpha$ and $\alpha^{-1}$
are extended to include local action on isometric combinations of canonical
tilings and on $\Pi(\theta)$ (without regard for $\theta$) so if two isometric
combinations represent the same tiling then $\alpha$ may act consistently
term-by-term to produce the same result, and similarly for $\alpha^{-1}.$ The
local actions of $\alpha$ and $\alpha^{-1}$ on tilings are defined using the
concepts of \textit{large} tiles and \textit{partners}. Theorem
\ref{welldefined} lists properties of $\alpha$ and $\alpha^{-1}$ acting on
isometric combinations of canonical rigid tilings, and leads to Corollary
\ref{corollaryX} which asserts that if $\Pi(\theta)\subset E\Pi(\psi)$ for
some $E\in\mathcal{U}$, then $\alpha^{K}$ can be applied to the two tilings
$\Pi(\theta)$ and $E\Pi(\psi)$ to yield $\alpha^{K}\left(  \Pi(\theta)\right)
\subset\alpha^{K}\left(  E\Pi(\psi)\right)  ,$ without knowing $\theta$ and
$\psi$.

Section \ref{charismasec} arrives at a main result of this paper, concerning
rigid tilings. Theorem \ref{basictheorem} tells us exactly when, for rigid
tilings, $\Pi(\theta)=E\Pi(\psi)$ for some $E\in\mathcal{U}$. The proof uses
properties of relative addresses given in Theorem \ref{keythm}. It is
necessary that the two addresses have a common tail, and $E$ is defined in
terms of the initial parts of the addresses of $\theta$ and $\psi$. Then
Corollary \ref{aperithm} tells us that if $\mathcal{U}$ contains the group of
Euclidean translations, then $\Pi(\theta)$ is not periodic for any infinite
$\theta\in\Sigma^{\dag}.$

Section \ref{notrigidsec} explores consequences of $\Pi(\theta)=E\Pi(\psi),$
where $E$ is some isometry, without requiring rigidity. It contains one
definition and one theorem. The section begins by showing by example that it
can occur that $\Pi(\theta)=E\Pi(\psi)$ implies $\alpha^{K}\left(  \Pi
(\theta)\right)  =\alpha^{K}\left(  E\Pi(\psi)\right)  $ without requiring
rigidity. Such examples are termed \textit{well-behaved} in Definition
\ref{welldef}. For well-behaved tilings, Theorem \ref{lastthm} details the
structure of $E$ such that $\Pi(\theta)=E\Pi(\psi).$ It is essentially
equivalent to Theorem \ref{basictheorem} in the case of rigid tilings.

Section \ref{relationsec} establishes a relationship between this work and
Anderson and Putnam \cite{anderson}. It is proved that the tiling space of
\cite{anderson} is conjugate to $\{E\Pi(\theta):$ specified set of
translations $E$ and addresses $\theta\}$.

\section{\label{sec:one}Foundations}

\subsection{\label{sec:two:one}Graph iterated function systems}

Let $\mathcal{F}$ be a finite set of invertible contraction mappings
$f:\mathbb{R}^{M}\rightarrow\mathbb{R}^{M}$ each with contraction factor
$0<\lambda<1$, that is $\left\Vert f(x)-f(y)\right\Vert \leq\lambda\left\Vert
x-y\right\Vert $ for all $x,y\in\mathbb{R}^{M}$. We suppose%
\[
\mathcal{F=}\left\{  f_{1},f_{2},...,f_{N}\right\}  ,\text{ }N>1
\]
Let $\mathcal{G}=\left(  \mathcal{E},\mathcal{V}\right)  $ be a strongly
connected primitive directed graph with edges $\mathcal{E}$ and vertices
$\mathcal{V}$ with
\[
\mathcal{E=}\left\{  e_{1},e_{2},...,e_{N}\right\}  \text{, }\mathcal{V=}%
\left\{  \upsilon_{1},\upsilon_{2},...,\upsilon_{V}\right\}  ,\text{ }1\leq
V<N
\]
$\mathcal{G}$ is strongly connected means there is a path, a sequence of
consecutive directed edges, from any vertex to any vertex. $\mathcal{G}$ is
primitive means that if $\mathcal{W}$ is the $V\times V$ matrix whose
$ij^{th}$ entry is the number of edges directed from vertex $j$ to vertex $i,$
then there is some power of $\mathcal{W}$ whose entries are all strictly positive.

We call $\left(  \mathcal{F},\mathcal{G}\right)  $ a \textit{graph
IFS}.\textit{\ }The directed graph $\mathcal{G}$ provides some orders in which
functions of $\mathcal{F}$ may be composed. The sequence of successive
directed edges $e_{\sigma_{1}}e_{\sigma_{2}}\cdots e_{\sigma_{k}}$ may be
associated with the composite function%
\[
f_{\sigma_{1}}f_{\sigma_{2}}\cdots f_{\sigma_{k}}:=f_{\sigma_{1}}\circ
f_{\sigma_{2}}\circ\cdots\circ f_{\sigma_{k}}%
\]

\subsection{Notation for paths in $\mathcal{G},$ $\mathcal{G}^{\dag}$ and
compositions of functions}

Let $\mathbb{N}$ be the strictly positive integers and $\mathbb{N}%
_{0}=\mathbb{N\cup\{}0\}$. For $N\in\mathbb{N}$, $\left[  N\right]
:=\{1,2,\ldots,N\}$.

$\Sigma$ is the set of directed paths in $\mathcal{G}$, each with an initial
vertex. A path $\sigma\in\Sigma$ is written $\sigma=\sigma_{1}\sigma_{2}%
\cdots$ corresponding to the sequence of successive directed edges
$e_{\sigma_{1}}e_{\sigma_{2}}\cdots$ in $\mathcal{G}$. The length of $\sigma$
is $\left\vert \sigma\right\vert \in\mathbb{N}_{0}\cup\left\{  \infty\right\}
.$ A metric $d_{\Sigma}$ on $\Sigma$ is
\[
d_{\Sigma}(\sigma,\omega):=2^{-\min\{k\in\mathbb{N}:\tilde{\sigma}_{k}%
\neq\tilde{\omega}_{k}\}}\text{ for }\sigma\neq\omega\text{ }%
\]
where $\tilde{\sigma}_{k}=\sigma_{k}$ for all $k\leq\left\vert \sigma
\right\vert $, $\tilde{\sigma}_{k}=0$ for all $k>\left\vert \sigma\right\vert
$. Then $\left(  \Sigma,d_{\Sigma}\right)  $ is a compact metric space.

The set $\Sigma_{\ast}\subset\Sigma$ is the set of directed paths of finite
lengths, and $\Sigma_{\infty}\subset\Sigma$ is the set of directed paths of
infinite length. For $\sigma\in\Sigma$, let $\sigma^{-}\in\mathcal{V}$ be the
initial vertex and, if $\sigma\in\Sigma_{\ast},$ let $\sigma^{+}\in
\mathcal{V}$ be the terminal vertex; and for $v\in\mathcal{V}$ let
\[
\Sigma_{v}:=\{\sigma\in\Sigma_{\infty}:\sigma^{-}=v\}
\]
For $\sigma\in\Sigma$, $k\in\mathbb{N},$
\[
\sigma|k:=\left\{
\begin{array}
[c]{cc}%
\sigma_{1}\sigma_{2}...\sigma_{k}\text{ } & \text{if }\left\vert
\sigma\right\vert >k\\
\sigma_{1}^{+} & \text{if }\left\vert \sigma\right\vert \leq k
\end{array}
\right.
\]
We try to reserve the symbol $\sigma$ to mean a directed path in $\Sigma.$

$\mathcal{G}^{\dag}=(\mathcal{E}^{\dag},\mathcal{V})$ is the graph
$\mathcal{G}$ modified so that the directions associated with all edges are
reversed. For any edge $e\in\mathcal{G}$, we use the same label $e$ for the
corresponding reversed edge in $\mathcal{G}^{\dag}$. The superscript $\dag$
means that the superscripted object relates to $\mathcal{G}^{\dag}$. For
example, $\mathcal{E}^{\dag}=\mathcal{E}$ is the set of edges of
$\mathcal{G}^{\mathcal{\dag}}$, $\Sigma_{\ast}^{\dag}$ is the set of directed
paths in $\mathcal{G}^{\dag}$ of finite length, ${\Sigma}_{\infty}^{\dag}$ is
the set of directed paths in $\mathcal{G}^{\dag}$, each of which starts at a
vertex and is of infinite length, and ${\Sigma}^{\dag}={\Sigma}_{\ast}^{\dag
}\cup{\Sigma}_{\infty}^{\dag}$. We try to reserve the symbol $\theta$ to mean
a directed path in $\Sigma^{\dag}.$

We refer to the edges in both $\mathcal{E}$ and $\mathcal{E}^{\dag}$ by the
same set of indices $\left\{  1,2,...,N\right\}  $. The vertices in both
$\mathcal{G}$ and $\mathcal{G}^{\dag}$ are referred to using the set of
indices $\left\{  1,2,...,V\right\}  $. Then both $f_{e}$ and the inverse of
$f_{e}$
\[
f_{-e}:=f_{e}^{-1}%
\]
are well-defined for all $e\in\mathcal{E}\cup\mathcal{E}^{\dag}$.

Typically in this paper, $\mathcal{G}$ and $\Sigma$ are associated with
compositions of functions in $\mathcal{F}$, while $\mathcal{G}^{\dag}$ and
$\Sigma^{\dag}$ are associated with compositions of their inverses. We use the
following notation.%

\begin{align*}
f_{\sigma|k}  &  :=\left\{
\begin{array}
[c]{cc}%
f_{\sigma_{1}}f_{\sigma_{2}}\cdots f_{\sigma_{k}}\text{ } & \text{if
}\left\vert \sigma\right\vert >k\\
f_{\sigma_{1}^{+}} & \text{if }\left\vert \sigma\right\vert \leq k
\end{array}
\right.  \text{ for all }\sigma\in\Sigma\\
f_{\sigma}  &  =f_{\sigma_{1}}f_{\sigma_{2}}\cdots f_{\sigma_{\left\vert
\sigma\right\vert }}\text{ for all }\sigma\in\Sigma_{\ast}\\
f_{-\left(  \theta|k\right)  }  &  :=\left\{
\begin{array}
[c]{cc}%
f_{\theta_{1}}^{-1}f_{\theta_{2}}^{-1}\cdots f_{\theta_{k}}^{-1}\text{ } &
\text{if }\left\vert \theta\right\vert >k\\
f_{-\theta_{1}^{+}} & \text{if }\left\vert \theta\right\vert \leq k
\end{array}
\right.  \text{ for all }\theta\in\Sigma^{\dag}\\
f_{-\theta}  &  =f_{\theta_{1}}^{-1}f_{\theta_{2}}^{-1}\cdots f_{\theta
_{\left\vert \sigma\right\vert }}^{-1}\text{ for all }\theta\in\Sigma_{\ast
}^{\dag}%
\end{align*}
We define $f_{v}=f_{-v}=\chi_{A_{v}}$ for all $v\in\mathcal{V}$ where
$\chi_{A_{v}}$ is the characteristic function of $A_{v}\subset\mathbb{R}^{M}$,
see Definition \ref{attractordef}(iii).

\subsection{Addresses and Attractors}

Let $\mathbb{H}$ be the nonempty compact subsets of $\mathbb{R}^{M}$ and let
$d_{\mathbb{H}}$ be the Hausdorff metric. Singletons in $\mathbb{H}$ are
identified with points in $\mathbb{R}^{M}$.

\begin{definition}
\label{attractordef}The \textbf{attractor} $A$ of the graph IFS $\mathbb{(}%
\mathcal{F},\mathcal{G)}$, its \textbf{components} $A_{v}$, and the
\textbf{address }map $\pi:\Sigma\cup\mathcal{V\rightarrow}\mathbb{H},$ are
defined as follows.
\begin{align*}
\text{(i)}\ \pi(\sigma)  &  :=\lim_{k\rightarrow\infty}f_{\sigma|k}(x)\text{
for }\sigma\in\Sigma_{\infty},\text{ independently of }x\in\mathbb{R}^{M}\\
\text{(ii) }\,\text{\ }\,\,\text{\ \ }A  &  :=\pi(\Sigma_{\infty})\\
\text{(iii)}\ \mathcal{\pi}(v)  &  :=A_{v}:=\pi(\Sigma_{v})\text{ for all
}v\in\mathcal{V}\\
\text{(iv)}\ \mathcal{\pi}(\sigma)  &  :=f_{\sigma}(A_{\sigma^{+}})\text{ for
all }\sigma\in\Sigma_{\ast}%
\end{align*}

\end{definition}

\begin{example}
Let $\mathcal{F}=\{\mathbb{R}^{M};f_{1},f_{2},f_{3},f_{4}\}$ where each
$f_{i}:\mathbb{R}^{M}\rightarrow\mathbb{R}^{M}$ is a contraction. Let
$\mathcal{G}$ be the directed graph with four edges $\{e_{1},e_{2},e_{3}%
,e_{4}\}$ and two vertices $\{v_{1},v_{2}\},$ where $e_{1}$ is directed from
$v_{1}$ to $v_{1},$ $e_{2}$ is directed from $v_{1}$ to $v_{2}$, $e_{3}$ is
directed from $v_{2}$ to $v_{1},$ and $e_{4}$ is directed from $v_{2}$ to
$v_{2}$. Then $A=A_{1}\cup A_{2}$ where $(A_{1},A_{2})$ is the unique pair of
nonempty closed bounded subsets of $\mathbb{R}^{M}$ such that
\begin{align*}
f_{1}(A_{1})\cup f_{2}(A_{2})  &  =A_{1}\\
f_{3}(A_{1})\cup f_{4}(A_{2})  &  =A_{2}%
\end{align*}
and $\pi(243)=f_{2}f_{4}f_{3}(A_{1}).$ Also $\pi(111...)=$ $\pi(\overline{1})$
is the singleton fixed point of $f_{1}$. For instance, if $M=1,$
$f_{1}(x)=0.5x,$ $f_{2}(x)=0.5x-0.5,$ $f_{3}(x)=0.5x+2,$ and $f_{4}%
(x)=0.5x+1.5,$ then $A_{1}=[0,1],$ $A_{2}=[2,3],$ $A=[0,1]\cup\lbrack2,3],$
$\pi(243)=f_{2}f_{4}f_{3}([0,1])=[0.75,0.875]$ and $\pi(\overline{1})=\left\{
0\right\}  $.
\end{example}

\begin{thm}
\label{thm:two} Let $(\mathcal{F},\mathcal{G)}$ be a graph IFS.

(1) $\pi:\Sigma\cup\mathcal{V\rightarrow}\mathbb{H}$ is well-defined and
independent of $x\in\mathbb{R}^{M}$

(2) $\pi:\Sigma\cup\mathcal{V\rightarrow}\mathbb{H}$ is continuous

(3) $\pi(\sigma)=\bigcap\limits_{k=1}^{\left\vert \sigma\right\vert }%
\pi(\sigma|k)$ for all $\sigma\in\Sigma$

(4) $f_{\sigma}(A_{\sigma^{+}})\subset A_{\sigma^{-}}$ for all $\sigma
\in\Sigma_{\ast}$
\end{thm}

\begin{proof}
(1) For all $\sigma\in\Sigma_{\infty}$, $\pi(\sigma)$ is well-defined by (i),
independently of $x$, because $\mathcal{F}$ is strictly contractive
\cite{hutchinson}. It follows that $A$ is well-defined by (ii). Also it
follows that $A_{v}$ and $\pi(v)$ are well-defined by (iii), for all
$v\in\mathcal{V}$ . In turn, $\mathcal{\pi}(\sigma)$ is well-defined for all
$\sigma\in\Sigma_{\ast}$ by Definition \ref{attractordef}(iv). (2) $\pi$ is
continuous because for all $\sigma\in\Sigma_{\infty}$
\[
d_{\mathbb{H}}(\pi(\sigma|k),\pi(\sigma|l))\leq\lambda^{\min\{k,l\}}\max
_{v,w}d_{\mathbb{H}}(A_{v},A_{w})
\]
(3) and (4) follow from Definition \ref{attractordef}(iv).
\end{proof}

\begin{definition}
\label{disjunctivedef}Define $\sigma\in\Sigma_{\infty}$ to be
\textbf{disjunctive} if, given any $\omega\in\Sigma_{\ast}$, there is
$p\in\mathbb{N}$ so that $\omega=\sigma_{p}\sigma_{p+1}...\sigma_{p+\left\vert
\omega\right\vert }$.
\end{definition}

Similarly, $\theta\in\Sigma_{\infty}^{\dag}$ is disjunctive if, given any
$\varphi\in\Sigma_{\ast}^{\dag}$, there is $p\in\mathbb{N}$ so that
$\varphi=\theta_{p}\theta_{p+1}...\theta_{p+\left\vert \varphi\right\vert }$.

\begin{theorem}
\label{theoremthree} Let $(\mathcal{F},\mathcal{G)}$ be a graph IFS. Let
$\theta\in\Sigma_{\infty}^{\dag},$ $x_{0}\in\mathbb{R}^{M},$ and
$x_{n}=f_{\theta_{n}}(x_{n-1})$ for all $n\in\mathbb{N}$. Then
\[
\bigcap\limits_{k\in\mathbb{N}}\overline{(%
{\displaystyle\bigcup\limits_{n=k}^{\infty}}
x_{n})}\subseteq A
\]
with equality when $\theta\in\Sigma_{\infty}^{\dag}$ is disjunctive.
\end{theorem}

\begin{proof}
$\Omega(\left\{  x_{n}:n\in\mathbb{N}\right\}  \mathbb{)}:=\bigcap
\limits_{k\in\mathbb{N}}\overline{(%
{\displaystyle\bigcup\limits_{n=k}^{\infty}}
x_{n})}$ is an $\Omega-$limit set. Specifically it is the set of accumulation
points of $\left\{  x_{n}:n\in\mathbb{N}\right\}  $ in $\mathbb{R}^{M}$. Since
$\pi$ is continuous%
\begin{align*}
\Omega\left(  \left\{  x_{n}:n\in\mathbb{N}\right\}  \right)   &
=\Omega\left(  \left\{  f_{\theta_{n}\theta_{n-1}\cdots\theta_{1}}(x_{0}%
):n\in\mathbb{N}\right\}  \right) \\
&  =\pi(\Omega\left(  \left\{  \theta_{n}\theta_{n-1}\cdots\theta_{1}%
:n\in\mathbb{N}\right\}  \right)  )
\end{align*}
The $\Omega-$limit set of $\left\{  \theta_{n}\theta_{n-1}\cdots\theta
_{1}:n\in\mathbb{N}\right\}  $ is contained in or equal to $\Sigma_{\infty},$
with equality when $\theta\in\Sigma_{\infty}^{\dag}$ is disjunctive.
\end{proof}

\subsection{Shift maps}

The shift map as defined here acts continuously on $\Sigma\cup\mathcal{V}$ and
commutes with $\pi$ according to Theorem \ref{theoremshift} (4). It is used in
Sections \ref{relsec} and \ref{notrigidsec}.

\begin{definition}
\label{shiftdef}The \textbf{shift map}\textit{\ }$S:\Sigma\cup\mathcal{V}%
\rightarrow\Sigma\cup\mathcal{V}$ is defined by $S(\sigma_{1}\sigma_{2}%
\cdots)=\sigma_{2}\sigma_{3}\cdots$ for all $\sigma\in\Sigma,Sv=v$ for all
$v\in\mathcal{V}$, with the conventions
\[
S^{k}\sigma=\sigma|k=\sigma_{1}^{+}\text{ when }k\geq\left\vert \sigma
\right\vert
\]

\end{definition}

\begin{theorem}
\label{theoremshift}Let $(\mathcal{F},\mathcal{G)}$ be a graph IFS.

(1) $S:\Sigma\cup\mathcal{V}\rightarrow\Sigma\cup\mathcal{V}$ is well-defined

(2) $S(\Sigma\cup\mathcal{V)}=\Sigma\cup\mathcal{V}$

(3) $S:\Sigma\cup\mathcal{V}\rightarrow\Sigma\cup\mathcal{V}$ continuous

(4) $f_{\sigma|k}\circ\pi\circ S^{k}\left(  \sigma\right)  =\pi\left(
\sigma\right)  $ for all $\sigma\in\Sigma$, for all $k\in\mathbb{N}_{0}$
\end{theorem}

\begin{proof}
(1) and (2) can be checked. (3) $S$ is continuous at every point in
$\Sigma_{\ast}\cup\mathcal{V}$ because this subset of $\Sigma\cup\mathcal{V}$
is discrete and it is mapped onto itself by $S$. A calculation using the
metric $d_{\Sigma}$ proves that $S$ is continuous at every point in
$\Sigma_{\infty}$. (4) If $\sigma=\sigma_{1}$ and $k=0$ then%
\[
f_{\sigma|k}\circ\pi\circ S^{k}\left(  \sigma\right)  =\chi_{A_{^{\sigma
_{1}^{+}}}}\circ\pi\left(  \sigma_{1}^{+}\right)  =\chi_{A_{^{\sigma_{1}^{+}}%
}}(A_{\sigma_{1}^{+}})=\pi\left(  \sigma_{1}^{+}\right)
\]
If $\sigma=\sigma_{1}$ and $k=1,$ then
\[
f_{\sigma|k}\circ\pi\circ S^{k}\left(  \sigma\right)  =f_{\sigma_{1}}\circ
\pi\left(  \sigma_{1}^{+}\right)  =f_{\sigma_{1}}(A_{\sigma_{1}^{+}}%
)=\pi\left(  \sigma_{1}\right)
\]
If $\sigma\in\Sigma_{\infty}$ and $k\in\mathbb{N}$, then
\begin{align*}
f_{\sigma|k}\circ\pi\circ S^{k}\left(  \sigma\right)   &  =f_{\sigma_{1}%
\sigma_{2}\cdots\sigma_{k}}(\pi(\sigma_{k+1}\sigma_{k+2}\cdots))\\
&  =f_{\sigma_{1}\sigma_{2}\cdots\sigma_{k}}(\lim_{m\rightarrow\infty}%
\pi(\sigma_{k+1}\sigma_{k+2}\cdots\sigma_{m}))\\
&  =\lim_{m\rightarrow\infty}\pi(\sigma_{1}\sigma_{2}\cdots\sigma_{m}%
)=\pi(\sigma)
\end{align*}
The remaining cases follow similarly.
\end{proof}

\subsection{\label{subsec:disj}Disjunctive orbits, ergodicity, subshifts of
finite type}

In this Subsection we discuss some stationary measures associated with
dynamics and Markov processes associated with the attractor of a graph IFS
$(\mathcal{F}$,$\mathcal{G)}$. These measures are useful because they assign
all their mass to the set of images of the disjunctive points. Since points of
intersection between tiles in tilings considered in Section \ref{tilemapsec}
are images of non-disjunctive points, we are able to say how these
intersections are small in a measure theoretic sense. We use this material in
Subsection \ref{criticalsec} in relation to the notions of the
\textquotedblleft interior\textquotedblright\ and the \textquotedblleft
boundary\textquotedblright\ of a tile.

Let $T=S|_{\Sigma_{\infty}}$. The dynamical system $T:\Sigma_{\infty
}\rightarrow\Sigma_{\infty}$ is chaotic in the purely topological sense of
Devaney \cite{devaney}: it has a dense set of periodic points, it is
sensitively dependent on initial conditions, and it is topologically
transitive. Topologically transitive means that if $Q$ and $R$ are open
subsets of $\Sigma_{\infty}$, then there is $K\in\mathbb{N}$ so that
\[
Q\cap T^{K}R\neq\varnothing
\]
This is true because the set of disjunctive points in $\Sigma_{\infty}$ is
dense in $\Sigma_{\infty}$ and the orbit under $T$ of any disjunctive point
passes arbitrarily close to any given point in $\Sigma_{\infty}$.

However, $T:\Sigma_{\infty}\rightarrow\Sigma_{\infty}$ also possesses many
invariant normalized Borel measures, each having support $\Sigma_{\infty}$ and
such that $T$ is ergodic with respect to each. An example of such a measure
$\mu_{\mathcal{P}}$ may be constructed by defining a Markov process on
$\Sigma_{\infty}$ using $\mathcal{G}$ and probabilities $\mathcal{P=}\left\{
p_{e}>0:e\in\mathcal{E}\right\}  $ where $%
{\textstyle\sum\limits_{\substack{d^{+}=e^{+} \\d\in\mathcal{E}}}}
p_{d}=1$ for all $e\in\mathcal{E}$. Then $\mu_{\mathcal{P}}$ is the unique
normalized measure on the Borel subsets $\mathcal{B}$ of $\Sigma_{\infty}$
such that
\[
\mu_{\mathcal{P}}(b)=\sum\limits_{e\in\mathcal{E}}p_{e}\mu_{\mathcal{P}%
}(eb\cap\Sigma_{\infty})\text{ for all }b\in\mathcal{B}%
\]
where $eb:=\{\sigma\in\Sigma_{\infty}:\sigma_{1}=e,S\sigma\in b\}$. In
particular, $\mu_{\mathcal{P}}$ is invariant under $T,$ that is
\[
\mu_{\mathcal{P}}(b)=\mu_{\mathcal{P}}(T^{-1}b)\text{ for all }b\in\mathcal{B}%
\]

The key point (1) in Theorem \ref{ergodictheorem} is well known: $T$ is
ergodic with respect to $\mu$. That is, if $Tb=T^{-1}b$ for some
$b\in\mathcal{B}$, then either $\mu_{\mathcal{P}}(b)=0$ or $\mu_{\mathcal{P}%
}(b)=1 $. As a consequence, the set of disjunctive points has full measure,
independent of $\mathcal{P}.$

\begin{theorem}
\label{ergodictheorem}Let $(\mathcal{F},\mathcal{G)}$ be a graph IFS. Let
$(\Sigma_{\infty},\mathcal{B},T,\mu_{\mathcal{P}})$ be the dynamical system
described above. Let $D$ be the disjunctive points in $\Sigma_{\infty}$. Then

(1) Parry\textbf{\ }\cite{parry}: $(\Sigma_{\infty},\mathcal{B},T,\mu
_{\mathcal{P}})$ is ergodic

(2) $D=TD=T^{-1}D\in\mathcal{B}$

(3) $\mu_{\mathcal{P}}(D)=1,$ and $\mu_{\mathcal{P}}(\Sigma_{\infty}\backslash
D)=0$
\end{theorem}

\begin{proof}
(1) This is a standard result in ergodic theory, see for example \cite{parry}.
(2) It is readily checked that $D\in\mathcal{B}$ and that $T^{-1}D=D=TD$. (3)
Let $\mu=\mu_{\mathcal{P}}.$ Since $(\Sigma_{\infty},\mathcal{B},T,\mu)$ is
ergodic and $D=T^{-1}D,$ it follows that $\mu\left(  D\right)  \in\left\{
0,1\right\}  .$ Also we have
\[
1=\mu\left(  \Sigma_{\infty}\right)  =\mu\left(  D\right)  +\mu\left(
\Sigma_{\infty}\backslash D\right)
\]
So either $\mu\left(  D\right)  =1$ and $\mu\left(  \Sigma_{\infty}\backslash
D\right)  =0$ or vice-versa. Now notice that
\[
\Sigma_{\infty}\backslash D\subset\bigcup\limits_{x\in\Sigma_{\ast}%
\backslash\varnothing}D_{x}%
\]
where $D_{x}=\{\sigma\in\Sigma_{\infty}:S^{n}\sigma\notin c[x]\forall
n\in\mathbb{N}_{0}\}$ where $c[x]$ is the cylinder set%
\[
c[x]:=\left\{  z\in\Sigma_{\infty}:z=xy,y\in\Sigma_{\infty}\right\}  .
\]
In particular%
\[
\mu\left(  \Sigma_{\infty}\backslash D\right)  \leq\sum\limits_{x\in
\Sigma_{\ast}}\mu(D_{x})
\]
But $\mu(D_{x})=0$ as proved next, so $\mu\left(  \Sigma_{\infty}\backslash
D\right)  =0.$ Proof that $\mu(D_{x})=0$: Let $f:\Sigma_{\infty}%
\rightarrow\mathbb{R}$ be defined by $f(\sigma)=0$ if $\sigma\in c[x]$ and
$f(\sigma)=1$ if $\sigma\in\Sigma_{\infty}\backslash c[x]$. Since $f\in
L_{1}(\mu),$ by the ergodic theorem we have%
\[
\int\limits_{\Sigma_{\infty}}fd\mu=\lim_{n\rightarrow\infty}\frac{1}{n}%
\sum\limits_{k=0}^{n-1}f(T^{k}\sigma)\text{ for }\mu\text{-almost all }%
\sigma\in\Sigma_{\infty}.
\]
But $\int fd\mu=1-\mu(c[x])>0$ because the support of $\mu$ is $\Sigma
_{\infty}$, and $\Sigma_{\infty}$ contains a cylinder set disjoint from $c[x]
$ because $\left\vert \mathcal{E}\right\vert \geq2$, and all cylinder sets
have strictly positive measure. Also $f(T^{k}\sigma)=0$ for all $x\in D_{x}$
so%
\[
\lim_{n\rightarrow\infty}\frac{1}{n}\sum\limits_{k=0}^{n-1}f(T^{k}%
\sigma)=0\text{ for all }x\in D_{x}%
\]
so $\int\limits_{\Sigma_{\infty}}fd\mu\neq\lim_{n\rightarrow\infty}\frac{1}%
{n}\sum\limits_{k=0}^{n-1}f(T^{k}\sigma)$ for all $x\in D_{x}$, so $\mu
(D_{x})=0$.
\end{proof}

\section{\label{tilingsec}Tilings}

\subsection{Similitudes}

A \textit{similitude} is an affine transformation $f:{\mathbb{R}}%
^{M}\rightarrow{\mathbb{R}}^{M}$ of the form $f(x)=\lambda\,O(x)+q$, where $O
$ is an orthogonal transformation and $q\in\mathbb{R}^{M}$ is the
translational part of $f(x)$. The real number $\lambda>0$, a measure of the
expansion or contraction of the similitude, is called its \textit{scaling}
\textit{ratio}. An \textit{isometry} is a similitude of unit scaling ratio and
we say that two sets are isometric if they are related by an isometry.

\subsection{\label{TIFSsec}Tiling iterated function systems}

\begin{definition}
\label{oscdef}The graph IFS $(\mathcal{F},\mathcal{G)}$ is said to obey the
\textbf{open set condition} (OSC) if there are non-empty bounded open sets
$\{\mathcal{O}_{v}:v\in\mathcal{V\}}$ such that for all $d,e\in\mathcal{E}$ we
have $f_{e}(\mathcal{O}_{e^{+}})\subset\mathcal{O}_{e^{-}}$ and $f_{e}%
(\mathcal{O}_{e^{+}})\cap f_{d}(\mathcal{O}_{^{d^{+}}})=\varnothing$ whenever
$e^{-}=d^{-}.$
\end{definition}

The OSC for graph IFS is discussed in \cite{boore} and \cite{das}. The paper
\cite{nguyen}, which discusses separation conditions for graph IFS more
generally, notes that many theorems for IFS carry over to graph IFS. We note
that some IFSs which obey the weaker \textit{restricted OSC} can be
transformed into graph IFSs that obey the OSC \cite{das}.

\begin{remark}
Concerning the problem of finding graph IFS that obey either the OSC or the
restricted OSC, we note the impressive digital computer application IFStile
\cite{IFStile}. The system uses exact integer arithmetic over quadratic and
higher order number fields and searches exhaustively over parameter spaces,
using theorems and methodology of Bandt, especially \cite{bandtneighbor1},
using neighbor maps to identify systems that obey the (restricted) OSC.
\end{remark}

\begin{definition}
\label{defONE} Let $\mathcal{F}=\{{\mathbb{R}}^{M};f_{1},f_{2},\cdots,f_{N}%
\}$, with $N\geq2$, be an IFS of contractive similitudes where the scaling
factor of $f_{n}$ is $\lambda_{n}=s^{a_{n}},$ where $0<s<1$ is fixed,
$a_{n}\in\mathbb{N}$ and $\gcd\{a_{1},a_{2},\cdots,a_{N}\}=1$. Let the graph
IFS $(\mathcal{F},\mathcal{G)}$ obey the OSC. Let%
\begin{equation}
A_{v}\cap A_{w}=\varnothing\label{equation01}%
\end{equation}
for all $v\neq w,$ and let the affine span of $A_{v}$ be $\mathbb{R}^{M}$for
all $v\in\mathcal{V}$. Then $\left(  \mathcal{F},\mathcal{G}\right)  $ is
called a\textit{\ }\textbf{tiling iterated function system} (\textit{tiling
IFS). Let }$a_{\max}=\max\{a_{1},a_{2},\cdots,a_{N}\}$.
\end{definition}

The requirement $A_{v}\cap A_{w}=\varnothing$ whenever $v\neq w$ is without
loss of generality in the following sense. By means of changes of coordinates
applied to some of the maps of the IFS, we can move $A_{v}$ to $T_{\upsilon
}A_{v},$ where $T_{\upsilon}:\mathbb{R}^{M}\rightarrow\mathbb{R}^{M}$ is a
translation, while holding $A_{w}$ fixed for all $w\neq v$. To do this$,$ let%
\[
\widetilde{f}_{e}=\left\{
\begin{array}
[c]{cc}%
T_{v}f_{e}T_{v}^{-1} & \text{if }e^{+}=v\text{ and }e^{-}=v\\
T_{v}f_{e} & \text{if }e^{+}\neq v\text{ and }e^{-}=v\\
f_{e}T_{\upsilon}^{-1} & \text{if }e^{+}=v\text{ and }e^{-}\neq v\\
f_{e} & \text{if }e^{+}\neq v\text{ and }e^{-}\neq v
\end{array}
\right.
\]
and let $\widetilde{\mathcal{F}}=\{f_{e}:e\in\mathcal{E\}}$. Then the
components of the attractor of $\left\{  \widetilde{\mathcal{F}}%
,\mathcal{G}\right\}  $ are $\widetilde{A}_{w}=A_{w}$ for $w\neq v$ and
$\widetilde{A}_{v}=T_{v}A_{v}$ for all $v\in\mathcal{V}$. By repeating this
process for each vertex, we can modify the IFS so that different components of
the attractor have empty intersections. Only the relative positions of the
components are changed, while their geometries are unaltered, and
(\ref{equation01}) holds. This being the case, the OSC is simply
\textquotedblleft there are non-empty open sets $\left\{  \mathcal{O}_{v}%
:v\in\mathcal{V}\right\}  $ such that $f_{e}(\mathcal{O}_{e^{+}})\cap
f_{d}(\mathcal{O}_{^{d^{+}}})=\varnothing$ for all $d,e\in\mathcal{V}$ with
$d\neq e$".

\subsection{\label{tilingdefsec}Tilings in this paper}

According to Grunbaum and Sheppard \cite{grun} a tiling is a countable family
of closed sets $\{t_{1},t_{2},...\}$ which cover $\mathbb{R}^{2}$ without gaps
or overlaps. More explicitly, they say that $\mathbb{R}^{2}=\cup\{t_{i}%
:i\in\mathbb{N\}}$ and the sets $t_{i}$ are called tiles. Here we consider
tilings of subsets of $\mathbb{R}^{M}$ such as fractal blow-ups
\cite{strichartz} where tiles are components of attractors of IFSs, which may
have empty interiors, as well as more standard self-similar tilings, such as
tilings of $\mathbb{R}^{2}$ by congruent squares. More precisely we define in
Subsection \ref{tilemapsec} the tiles and tilings we consider. We refer to our
tilings loosely as `fractal tilings'. In Theorem \ref{thm:three} (1) we show
that the intersection of two tiles $t_{1}$ and $t_{2}$ in a fractal tiling is
small both topologically and measure theoretically, relative to the tiles
themselves. This matches the customary situation:\ in a tiling of
$\mathbb{R}^{2}$ by congruent square tiles, tiles have positive
two-dimensional Lebesgue measure, intersections of distinct tiles have zero
two-dimensional Lebesgue measure and are subsets of their topological boundaries.

\subsection{The tiling map\label{tilemapsec}}

Define subsets of $\Sigma_{\ast}$ as follows:%
\begin{align*}
\Omega_{k}  &  =\{\sigma\in\Sigma_{\ast}:\xi^{-}(\sigma)\leq k<\xi
(\sigma)\},\text{ }\Omega_{0}=\mathcal{[}N]\\
\Omega_{k}^{v}  &  =\{\sigma\in\Omega_{k}:\sigma^{-}=v\},\text{ }\Omega
_{0}^{v}=\{\sigma_{1}\in\lbrack N]:\sigma_{1}^{-}=v\}
\end{align*}
for all $k\in\mathbb{N}$, $v\in\mathcal{V}$. Here $\xi:\Sigma_{\ast
}\rightarrow\mathbb{N}_{0}$ is defined for all $\sigma\in\Sigma_{\ast}$ by
\[
\xi(\sigma)=\sum\limits_{k=1}^{\left\vert \sigma\right\vert }a_{\sigma_{k}%
},\text{ }\xi^{-}(\sigma)=\sum\limits_{k=1}^{\left\vert \sigma\right\vert
-1}a_{\sigma_{k}},\text{ }\xi(\varnothing)=\xi^{-}(\varnothing)=0
\]

\begin{definition}
\label{def10} The \textbf{tiling map }$\Pi$ from ${\Sigma}^{\dag}$ to
collections of subsets of $\mathbb{H(R}^{M})$ is defined as follows. For
$\theta\in\Sigma_{\ast}^{\dag},$%
\[
\Pi(\theta)=f_{_{-\theta}}\pi\left(  \Omega_{\xi(\theta)}^{\theta^{+}}\right)
,\text{ }\Pi(\theta|0)=\pi\left(  \Omega_{0}^{\theta^{-}}\right)
\]
and for $\theta\in\Sigma_{\infty}^{\dag}$,%
\[
\Pi(\theta)=\bigcup\limits_{k\in\mathbb{N}}\Pi(\theta|k)
\]
For $\sigma\in\Omega_{\xi(\theta)}^{\theta^{+}}$ and $\theta\in\Sigma^{\dag},$
the set $f_{_{-\theta}}\pi\left(  \sigma\right)  $ is called a \textbf{tile}
and $\Pi(\theta)$ is called a \textbf{tiling}. The \textbf{support} of the
tiling $\Pi(\theta)$ is the union of its tiles, and $\Pi(\theta)$ is said to
tile its support.
\end{definition}

\begin{theorem}
\label{thm:three} Let $(\mathcal{F},\mathcal{G)}$ be a tiling IFS.

\begin{enumerate}
\item For all $\theta\in{\Sigma}_{\infty}^{\dag},$ for each $k\in
\mathbb{N}_{0},$ $\Pi(\theta|k)$ is a well-defined tiling. In particular, if
$t_{1},t_{2}\in\Pi(\theta|k)$ with $t_{1}\neq t_{2}$, then $t_{1}\cap t_{2}$
is small both topologically and measure theoretically, compared to $t_{1}$.
That is, $\mu_{\mathcal{P}}(t_{1}\cap t_{2})=0$ and, if $x=f_{-(\theta|k)}%
(\pi(\sigma))\in t_{1}\cap t_{2}$, for some $\sigma\in\Sigma_{\infty},$ where
$(\theta|k)^{+}=\sigma^{-}$, then $\sigma$ is not disjunctive (i.e. $\sigma
\in\Sigma_{\infty}\backslash\mathcal{D}$).

\item For all $\theta\in{\Sigma}_{\infty}^{\dag}$ the sequence of tilings
$\left\{  \Pi(\theta|k)\right\}  _{k=1}^{\infty}$ obeys
\begin{equation}
\Pi(\theta|0)\subset\Pi(\theta|1)\subset\Pi(\theta|2)\subset\cdots\text{ }
\label{eqthmONE}%
\end{equation}
In particular, $\Pi(\theta)$ is a well-defined tiling for all $\theta
\in{\Sigma}_{\infty}^{\dag}.$

\item $\Pi(\theta)$ is a tiling of a subset of $\mathbb{R}^{M}$ that is
bounded when $\theta\in{\Sigma}_{\ast}^{\dag}$ and unbounded when $\theta
\in{\Sigma}_{\infty}^{\dag}$.

\item For all $\theta\in{\Sigma}_{\infty}^{\dag}$
\begin{equation}
\Pi(\theta)=\lim_{k\rightarrow\infty}f_{_{-(\theta|k)}}(\{\pi\left(
\sigma\right)  :\sigma\in\Omega_{\xi(\theta|k)},\sigma^{-}=\theta^{+}\})
\label{eqn:whole}%
\end{equation}
The limit here is equivalently the union of an increasing sequence (each set
of sets in the sequence is contained in its successor), or the limit with
respect to the metric defined in Section \ref{metricsec}, using the
Hausdorff-Hausdorff metric on a sphere.

\item Any tile $t\in\Pi(\theta)$ can be written $t=s^{m}EA_{v}$ for some
isometry of the form $E=f_{-\theta}f_{\sigma},$ for some $m\in
\{0,1,2,...,a_{\max}-1\},\theta\in{\Sigma}_{\ast}^{\dag}$, $\sigma\in{\Sigma
}_{\ast}$, $\theta^{+}=\sigma^{-}$, $\sigma^{+}=$ $v\in\mathcal{V}$.
\end{enumerate}
\end{theorem}

\begin{proof}
(1) $\Pi(\theta|0)$ is a tiling in the sense described in Section
\ref{tilingdefsec}. $\Pi(\theta|0)=\pi\left(  \Omega_{0}^{\theta^{-}}\right)
=\pi\left(  \left\{  e\in\left[  N\right]  :e^{-}=\theta^{-}\right\}  \right)
=\{f_{e}(A_{e^{+}}):e^{-}=\theta^{-}\}$ has support $A_{e^{-}}$ and its tiles
are supposed to be $\{f_{e}(A_{e^{+}}):e^{-}=\theta^{-}\}$. We need to check
(i) that they are components of attractors of tiling IFSs and (ii) that their
intersections are relatively small. (i) is true because for each
$e\in\mathcal{[}N],$ the set $f_{e}(A_{e^{+}})$ is a component of the
attractor of the tiling IFS $(f_{e}\mathcal{F}f_{e}^{-1},\mathcal{G)}$. (ii)
This is a consequence of the OSC. It follows from Theorem 5 parts (3) and (4).
Similarly, $\Pi(\theta|k)$ and $\Pi(\theta)$ are tilings as in Section
\ref{tilingdefsec}: the tiles are components of attractors of appropriately
shifted versions of the original tiling IFS and their intersections are
isometric to subsets of the critical set of the original tiling IFS. (2) The
proof is algebraic, independent of topology, essentially the same as for the
case where $V=1$ \cite{barnsleyvince}. Briefly,
\begin{align*}
\Pi(\theta|k+1)  &  =\{f_{\theta_{1}}^{-1}...f_{\theta_{k+1}}^{-1}%
f_{\sigma_{1}}..f_{\sigma_{\left\vert \sigma\right\vert }}(A_{\sigma
_{\left\vert \sigma\right\vert }^{+}}):\xi(\sigma_{1}..\sigma_{\left\vert
\sigma\right\vert -1})\leq\xi(\theta_{1}..\theta_{\left\vert \sigma\right\vert
})<\xi(\sigma_{1}..\sigma_{\left\vert \sigma\right\vert })\}\\
&  \supset\{f_{\theta_{1}}^{-1}...f_{\theta_{k}}^{-1}f_{\sigma_{2}}%
..f_{\sigma_{\left\vert \sigma\right\vert }}(A_{\sigma_{\left\vert
\sigma\right\vert }^{+}}):\xi(\sigma_{2}..\sigma_{\left\vert \sigma\right\vert
-1})\leq\xi(\theta_{2}..\theta_{\left\vert \sigma\right\vert })<\xi(\sigma
_{2}..\sigma_{\left\vert \sigma\right\vert })\}\\
&  =\{f_{\theta_{1}}^{-1}...f_{\theta_{k}}^{-1}f_{\sigma_{1}}..f_{\sigma
_{\left\vert \sigma\right\vert -1}}(A_{\sigma_{\left\vert \sigma\right\vert
-1}^{+}}):\xi(\sigma_{1}..\sigma_{\left\vert \sigma\right\vert -2})\leq
\xi(\theta_{1}..\theta_{\left\vert \sigma\right\vert -1})<\xi(\sigma
_{1}..\sigma_{\left\vert \sigma\right\vert -1})\}\\
&  =\Pi(\theta|k)
\end{align*}
(3) For $\theta\in{\Sigma}_{\ast}^{\dag}$, $\Pi(\theta)=f_{_{-\theta}}%
\pi\left(  \Omega_{\xi(\theta)}^{\theta^{+}}\right)  $ so the support of
$\Pi(\theta)$ is $f_{_{-\theta}}(%
{\textstyle\bigcup}
\{\pi(\sigma):\sigma\in\Omega_{\xi(\theta)}^{\theta^{+}})=f_{_{-\theta}%
}A_{\theta^{+}}\}.$ Here $f_{_{-\theta}}$ is a similitude of expansion factor
$\left\vert s\right\vert ^{-\xi(\theta)}$ which diverges with $\left\vert
\theta\right\vert ,$ and $A_{\theta^{+}}$ spans $\mathbb{R}^{M}$. (4) This
follows from (3). (5) For $t\in\Pi(\theta)$ we have $t=f_{-(\theta
|k)}f_{\sigma}(A_{v})$ for some $k,$ $\theta,\sigma$ and $v,$ with $\xi
^{-}(\sigma)\leq\xi(\theta|k)<\xi(\sigma)$. Here $f_{-(\theta|k)}f_{\sigma
}=s^{-m}E$ where $m=\xi(\theta|k)-\xi(\sigma)$ is an integer that lies between
$1$ and $a_{\max}$ and $E$ is an isometry on $\mathbb{R}^{M}$ of the form
$s^{m}f_{-(\theta|k)}f_{\sigma}$ for some $m\in\{1,2,...,a_{\max}\}$.
\end{proof}

\subsection{\label{criticalsec}How tiles in a tiling can intersect: the
dynamical boundary, critical set and inner boundaries}

\begin{definition}
\label{def6}The \textbf{critical set} of the (attractor of the) tiling IFS
$(\mathcal{F},\mathcal{G)}$ is%
\[
\mathcal{C}:\mathcal{=}\bigcup\limits_{\substack{d\neq e\\d,e\in\mathcal{E}%
}}f_{d}(A_{d^{+}})\cap f_{e}(A_{e^{+}})
\]

\end{definition}

\begin{definition}
\label{def7}The \textbf{dynamical boundary }of the (attractor of the) tiling
IFS $(\mathcal{F},\mathcal{G)}$ is $\theta\in\Sigma_{\ast}^{\dag}$%
\[
\partial A:\mathcal{=}\overline{\bigcup\limits_{\theta\in\Sigma_{\ast}^{\dag}%
}f_{-\theta}(A_{\theta^{+}}\cap\mathcal{C})\cap A_{\theta^{-}}}%
\]
where $f_{-\theta}$ is as defined near the start of Subsection
\ref{tilemapsec}.
\end{definition}

If $(\mathcal{F},\mathcal{G)}$ obeys the OSC, then $A\backslash\partial
A\neq\varnothing$. If $A\backslash\partial A\neq\varnothing$, we say that the
tiling IFS is non-overlapping$.$ See the discussions in \cite{bandtneighbor2,
bandt} which also apply to the present situation. We expect that if a tiling
is non-overlapping then it obeys the OSC, but this has not been proven even in
the case $V=1$. We know of no counterexample.

\begin{definition}
\label{def8}The \textbf{inner} \textbf{boundary }of the (attractor of the)
tiling IFS $(\mathcal{F},\mathcal{G)}$ is
\[
\widehat{\mathcal{C}}:\mathcal{=}\bigcup\limits_{\sigma\in\Sigma_{\ast}%
}f_{\sigma}(A_{\sigma^{+}}\cap\mathcal{C})\cap A_{\sigma^{-}}%
\]

\end{definition}

\begin{definition}
\label{def9}The \textbf{inner boundaries to depth }$k\in\mathbb{N}_{0},$ of
the (attractor of the) tiling IFS $(\mathcal{F},\mathcal{G)}$, are
\[
\widehat{\mathcal{C}}_{k}:\mathcal{=}\bigcup\limits_{\sigma\in\Omega_{k}%
}f_{\sigma}(A_{\sigma^{+}}\cap\mathcal{C})\cap A_{\sigma^{-}}\text{ and
}\widehat{\mathcal{C}}_{k}^{v}:\mathcal{=}\bigcup\limits_{\sigma\in\Omega
_{k}^{v}}f_{\sigma}(A_{\sigma^{+}}\cap\mathcal{C})\cap A_{\sigma^{-}},
\]
where $\Omega_{k}$ and $\Omega_{k}^{v}\ $are as defined at the start of
Subsection \ref{tilemapsec}.
\end{definition}

The following theorem tells us that the critical set of a tiling IFS\ is
small, not only topologically, but also measure theoretically, compared to the attractor.

\begin{theorem}
\label{theorem5}Let $\mathcal{(F},\mathcal{G)}$ be a tiling IFS, let
$\mathcal{C}$ be the critical set, $\mathcal{\partial}A$ be the dynamical
boundary, $\widehat{\mathcal{C}}_{k}$ be the inner boundary to depth
$k\in\mathbb{N}$, and let $D$ be the disjunctive points in $\Sigma_{\infty}$.

(1) Bedford\textbf{\ \cite{bedford}} and Mauldin and
Williams\textbf{\ \cite{mauldin}: }The Hausdorff dimension $\mathcal{D}%
_{H}(A)$ of the attractor $A$ of $\mathcal{(F},\mathcal{G)}$\ is the unique
$t\in\lbrack0,M]$ such that the spectral radius of the matrix%
\[
\mathcal{W}_{w,v}(t)=\sum_{\left\{  e\in\mathcal{E}:e^{+}=v,e^{-}=w\right\}
}s^{ta_{e}}%
\]
equals one. Also $0<\mu_{\mathcal{H}}(A)<\infty$ where $\mu_{\mathcal{H}}$ is,
up to a strictly positive constant factor, the Hausdorff measure on $A.$

(2) $\mathcal{\partial}A\cup\widehat{\mathcal{C}}_{k}\subset\pi(\Sigma
_{\infty}\backslash D).$

(3) $(\partial A\cup(\cup_{k\in\mathbb{N}}\widehat{\mathcal{C}}_{k}))\cap
\Pi(D)=\emptyset$, in the relative topology induced on $A$ by the natural
topology of $\mathbb{R}^{M},$ $\partial A$ is closed and $A\backslash\partial
A$ is open.

(4) $\mathcal{\partial}A$ $\mathcal{\cap}A^{\circ}=\emptyset$ where $A^{\circ
}$ is the interior of $A$ as a subset of $\mathbb{R}^{M}$.

(5) $\mathcal{\mu}_{\mathcal{P}}(\pi^{-1}(\mathcal{C)})=0,$ $\mathcal{\mu
}_{\mathcal{P}}(\pi^{-1}(\mathcal{\partial A)})=0,$ $\mathcal{\mu
}_{\mathcal{P}}(\pi^{-1}(\widehat{\mathcal{C}}_{k}\mathcal{)})=0,$ for all
$\mathcal{P}$.

(6) If $%
{\textstyle\sum\limits_{v}}
\mathcal{W}_{w,v}(t)=1$ then $\mu_{\mathcal{H}}=\mu_{\widehat{\mathcal{P}}%
}\circ\pi^{-1}$ where $\mu_{\widehat{\mathcal{P}}}$ is the stationary measure
on $\Sigma_{\infty}$ obtained when $p_{e}=s^{\mathcal{D}_{H}(A)a_{e}}$ in the
Markov process described before Theorem \ref{ergodictheorem}. In this case for
all $k\in\mathbb{N}$
\[
\mu_{\mathcal{H}}(\mathcal{\partial A)=}0,\mu_{\mathcal{H}}(\mathcal{C)=}%
0,\mu_{\mathcal{H}}(\text{ }\widehat{\mathcal{C}}_{k}\mathcal{)=}0
\]

\end{theorem}

\begin{remark}
The dynamical boundary is a subset of the topological boundary of $A,$ viewed
as a subset of $\mathbb{R}^{M}$. In the relative topology of $A$, that is the
topology of $A$ as a metric space in its own right, the boundary is empty and
the dynamical boundary acts as a kind of boundary of the attractor. In
particular, the dynamical boundary, the critical set, and the inner boundary
to any finite depth, are closed sets in the (relative) topology of $A,$ and
their complements, $A\backslash\partial A,$ $A\backslash\mathcal{C}$,
$A\backslash\widehat{\mathcal{C}}_{k}^{(v)},$ are open. Around every
disjunctive point (i.e. image of a disjunctive point in $\Sigma_{\infty}$
under $\pi$) in the attractor there is an open ball that does not meet any of
the sets $\partial A,$ $\mathcal{C}$, $\widehat{\mathcal{C}}_{k}^{(v)}.$ Also,
Baire's theorem tells us that $\widehat{\mathcal{C}}$ does not contain an open
set of $A$. The complements of $\partial A,$ $\mathcal{C}$,
$\widehat{\mathcal{C}}_{k}^{(v)}$and $\widehat{\mathcal{C}}$ provide types of
`interiors' of $A.$ We say that the critical set, the dynamical boundary, and
the inner boundaries are small in a topological sense.
\end{remark}

\begin{proof}
(1) To apply \cite{mauldin} there must be at most one edge of $\mathcal{G}$
directed from vertex $v$ to vertex $w,$ for all $v$ and $w.$ This can always
be contrived, without changing either the dimension or the geometries of the
components of the attractor using the state-spitting technique of \cite{lind},
described here. If $v,w\in\mathcal{V}$ are such that $d=\left\vert
{\textstyle\sum\limits_{\substack{d^{-}=v\\d^{+}=w}}}
1\right\vert >1$, then introduce new vertices $w^{(1)},w^{(2)},...,w^{(d)}$ to
replace $w$, and new components of the attractor $A_{w^{(1)}}=A_{w^{(2)}%
}=...=A_{w^{(d)}}$ all equal to $A_{w}$, and replace the $d$ outgoing edges
from $v$ to $w$ by one outgoing edge to each of the new vertices. All other
edges associated with the vertex $w$, both inward and outward pointing, are
replaced by copies of them at each of the duplicated vertices. Likewise the
maps associated with the new edges are duplicates of the originals. Now
translate the coincident attractors so that they have empty intersections and
modify the maps accordingly, as described following Definition \ref{defONE},
relating them to the original ones by isometric changes of coordinates.
Repeating this process in connection with every ordered pair of vertices
ensures that there is at most one outward pointing edge from vertex $v$ to
vertex $w,$ for all $v$ and $w$ in $\mathcal{G}$. This reduces the present
situation to that in \cite{mauldin}, who makes this assumption. Clearly the
dimension of the attractor is unaltered. We also have $0<\mu_{H}(A)<\infty$ by
\cite[Theorem 3]{mauldin}. Note that \cite[Theorem 3]{mauldin} requires a
different separation condition than the OSC, but both \cite[Theorem
2.1]{boore} and \cite{das} refer to \cite[Theorem 3]{mauldin} as though the
two conditions are equivalent, and we have assumed that this is true. (2) This
is the generalization to the graph-directed case of the definitions and
argument in \cite[Proposition 2.2]{bandt}. We present the proof in parts
(a)\ and (b) for the case $\mathcal{V}=1.$ The proof is carries over to the
tiling IFS case. We focus on showing that $\mathcal{C}\subset\pi
(\Sigma_{\infty}\backslash D).$ The other containments follow similarly. (a)
The OSC\ implies, for similitudes, the open set $\mathcal{O}=\bigcup
\limits_{v\in\mathcal{V}}\mathcal{O}_{v}$ can be chosen so that $\mathcal{O}%
\cap A\neq\emptyset$ \cite{schief}, which implies $A\backslash\partial
A\neq\emptyset$ because in this case $\mathcal{O\cap}\partial A=\emptyset$ by
\cite[Theorem 2.3 via (iii) implies (i) implies (ii)]{moran}$.$ (b)
$A\backslash\partial A\neq\emptyset$ implies $\partial A\cap\pi(D)=\varnothing
$ because if $x=\pi(\sigma)\in\mathcal{C}$ with $\sigma\in D$ then $\partial
A=A$ as in \cite[Proposition 2.2]{bandt} Prop 2.2. It follows that
$\mathcal{C}\subset\pi(\Sigma_{\infty}\backslash D)$. (3) This follows from
(2) and $\partial A\cap\pi(D)=\varnothing$. (4) This is \cite[Proposition
2.1]{bandt} carried over to the tiling IFS case, using the non-overlappingness
of $A$, namely $A\backslash\partial A\neq\emptyset.$ (5) This follows from (2)
and Theorem \ref{ergodictheorem} part (3). (6) Using the thermodynamic
formalism \cite{bedford} and the assumption that $%
{\textstyle\sum\limits_{v}}
\mathcal{W}_{w,v}(t)=1$, we find that $\mu_{H}=\mu_{\widehat{\mathcal{P}}%
}\circ\pi^{-1}$ is, up to a positive multiplicative constant, the Hausdorff
measure obtained when
\[
p_{e}=s^{\mathcal{D}_{H}(A)a_{e}}/\sum\limits_{d^{+}=e^{+}}s^{\mathcal{D}%
_{H}(A)a_{d}}%
\]

\end{proof}

\section{\label{continuitysec}Continuity properties of $\Pi:{\Sigma}^{\dag
}\rightarrow\mathbb{T}$.}

\subsection{\label{metricsec}A convenient compact tiling space}

Let
\[
\mathbb{T=}\left\{  \Pi(\theta):\theta\in\Sigma^{\dag}\right\}
\]

Let $\rho:\mathbb{R}^{M}\rightarrow\mathbb{S}^{M}$ be the usual $M$%
-dimensional stereographic projection to the $M$-sphere, obtained by
positioning $\mathbb{S}^{M}$ tangent to $\mathbb{R}^{M}$ at the origin. Let
$\mathbb{H(S}^{M})$ be the non-empty closed (w.r.t. the usual topology on
$\mathbb{S}^{M}$) subsets of $\mathbb{S}^{M}.$ Let $d_{\mathbb{H(S}^{M})}$ be
the Hausdorff distance with respect to the round metric on $\mathbb{S}^{M}$,
so that $(\mathbb{H(S}^{M}),d_{\mathbb{H(S}^{M})})$ is a compact metric space.
Let $\mathbb{H(H(S}^{M}))$ be the nonempty compact subsets of $(\mathbb{H(S}%
^{M}),d_{\mathbb{H(S}^{M})})$, and let $d_{\mathbb{H(H(S}^{M}))}$ be the
associated Hausdorff metric. Then $(\mathbb{H(H(S}^{M})),d_{\mathbb{H(H(S}%
^{M}))})$ is a compact metric space. Finally, define a metric $d_{\mathbb{T}}$
on $\mathbb{T}$ by%
\[
d_{\mathbb{T}}(T_{1},T_{2})=d_{\mathbb{H(H(S}^{M}))}(\rho\left(  T_{1}\right)
,\rho\left(  T_{2}\right)  )
\]
for all $T_{1},T_{2}\in\mathbb{T}$.

\begin{theorem}
\label{compactthm}$\mathbb{(T},d_{\mathbb{T}})$ is a compact metric space.
\end{theorem}

\begin{proof}
We make these comments. There is an absolute upper bound to the diameter of
all tiles in all tilings. Every ball $B_{R}(O),$ the ball centered at the
origin of radius $R$, meets at least one tile of any tiling $T$. The
projection of the collection of sets obtained by intersecting each tiling in
$\mathbb{T}$ with $B_{R}(O)$ and keeping the subset of each set that meets
$B_{R}(O)$ is a compact metric space with respect to $d_{\mathbb{T}}$. Note
that $d_{\mathbb{T}}(T_{1}\cup B_{R}^{C}(O),T_{2}\cup B_{R}^{C}(O))\rightarrow
0$ as $R\rightarrow\infty,$ where $B_{R}^{C}(O)=\mathbb{R}^{M}\backslash
B_{R}(O),$ for any pair of tilings $T_{1},T_{2}\in\mathbb{T}$. A diagonal
argument may be used to prove the theorem, as follows. Any $T\in\mathbb{T}$
can be expressed as an infinite sequence of tiles, with possible repetitions
of tiles. Let $(T_{k})$ be a sequence of tilings. Let $(T_{k_{1}})$ be a
subsequence of $(T_{k})$ that converges inside (the projection of) $B_{1}(O).$
Recursively, let $(T_{k_{n+1}})$ be a subsequence of $(T_{k_{n}})$ that
converges inside $B_{n+1}(O)$. Then the sequence of tilings $(T_{k_{n},n})$
converges to a tiling, with respect to the metric $d_{\mathbb{T}}$.
\end{proof}

See also for example \cite{anderson, bedfordScene, sadun, solomyak, wicks}
where related metrics and topologies are defined. The Hausdorff-Gromov metric
applied to collections of subsets of the $M$-sphere might also be used to
measure distances between tilings. This does not suit the present setting,
where non-trivial isometries of tilings are distinguished.

\subsection{Continuity}

The following definition generalizes a related concept for the case where $A$
is a topological disk and $\left\vert \mathcal{V}\right\vert =1$, see
\cite{tilings}. For $\theta\in{\Sigma}_{\infty}^{\dag}$ define $I(\theta
)\subset{\Sigma}_{\infty}$ to be the set of limit points of $\{\theta
_{l+m}\theta_{l+m-1}...\theta_{m+1}:l,m\in\mathbb{N}\}$. Define for all
$v\in\mathcal{V}$
\[
H_{v}:=\cup\{f_{-\theta}f_{\sigma}(A_{\sigma^{+}}):\theta^{+}=\sigma
^{-}=v,\theta\in\Sigma_{\ast}^{\dag},\sigma\in\Sigma_{\ast},\theta_{\left\vert
\theta\right\vert }\neq\sigma_{1}\}
\]
$H_{v}$ is the union of all images of $A_{w}$ under the stated neighbor maps,
for all $w\in\mathcal{V}$, namely the maps $f_{-\theta}f_{\sigma}$ in the
definition of $H_{v}$. It is a generalization of the same definition in
the\ case $V=1$, \cite{bandtneighbor1, bandtneighbor2, bandtneighbor3}. Define
the \textbf{central open sets} to be%
\[
O_{v}=\{x\in\mathbb{R}^{M}:d(x,A_{v})<d(x,H_{v})\}
\]
It is the case that $\{O_{v}:v\in\mathcal{V}$\} obeys the open set condition
and \textquotedblleft$(\mathcal{F},\mathcal{G)}$ obeys the OSC" if and only if
\textquotedblleft$A_{v}$ is not contained in $\overline{H}_{v}$ for all
$v\in\mathcal{V}$ "$.$ This follows from the argument in \cite{bandtneighbor2}
generalized in obvious ways, for example to ensure that chains of functions of
the form $f_{-\theta}f_{\sigma}$ are consistent with $\mathcal{G}$.

Call $\theta\in{\Sigma}_{\infty}^{\dag}$ \textbf{reversible} if
\[
{\Sigma}_{rev}^{\dag}:=I(\theta)\cap\{\sigma\in{\Sigma}_{\infty}:\pi
(\sigma)\subset\cup_{v}O_{v}\}\neq\varnothing\text{.}%
\]
Equivalently, $\theta\in{\Sigma}_{rev}^{\dag}$ if the following holds: there
exists $\sigma\in{\Sigma}_{\infty}$ with $\pi(\sigma)\in\cup_{v}O_{v}$ such
that, for all $L,M\in\mathbb{N}$ there is $m\geq M$ so that
\[
\sigma_{1}\sigma_{2}...\sigma_{L}=\theta_{m+L}\theta_{m+L-1}...\theta_{m+1}%
\]
Equivalently, in terms of the notion of \textquotedblleft full" words, see
\cite{tilings}, $\theta\in{\Sigma}_{rev}^{\dag}$ if there is a nonempty
compact set $A^{\prime}\subset\cup_{v}O_{v}$ such that for any positive
integer $M$ there exists $n>m\geq M$ so that
\[
f_{\theta_{n}}f_{\theta_{n-1}}...f_{\theta_{m+1}}(A_{\theta_{m+1}^{+}})\subset
A^{\prime}\text{.}%
\]

\begin{theorem}
\label{ctytheorem}Let $(\mathcal{F},\mathcal{G)}$ be a tiling IFS.
\textit{Then}%
\[
\Pi|_{{\Sigma}_{rev}^{\dag}}:{\Sigma}_{rev}^{\dag}\subset{\Sigma}_{\infty
}^{\dag}\rightarrow\mathbb{T}%
\]
is continuous and
\[
\Pi:{\Sigma}_{\infty}^{\dag}\rightarrow\mathbb{T}%
\]
is upper semi-continuous in this sense: if $\Pi(\theta^{(n)})$ is a sequence
of tilings that converges to a tiling $T\in\mathbb{T}$ as $\theta^{(n)}$
converges to $\theta\in{\Sigma}_{\infty}^{\dag},$ then $\Pi(\theta)\subset T$.
\end{theorem}

\begin{proof}
Proof of upper semi-continuity: let $\{\theta^{(n)}\}$ be a sequence of points
in ${\Sigma}_{\infty}^{\dag}$ that converges to $\theta$ and such that
$\lim\Pi(\theta^{(n)})=T$ with respect to the tiling metric. Let $m$ be given.
Then there is $l_{m}$ so that for all $n\geq l_{m}$ we have $\theta
|m=\theta^{(n)}|m$ and hence $\Pi(\theta|m)=\Pi(\theta^{(n)}|m)\subset
\Pi(\theta^{(n)}).$ Hence we have $\Pi(\theta|m)\subset\underset{n\rightarrow
\infty}{\lim}\Pi(\theta^{(n)})$ and hence, since this is true for all $m,$
$\Pi(\theta)\subset\underset{n\rightarrow\infty}{\lim}\Pi(\theta^{(n)})$.
Proof that $\Pi|_{{\Sigma}_{rev}^{\dag}}:{\Sigma}_{rev}^{\dag}\rightarrow
\mathbb{T}$ is continuous involves blow-ups \cite{strichartz} of central opens
sets. Analogously to the definition of $\Pi,$ define a mapping $\Xi$ from
${\Sigma}^{\dag}$ to subsets of $\mathbb{H(R}^{M})$ as follows. For $\theta
\in\Sigma_{\ast}^{\dag},\theta\neq\varnothing,$%
\[
\Xi(\theta):=\{f_{_{-\theta}}f_{\sigma}(\overline{O_{\sigma^{+}}}):\sigma
\in\Omega_{\xi(\theta)}^{\theta^{+}}\},
\]
and for $\theta\in\Sigma_{\infty}^{\dag}$
\[
\Xi(\theta):=\bigcup\limits_{k\in\mathbb{N}}\Xi(\theta|k).
\]
As is the case for $\Pi,$ increasing families of sets are obtained: each
collection $\Xi(\theta)$ comprises a covering by compact sets of a subset of
$\mathbb{R}^{M}$, the subset being bounded when $\theta\in{\Sigma}_{\ast
}^{\dag}$ and unbounded when $\theta\in{\Sigma}_{\infty}^{\dag}$. For all
$\theta\in{\Sigma}_{\infty}^{\dag}$ the sequence of collections of sets
$\left\{  \Xi(\theta|k)\right\}  _{k=1}^{\infty}$ is nested according to
\[
\Xi(\theta|1)\subset\Xi(\theta|2)\subset\Xi(\theta|3)\subset\cdots\text{ .}%
\]
and we have $\{\Xi(\theta|k)\}$ converges to $\Xi(\theta)$ in the metric
introduced in Section \ref{metricsec}. We refer to $\Xi(\theta)$ as a
\textbf{central open set tiling}. (Examples of such tilings are illustrated in
Figures \ref{fige002} and 5.) In particular, when reversible, the new tiles,
those in $\Xi(\theta|k+1)\backslash$ $\Xi(\theta|k),$ are located further and
further away from the origin as $k$ increases. The result follows.
\end{proof}

\begin{example}
\label{Example01}Let $\mathcal{F}=\{\mathbb{R};f_{1}(x)=x/2,f_{2}%
(x)=(x+1)/2\}$, and consider the sequence of tilings $\{\Pi
(111..(k-times)...12):k\in\mathbb{N\}}$. This sequence converges to a tiling
of $[-1,\infty)$, whilst the sequence of tilings $\{\Pi
(111..(k-times)...11):k\in\mathbb{N\}}$ converges to a tiling of $[0,\infty]$.
\end{example}

\begin{example}
\label{Example02}Example of a central open set tiling. See Figures
\ref{fige002} and \ref{fige003}. In this case the maps of the IFS are, in
complex number representation
\begin{align*}
f_{1}(z)  &  =\frac{z}{2},f_{2}(z)=\frac{1}{2}e^{-\frac{2\pi i}{3}}%
(z-i)-\frac{1}{4}\left(  3-\sqrt{3}i\right)  ,\\
f_{3}(z)  &  =\frac{1}{2}e^{\frac{2\pi i}{3}}(z-i)+\frac{1}{4}\left(
3+\sqrt{3}i\right)  ,f_{4}(z)=\frac{1}{2}z+\frac{i}{2},\\
f_{5}(z)  &  =\frac{1}{2}e^{-\frac{2\pi i}{3}}z+i(1+\frac{\sqrt{3}}{4}%
)-\frac{3}{4},f_{6}(z)=\frac{1}{2}e^{\frac{2\pi i}{3}}z+i(1+\frac{\sqrt{3}}%
{4})+\frac{3}{4}%
\end{align*}
The tilings illustrated in Figure \ref{fige002} are $\Pi(1111...)\ $and
$\Xi(1111...)$.
\end{example}

\begin{figure}[ptb]
\centering
\includegraphics[
height=2.3246in,
width=2.9772in
]{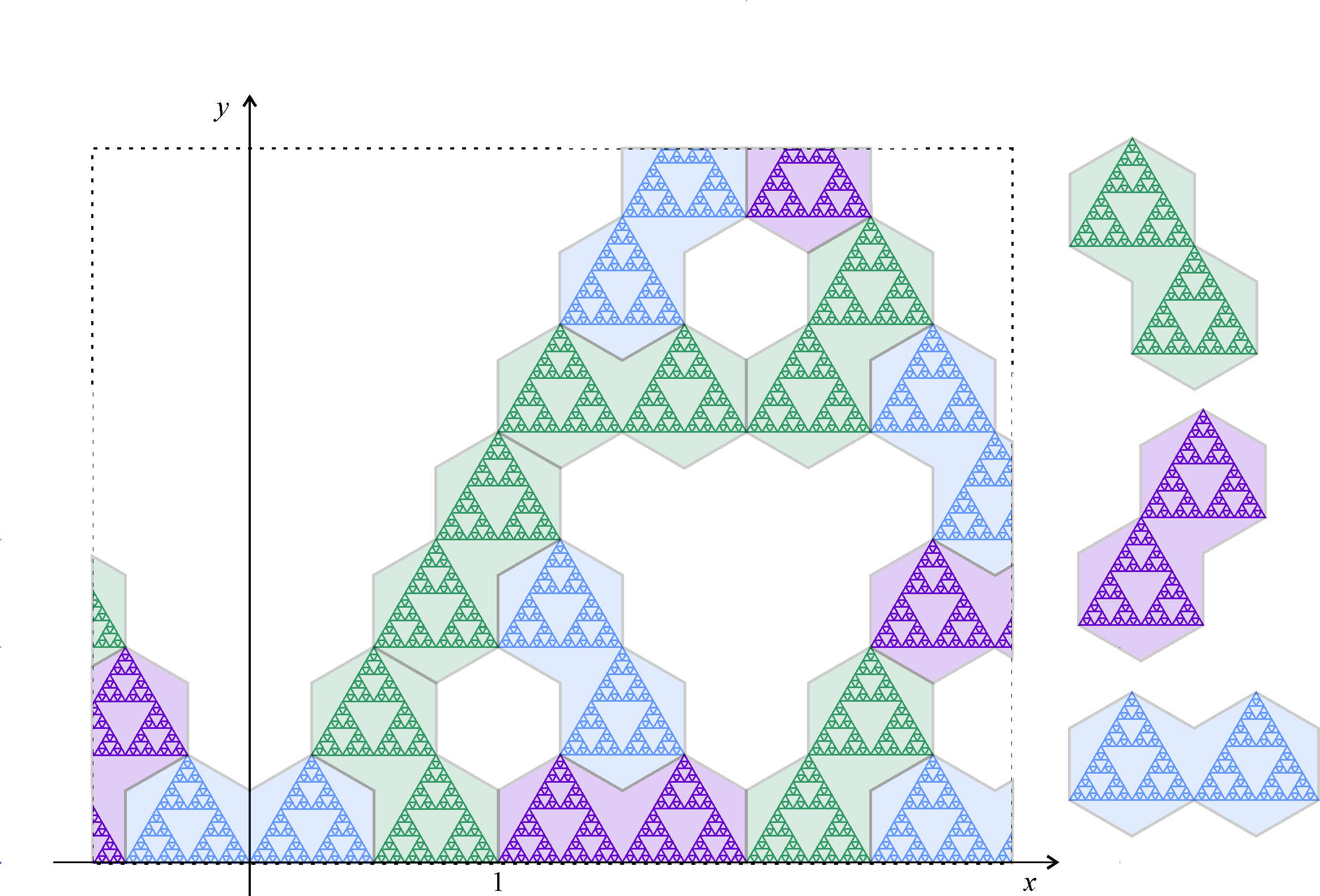}\caption{Example of a central open set tiling. The underlying
fractal tiling is also shown. The three prototiles are shown on the right.}%
\label{fige002}%
\end{figure}

\begin{figure}[ptb]
\centering
\includegraphics[
height=1.0668in,
width=2.5222in
]{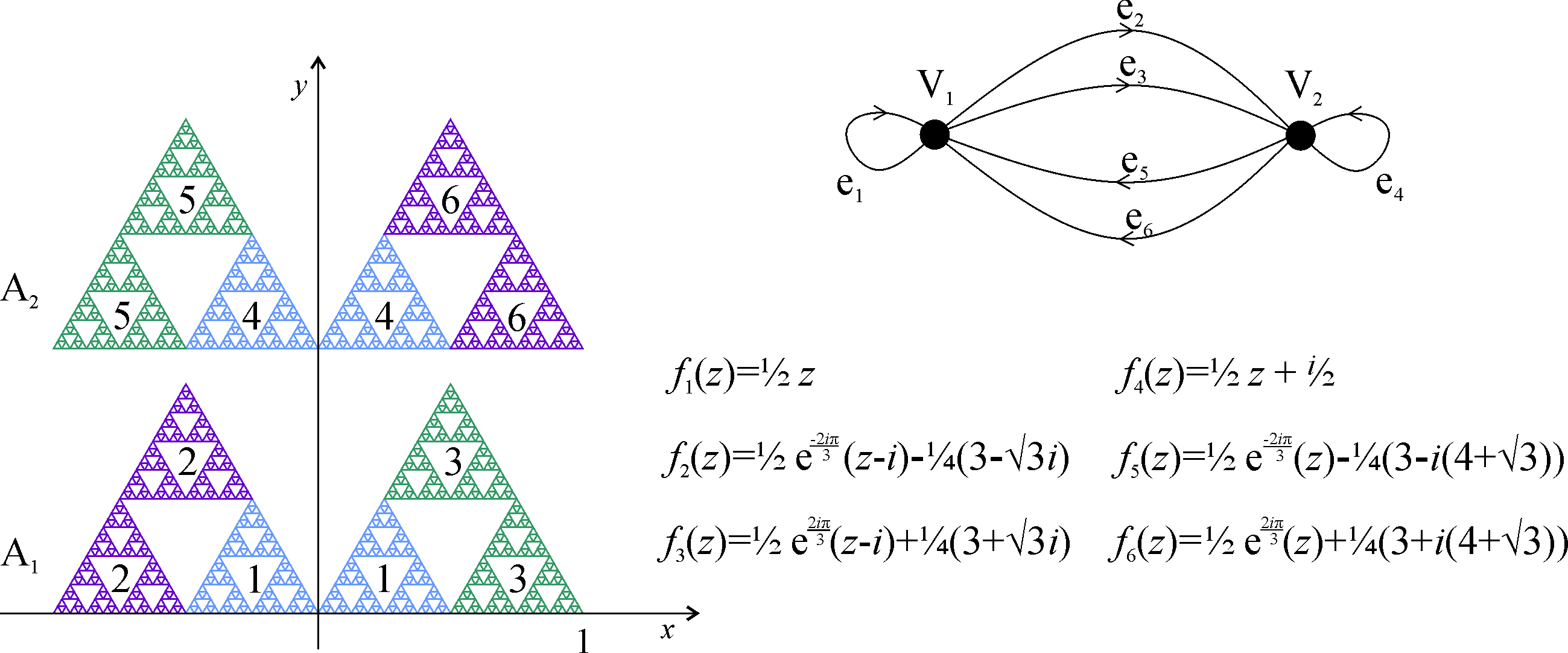}\caption{Illustration of the graph IFS in Figure \ref{fige002}
and Example \ref{Example02}.}%
\label{fige003}%
\end{figure}

\begin{example}
\label{new12}Let $\mathcal{F}$ be the tiling IFS on $\mathbb{R}^{2}$ defined
by $\left\vert \mathcal{V}\right\vert =1$ and the two similitudes
\begin{align*}
f_{1}%
\begin{bmatrix}
x\\
y
\end{bmatrix}
&  =%
\begin{bmatrix}
.6413 & -.3283\\
.3283 & .6413
\end{bmatrix}%
\begin{bmatrix}
x\\
y
\end{bmatrix}
+%
\begin{bmatrix}
.3231\\
-.133
\end{bmatrix}
\\
f_{2}%
\begin{bmatrix}
x\\
y
\end{bmatrix}
&  =%
\begin{bmatrix}
-.2362 & .4620\\
.4620 & .2362
\end{bmatrix}%
\begin{bmatrix}
x\\
y
\end{bmatrix}
+%
\begin{bmatrix}
.8052\\
.5093
\end{bmatrix}
\end{align*}
Part of the associated central open set tiling $\Xi(111...)$ is illustrated in
Figure \ref{new12x}, overlayed on the corresponding tiling $\Xi(111...)$.
Computations are approximate. By inspection, assuming the attractor is
connected and obeys the OSC, this IFS is rigid (see Section 16, Definition 9)
with respect to euclidean transformations.
\end{example}

\begin{figure}[ptb]
\centering
\includegraphics[
height=3.5848in,
width=3.7958in
]{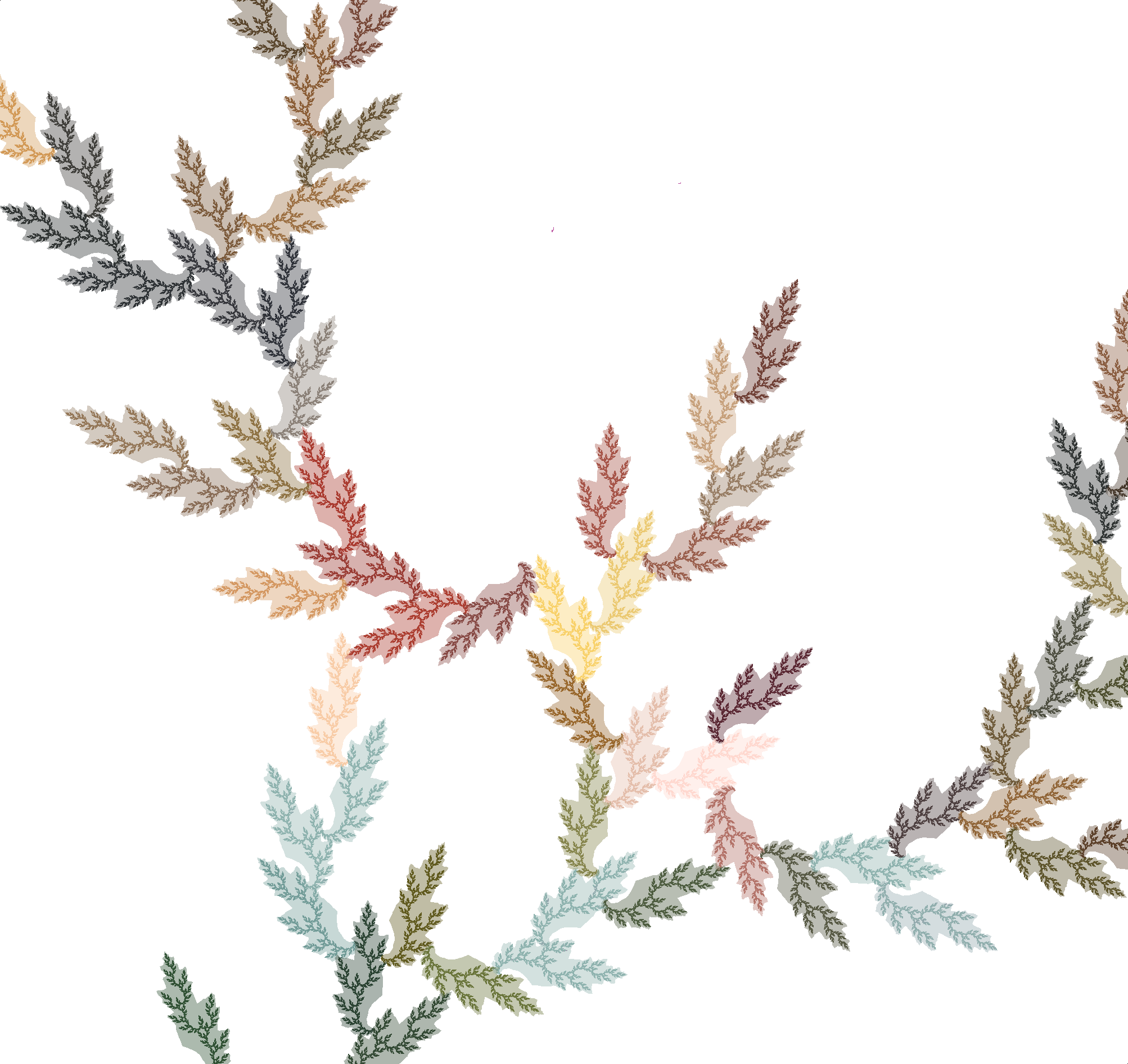}\caption{Part of a central open set tiling. See Example
\ref{new12}. The open set tiles and the underlying fractal tiles are
illustrated, a constant colour for each tile. }%
\label{new12x}%
\end{figure}

\begin{example}
Figure \ref{exacttiling} shows a patch of a central open set tiling associated
with a fractal example in the Introduction. \begin{figure}[ptb]
\centering
\includegraphics[
height=1.7916in,
width=1.8771in
]{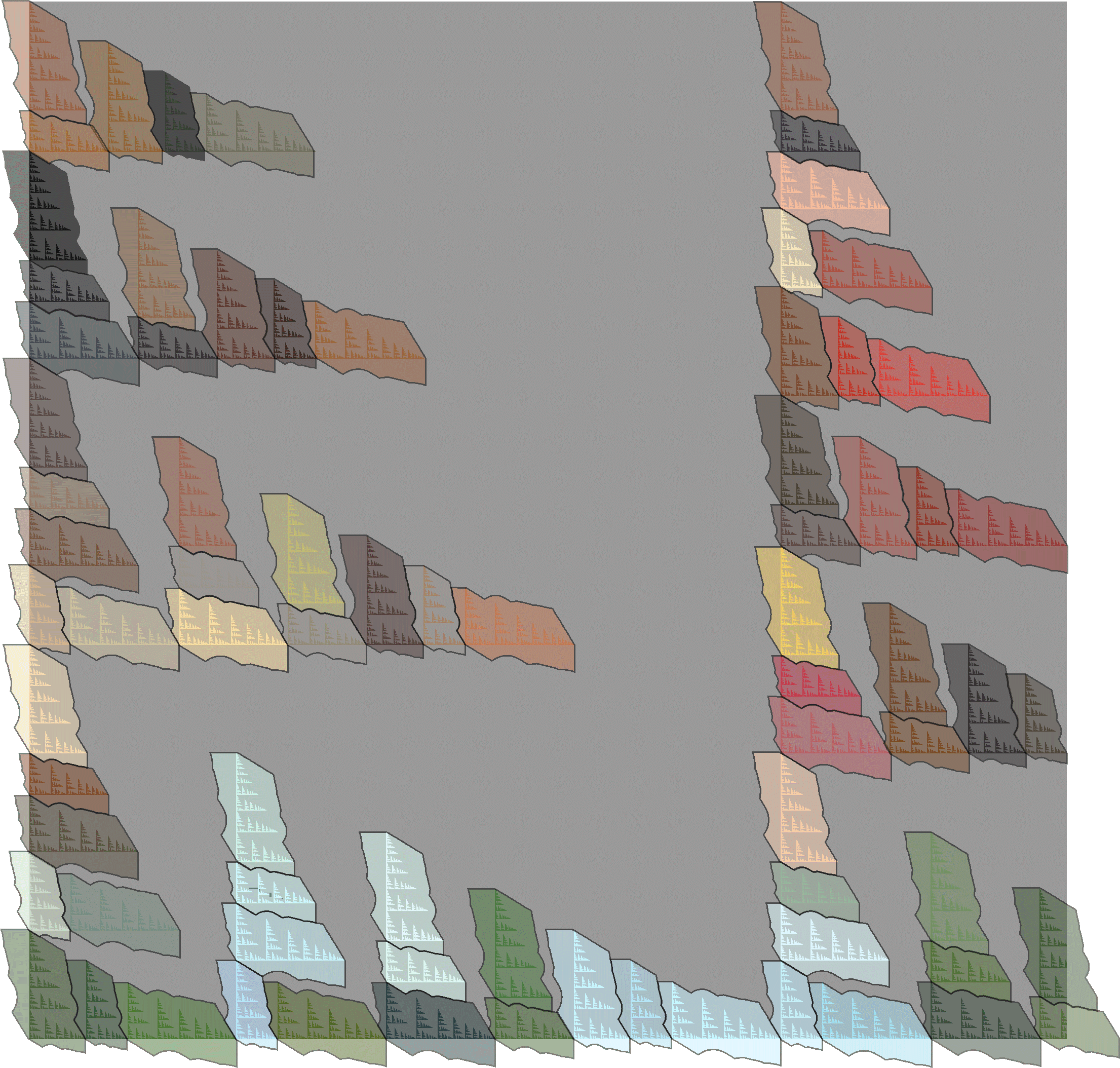}\caption{A patch of a central open set tiling associated
with the fractal example in Figure \ref{fig1-intro}. All tiles that intersect
the grey square are shown. Note the limited ways in which the tiles, of two
sizes, can meet. Each tile has its own colour, but some tiles have the same
colour.}%
\label{exacttiling}%
\end{figure}
\end{example}

\section{\label{symbosection}Symbolic structure : canonical symbolic tilings
and symbolic inflation and deflation}

Write $\Omega_{k}^{(v)}$ to mean any of $\Omega_{k}^{v}$ or $\Omega_{k}.$ The
following lemma tells us that $\Omega_{k+1}^{(v)}$ can be obtained from
$\Omega_{k}^{(v)}$ by adding symbols to the right-hand end of some strings in
$\Omega_{k}^{(v)}$ and leaving the other strings unaltered.

\begin{lemma}
\label{lemma:split}(\textbf{Symbolic Splitting}) For all $k\in\mathbb{N}$ and
$v\in\mathcal{V}$ the following relations hold:%
\[
\Omega_{k+1}^{(v)}=\left\{  \sigma\in\Omega_{k}^{(v)}:k+1<\xi\left(
\sigma\right)  \right\}  \cup\left\{  \sigma j\in\Sigma_{\ast}^{\left(
v\right)  }:\sigma\in\Omega_{k}^{(v)},k+1=\xi\left(  \sigma\right)  \right\}
\text{.}%
\]

\begin{proof}
The assertion follows at once from definition of $\Omega_{k}^{(v)}$.
\end{proof}
\end{lemma}

Define $\alpha_{s}^{-1}:\Omega_{k}^{(v)}\rightarrow2^{\Omega_{k+1}^{(v)}}$ by%
\[
\alpha_{s}^{-1}\sigma=\left\{
\begin{array}
[c]{c}%
\sigma\text{ if }k+1<\xi(\sigma)\\
\{\sigma e:\sigma_{\left\vert \sigma\right\vert }^{+}=e^{-},e\in
\mathcal{E\}}\text{ if }k+1=\xi(\sigma)
\end{array}
\right.
\]
Then
\[
\left\{  \sigma\in\alpha_{s}^{-1}(\omega):\omega\in\Omega_{k}^{v}\right\}
=\Omega_{k+1}^{v}%
\]
This defines symbolic inflation or \textquotedblleft splitting and expansion"
of $\Omega_{k}^{(v)}$, some words in $\Omega_{k+1}^{(v)}$ being the same as in
$\Omega_{k}^{(v)}$ while all the other words in $\Omega_{k}^{(v)}$, namely
those $\sigma$ for which $k+1=\xi(\sigma)$, are split. The inverse operation
is symbolic deflation or \textquotedblleft amalgamation and shrinking",
described by the function%
\[
\alpha_{s}:\Omega_{k+1}^{(v)}\rightarrow\Omega_{k}^{(v)}\text{, }\alpha
_{s}(\Omega_{k+1}^{(v)})=\Omega_{k}^{(v)}%
\]
where $\alpha_{s}(\sigma)$ is the unique $\omega\in\Omega_{k}^{(v)}$ such that
$\sigma=\omega\beta$ for some $\beta\in\Sigma_{\ast}$. Note that $\beta$ may
be the empty string.

We can use $\Omega_{k}^{(v)}$ to define a partition of $\Omega_{m}^{(v)}$ for
$m\geq k$. The partition of $\Omega_{k+j}^{(v)}$ is $\Omega_{k+j}^{(v)}/\sim$
where $x\sim y$ if $\alpha_{s}^{j}(x)=\alpha_{s}^{j}(y)$.

\begin{lemma}
\label{cor:partition}(\textbf{Symbolic Partitions)} For all $m\geq k\geq0,$
the set $\Omega_{k}^{(v)}$ defines a partition $P_{m,k}^{(v)}$ of $\Omega
_{m}^{(v)}$ according to $p\in P_{m,k}^{(v)}$ if and only if there is
$\omega\in\Sigma_{\ast}$ such that
\[
p=\{\omega\beta\in\Omega_{m}^{(v)}:\beta\in\Omega_{k}^{(v)}\}.
\]

\end{lemma}

\begin{proof}
This follows from Lemma \ref{lemma:split}: for any $\theta\in\Omega_{m}^{(v)}$
there is a unique $\omega\in\Omega_{k}^{(v)}$ such that $\theta=\omega\beta$
for some $\beta\in\Sigma_{\ast}$. Each word in $\Omega_{m}^{(v)}$ is
associated with a unique word in $\Omega_{k}^{(v)}$. Each word in $\Omega
_{k}^{(v)}$ is associated with a set of words in $\Omega_{m}^{(v)}$.
\end{proof}

According to Lemma \ref{lemma:split}, $\Omega_{k+1}^{(v)}$ may be calculated
by tacking words (some of which may be empty) onto the right-hand end of the
words in $\Omega_{k}^{(v)}$. We can invert this description by expressing
$\Omega_{k}^{(v)}$ as a union of predecessors ($\Omega_{j}^{(v)}$s with $j<k$)
of $\Omega_{k}^{(v)}$ with words tacked onto their other ends, that is, their
left-hand ends.

\begin{theorem}
\label{lem:struct}(\textbf{Symbolic Predecessors)} For all $k\geq a_{\max}+l$,
for all $v\in\mathcal{V}$, for all $l\in\mathbb{N}_{0},$%
\[
\Omega_{k}^{\left(  v\right)  }=\bigsqcup\limits_{\omega\in\Omega_{l}^{\left(
v\right)  }}\omega\Omega_{k-\xi\left(  \omega\right)  }^{\omega^{+}}%
\]

\end{theorem}

\begin{proof}
It is clear that the union is indeed a disjoint union. It is easy to check
that the r.h.s. is contained in the l.h.s. Conversely, if $\sigma\in\Omega
_{k}^{\left(  v\right)  }$ then there is unique $\omega\in\Omega_{l}^{\left(
v\right)  }$ such that $\sigma=\omega\beta$ for some $\beta\in\Sigma_{\ast}$
by Corollary \ref{cor:partition}. Because $\omega\beta\in\Sigma_{\ast}$ it
follows that $\beta_{1}$ is an edge that starts where the last edge in
$\omega$ is directed, namely the vertex $\omega^{+}$. Finally, since
$\xi\left(  \omega\beta\right)  =\xi\left(  \omega)+\xi(\beta\right)  $ it
follows that $\beta\in\Omega_{k-\xi\left(  \omega\right)  }^{\omega^{+}}$.
\end{proof}

\section{\label{canonical}Canonical tilings and their relationship to
$\Pi(\theta)$}

\begin{definition}
\label{canondef}We define the \textbf{canonical tilings }of the tiling IFS
$(\mathcal{F},\mathcal{G)}$ to be
\[
T_{k}:=s^{-k}\pi(\Omega_{k}),\text{ }T_{k}^{v}:=s^{-k}\pi(\Omega_{k}^{v})
\]
$k\in\mathbb{N},v\in\mathcal{V},$ also%
\begin{align*}
T_{0}  &  :=\Pi(0):=\cup_{v\in\mathcal{V}}T_{0}^{v},T_{0}^{v}:=\Pi
(e|0):=\{f_{e}(A^{e^{+}}):e^{-}=v\},\\
T_{-1}^{v}  &  :=sA_{v}\text{, }T_{-1}:=\cup_{v\in\mathcal{V}}sA_{v}%
\end{align*}

\end{definition}

A canonical tiling may be written as a disjoint union of images under
isometries applied to other canonical tilings as described in Lemma
\ref{lem:canonical}. More generally we may say, concerning any tiling $T$
which is a union of images under isometries applied to canonical tilings, that
\textquotedblleft$T\ $can be written as an \textbf{isometric combination} of
canonical tilings\textquotedblright.

\begin{lemma}
\label{lem:canonical}For all $k\geq a_{\max}+l$, for all $l\in\mathbb{N}_{0},
$ for all $v\in\mathcal{V}$
\[
T_{k}^{v}=\bigsqcup\limits_{\omega\in\Omega_{l}^{v}}E_{k,\omega}%
T_{k-\xi\left(  \omega\right)  }^{\omega^{+}}\text{ and }T_{k}=\bigsqcup
\limits_{\omega\in\Omega_{l}}E_{k,\omega}T_{k-\xi\left(  \omega\right)
}^{\omega^{+}}%
\]
where $E_{k,\omega}=s^{-k}f_{\omega}s^{k-\xi(\omega)}\in\mathcal{U}$ is an isometry.
\end{lemma}

\begin{proof}
Direct calculation using Theorem \ref{lem:struct}.
\end{proof}

\begin{theorem}
\label{piusingts}For all $\theta\in\Sigma_{\ast}^{\dag},$%
\[
\Pi(\theta)=E_{\theta}T_{\xi(\theta)}^{\theta^{+}},
\]
where $E_{\theta}=f_{-\theta}s^{\xi(\theta)}\in\mathcal{U}$. Also if
$l\in\mathbb{N}_{0},$ and $\xi(\theta)\geq a_{\max}+l$, then%
\[
\Pi(\theta)=\bigsqcup\limits_{\omega\in\Omega_{l}^{\theta^{+}}}E_{\theta
,\omega}T_{\xi(\theta)-\xi\left(  \omega\right)  }^{\omega^{+}}%
\]
where $E_{\theta,\omega}=f_{_{-\theta}}f_{\omega}s^{\xi(\theta)-\xi(\omega
)}\in\mathcal{U}$ is an isometry.
\end{theorem}

\begin{proof}
Writing $\theta=\theta_{1}\theta_{2}...\theta_{k}$ so that $\left\vert
\theta\right\vert =k,$ we have from the definitions
\begin{align*}
\Pi(\theta_{1}\theta_{2}...\theta_{k})  &  =f_{_{-\theta_{1}\theta
_{2}...\theta_{k}}}\{\pi\left(  \sigma\right)  :\sigma\in\Omega_{\xi
(\theta_{1}\theta_{2}...\theta_{k})}^{\theta_{k}^{+}}\}\\
&  =f_{_{-\theta_{1}\theta_{2}...\theta_{k}}}s^{\xi(\theta_{1}\theta
_{2}...\theta_{k})}s^{-\xi(\theta_{1}\theta_{2}...\theta_{k})}\{\pi
(\sigma):\sigma\in\Omega_{\xi(\theta_{1}\theta_{2}...\theta_{k})}^{\theta
_{k}^{+}}\}\\
&  =E_{\theta}T_{\xi(\theta)}^{\theta_{\left\vert \theta\right\vert }^{+}}%
\end{align*}
where $E_{\theta}=f_{-\theta}s^{\xi(\theta)}$. The last statement of the
theorem follows similarly from Lemma \ref{lem:canonical}.
\end{proof}

\section{\label{quasisec}Tilings in $\mathbb{T}^{\infty}$ that are
quasiperiodic}

We recall from \cite{barnsleyvince} the following definitions. A subset $P$ of
a tiling $T$ is called a \textit{patch} of $T$ if it is contained in a ball of
finite radius. A tiling $T$ is \textit{quasiperiodic} if, for any patch $P$,
there is a number $R>0$ such that any ball \textit{centered at a point in the
support of }$T,$ of radius $R,$ contains an isometric copy of $P$. Two tilings
are \textit{locally isomorphic} if any patch in either tiling also appears in
the other tiling. A tiling $T$ is\textit{\ self-similar} if there is a
similitude $\psi$ such that $\psi(t)$ is a union of tiles in $T$ for all $t\in
T$. In this case $\psi$ is called a \textit{self-similarity }for\textit{\ }%
$T$. These definitions are consistent with \cite{radin, solomyak2} when
applied to \textquotedblleft classical\textquotedblright\ self-similar tilings
supported on $\mathbb{R}^{M}$.

We say that the tiling IFS $(\mathcal{F},\mathcal{G)}$ is \textit{coprime} if
there is a pair $v,w\in\mathcal{V}$ and there are $\sigma,\omega\in
\Sigma_{\ast}$ with $\sigma^{+}=\omega^{+}=v$ and $\sigma^{-}=\omega^{-}=w$
such that the greatest common factor of $\xi\left(  \sigma\right)  $ and
$\xi\left(  \omega\right)  $ is $1$.

\begin{theorem}
\label{theoremTHREE} Let $(\mathcal{F},\mathcal{G)}$ be a tiling IFS.

\begin{enumerate}
\item If $(\mathcal{F},\mathcal{G)}$ is coprime, then each tiling in
$\mathbb{T}_{\infty}:=\{\Pi(\theta):\theta\in\Sigma_{\infty}^{\dag}\}$ is quasiperiodic.

\item If $(\mathcal{F},\mathcal{G)}$ is coprime, then each pair of tilings in
$\mathbb{T}_{\infty}$ are locally isomorphic.

\item If $\theta\in{\Sigma}_{\infty}^{\dag}$ is eventually periodic, then
$\Pi(\theta)$ is self-similar: if $\theta=\alpha\overline{\beta}$ for some
$\alpha,\beta\in{\Sigma}_{\ast}^{\dag},$ then $f_{-\alpha}f_{-\beta}\left(
f_{-\alpha}\right)  ^{-1}$ is a self-similarity for $\Pi(\theta)$.
\end{enumerate}
\end{theorem}

\begin{proof}
This uses Theorem \ref{piusingts}, and follows similar lines to \cite[proof of
Theorem 2]{barnsleyvince}. (1) Let $\theta\in\left[  N\right]  ^{\infty}$ be
given and let $P$ be a patch in $\Pi(\theta)$. There is a $K_{1}\in\mathbb{N}$
such that $P$ is contained in $\Pi(\theta|K_{1})$. Hence an isometric copy of
$P$ is contained in $T_{K_{2}}^{\left(  \theta|K_{1}\right)  ^{+}}$ where
$K_{2}=\xi(\theta|K_{1})$. Now choose $K_{3}\in\mathbb{N}$ so that an
isometric copy of $T_{K_{2}}^{\left(  \theta|K_{1}\right)  ^{+}}$ is contained
in each $T_{k}^{v}$ with $k\geq K_{3}.$ That this is possible follows from the
recursion in Lemma \ref{lem:canonical} and the assumption that $(\mathcal{F}%
,\mathcal{G)}$ is coprime. In particular, $T_{K_{2}}\subset T_{K_{3}+i}$ for
all $i\in\{1,2,...,a_{\max}\}$. Now let $K_{4}=K_{3}+a_{\max}$. Then, for all
$k\geq K_{4}$ and all $v\in\mathcal{V}$, the tiling $T_{k}^{v}$ is an
isometric combination of $\{T_{K_{3}+i}^{w}:$\textit{\ }$i=1,2,...,a_{\max}$,
$w\in\mathcal{V}\}$, and each of these tilings contains a copy of $T_{K_{2}%
}^{\left(  \theta|K_{1}\right)  ^{+}}$ and, in particular, a copy of $P $. Let
$D=\max\{\left\Vert x-y\right\Vert :x,y\in A\}$ be the diameter of $A$. The
support of $T_{k}$ is $s^{-k}A$ which has diameter $s^{-k}D.$ Hence
$\cup\{t\in T_{k}\}\subset$ $B_{x}(2s^{-k}D)$, the ball centered at $x$ of
radius $2s^{-k}D$, for all $x\in$ $\cup\{t\in T_{k}\}$. It follows that if
$x\in\cup\{t\in\Pi(\theta^{\prime})\}$ for any $\theta^{\prime}\in\left[
N\right]  ^{\infty}$, then $B(x,2s^{-K_{4}}D)$ contains a copy of $T_{K_{2}}$
and hence a copy of $P$. Therefore all unbounded tilings in $\mathbb{T}$ are
quasiperiodic. (2) This is essentially the same as (1). (3) Let $\theta
=\alpha\overline{\beta}=\alpha_{1}\alpha_{2}\cdots\alpha_{l}\beta_{1}\beta
_{2}\cdots\beta_{m}\beta_{1}\beta_{2}\cdots\beta_{m}\cdots\in{\Sigma}_{\infty
}^{\dag}$, and
\[
E_{\theta|k}:=f_{-(\theta|k)}s^{\xi(\theta|k)},T(\theta|k):=T_{\xi(\theta
|k)}^{(\theta|k)^{+}}%
\]
We have the increasing union
\[
\Pi(\theta)=%
{\textstyle\bigcup\limits_{j\in\mathbb{N}}}
E_{\theta|(l+jm+m)}T(\theta|(l+jm+m))
\]
We can write
\[
\Pi(\theta)=%
{\textstyle\bigcup\limits_{j\in\mathbb{N}}}
E_{\theta|(l+jm)}T(\theta|(l+jm))=f_{-\alpha}%
{\textstyle\bigcup\limits_{j\in\mathbb{N}}}
f_{-\beta}^{j}s^{\xi(\theta|(l+jm))}T(\theta|(l+jm)),
\]
and also
\[
\Pi(\theta)=%
{\textstyle\bigcup\limits_{j\in\mathbb{N}}}
E_{\theta|(l+jm+m)}T(\theta|(l+jm+m))=f_{-\alpha}f_{-\beta}%
{\textstyle\bigcup\limits_{j\in\mathbb{N}}}
f_{-\beta}^{j}s^{\xi(\theta|(l+jm+m))}T(\theta|(l+jm+m))\text{.}%
\]
Here $f_{-\beta}^{j}s^{e(\theta|(l+jm+m))}T(\theta|(l+jm+m))$ is a refinement
of $f_{-\beta}^{j}s^{e(\theta|(l+jm))}T(\theta|(l+jm))$. It follows that
$\left(  f_{-\alpha}f_{-\beta}\right)  ^{-1}\Pi(\theta)$ is a refinement of
$\left(  f_{-\alpha}\right)  ^{-1}\Pi(\theta)$, from which it follows that
$\left(  f_{-\alpha}\right)  \left(  f_{-\alpha}f_{-\beta}\right)  ^{-1}%
\Pi(\theta)$ is a refinement of $\Pi(\theta)$. Therefore, every set in
$\left(  f_{-\alpha}f_{-\beta}\right)  \left(  f_{-\alpha}\right)  ^{-1}%
\Pi(\theta)$ is a union of tiles in $\Pi(\theta)$.
\end{proof}

\section{\label{relsec}Addresses}

Addresses, both relative and absolute, are described in \cite{barnsleyvince}
for the case $\left\vert \mathcal{V}\right\vert =1$. See also \cite{bandt2}.
Here we add information and generalize. The relationship between these two
types of addresses is subtle.

Write $T_{k}^{(v)}$ to mean any of $T_{k}^{v}$ or $T_{k}.$

\begin{definition}
\label{relativedef}The \textbf{relative address} of $t\in T_{k}^{(v)}$ is
defined to be $\varnothing.\pi^{-1}s^{k}(t)\in\varnothing.\Omega_{k}^{(v)}$.
The relative address\textit{\ }of a tile $t\in T_{k}$ depends on its context,
its location relative to $T_{k}$, and depends in particular on $k\in
\mathbb{N}_{0}$. Relative addresses also apply to the tiles of $\Pi(\theta)$
for each $\theta\in\Sigma_{\ast}^{\dag}$ because $\Pi(\theta)=E_{\theta}%
T_{\xi(\theta)}^{\theta_{\left\vert \theta\right\vert }^{+}}$ where
$E_{\theta}=f_{-\theta}s^{\xi(\theta)}$ (by Theorem \ref{piusingts}) is a
known isometry applied to $T_{\xi(\theta)}$. Thus, the relative address of
$t\in\Pi(\theta)$ relative to $\Pi(\theta)$ is $\varnothing.\pi^{-1}%
f_{-\theta}^{-1}(t)$, for $\theta\in\Sigma_{\ast}^{\dag}$. When it is clear
from context we may drop the symbols \textquotedblleft$\varnothing.".$
\end{definition}

\begin{lemma}
\label{bijlemma}The tiles of $T_{k}$ are in bijective correspondence with the
set of relative addresses $\varnothing.\Omega_{k}$. The tiles of $T_{k}^{v}$
are in bijective correspondence with the set of relative addresses
$\varnothing.\Omega_{k}^{v}$.
\end{lemma}

\begin{proof}
The correspondences are provided by the bijective map%
\[
H:\varnothing.\Omega_{k}\rightarrow T_{k}%
\]
defined by $H(\varnothing.\sigma)=s^{-k}\pi(\sigma)$. We have $T_{k}=s^{-k}%
\pi(\Omega_{k})$ so $H$ maps $\varnothing.\Omega_{k}$ onto $T_{k} $. Also $H$
is one-to-one: if $\beta\neq\gamma,$ for $\beta,\gamma\in\Sigma_{\ast}$ then
$f_{\beta}(A)\neq f_{\gamma}(A)$ because $H(\varnothing.\beta)=H(\varnothing
.\gamma)$ implies $\pi(\beta)=\pi(\gamma)$ which implies $\beta=\gamma$
because the tiling IFS\ obeys the open set condition and $A_{v}\cap
A_{w}=\emptyset$ for $v\neq w$. If the requirement $A_{v}\cap A_{w}=\emptyset$
does not hold, it may not be true that $H:\varnothing.\Omega_{k}\rightarrow
T_{k}$ is one-to-one; but it remains true that $H|_{\varnothing.\Omega_{k}%
^{v}}:\varnothing.\Omega_{k}^{v}\rightarrow T_{k}^{v}$ is bijective.
\end{proof}

For precision we should write \textquotedblleft the relative address of $t$
relative to $T_{k}$": however, when the context $t\in T_{k}$ is clear, we may
simply refer to \textquotedblleft the relative address of $t$". For example,
if $t\in ET_{k}$ where $E$ is an isometry that is either known or can be
inferred from the context, then we may say that $t$ has a unique relative address.

\begin{example}
\label{ex01}(Standard 1D binary tiling) For the IFS $\mathcal{F}%
_{0}=\{\mathbb{R};f_{1},f_{2}\}$ with $f_{1}(x)=0.5x,f_{2}(x)=$ $0.5x+0.5$ we
have $\Pi(\theta)$ for $\theta\in\Sigma_{\ast}^{\dag}$ is a tiling by copies
of the tile $t=[0,0.5]$ whose union is an interval of length $2^{\left\vert
\theta\right\vert }$ and is isometric to $T_{\left\vert \theta\right\vert }$
and represented by $tttt....t$ with relative addresses in order from left to
right
\[
\varnothing.111...11,\varnothing.111...12,\varnothing
.111...21,....,\varnothing.222...22,
\]
the length of each string (address) being $\left\vert \theta\right\vert +1.$
Notice that here $T_{k}$ contains $2^{\left\vert \theta\right\vert }-1$ copies
of $T_{0}$ (namely tt) where a copy is $ET_{0}$ where $E\in\mathcal{T}%
_{\mathcal{F}_{0}}$, the group of isometries generated by the functions of
$\mathcal{F}_{0}$.
\end{example}

\begin{example}
\label{ex02}(Fibonacci 1D tilings) $\mathcal{F}_{1}\mathcal{=\{}%
ax,a^{2}x+1-a^{2}\}$ where $a+a^{2}=1,$ $a>0$. The tiles of $\Pi(\theta)$ for
$\theta\in\Sigma_{\ast}^{\dag}$ are images under isometries (that belong to
the group of isometries generated by the IFS) applied to the tiles $[0,a]$ and
$[a,1]$ of the attractor $A=[0,1]$. Writing the tiling $T_{0}$ as $ls$ where
$l$ is a copy of $[0,a]$ and (here) $s$ is a copy of $[0,a^{2}]$ we have:

$T_{0}=ls$ has relative addresses $\varnothing.1,\varnothing.2$ (i.e. the
address of $l$ is $1$ and of $s$ is $2$)

$T_{1}=lsl$ has relative addresses $\varnothing.11,\varnothing.12,\varnothing
.2$

$T_{2}=lslls$ has relative addresses $\varnothing.111,\varnothing
.112,\varnothing.12,\varnothing.21,\varnothing.22$

$T_{3}=lsllslsl$ has relative addresses $\varnothing.1111,\varnothing
.1112,\varnothing.112,\varnothing.121,...$

We remark that $T_{k}$ comprises $F_{k+1}$ distinct tiles and contains exactly
$F_{k}$ copies of $T_{0}$, where $\{F_{k}:k\in\mathbb{N}_{0}\}$ is a sequence
of Fibonacci numbers $\{1,2,3,5,8,13,21,...\}$. Also $T_{4}=lsllslsllslls$
contains two overlapping copies of $T_{2}$.
\end{example}

The following theorem defines hierarchies of canonical tilings. It points out
that each relative address is associated with a specific hierarchy.

\begin{theorem}
\label{hierarchythm}Let $\left(  \mathcal{F},\mathcal{G}\right)  $ be a tiling
IFS. The following \textbf{hierarchy of} canonical tilings is associated with
any given relative address $\sigma\in\Sigma_{\ast}$:
\begin{equation}
F_{0}T_{0}^{\sigma_{\left\vert \sigma\right\vert }|0}\subset F_{1}%
T_{\xi\left(  \sigma_{\left\vert \sigma\right\vert }\right)  }^{\sigma
_{\left\vert \sigma\right\vert }^{+}}\subset F_{2}T_{\xi\left(  \sigma
_{\left\vert \sigma\right\vert }\sigma_{\left\vert \sigma\right\vert
-1}\right)  }^{\sigma_{\left\vert \sigma\right\vert -1}^{+}}\subset
...F_{\left\vert \sigma\right\vert -1}T_{\xi\left(  \sigma_{\left\vert
\sigma\right\vert }\sigma_{\left\vert \sigma\right\vert -1}...\sigma
_{2}\right)  }^{\sigma_{2}^{+}}\subset T_{\xi\left(  \sigma_{\left\vert
\sigma\right\vert }\sigma_{\left\vert \sigma\right\vert -1}...\sigma
_{1}\right)  }^{\sigma_{1}^{+}} \label{hierarchy}%
\end{equation}
where $F_{k}$ is the isometry $s^{-\xi\left(  \sigma\right)  }(f_{-\sigma
_{\left\vert \sigma\right\vert -k}\sigma_{\left\vert \sigma\right\vert
-k-1}...\sigma_{1}}s^{\xi(\sigma_{1}...\sigma_{\left\vert \sigma\right\vert
-k})})^{-1}s^{\xi\left(  \sigma\right)  }$ for $k=0,1,...,\xi\left(
\sigma\right)  $.
\end{theorem}

\begin{proof}
The chain of inclusions
\[
\Pi(\sigma_{\left\vert \sigma\right\vert }|0)\subset\Pi(\sigma_{\left\vert
\sigma\right\vert })\subset\Pi(\sigma_{\left\vert \sigma\right\vert }%
\sigma_{\left\vert \sigma\right\vert -1})\subset...\subset\Pi(\sigma
_{\left\vert \sigma\right\vert }\sigma_{\left\vert \sigma\right\vert
-1}...\sigma_{1})
\]
can be rewritten%
\begin{align*}
T_{0}^{\sigma_{\left\vert \sigma\right\vert }|0}  &  \subset f_{-\sigma
_{\left\vert \sigma\right\vert }}s^{\xi\left(  \sigma_{\left\vert
\sigma\right\vert }\right)  }T_{\xi\left(  \sigma_{\left\vert \sigma
\right\vert }\right)  }^{\sigma_{\left\vert \sigma\right\vert }^{+}}\subset
f_{-\sigma_{\left\vert \sigma\right\vert }\sigma_{\left\vert \sigma\right\vert
-1}}s^{\xi\left(  \sigma_{\left\vert \sigma\right\vert }\sigma_{\left\vert
\sigma\right\vert -1}\right)  }T_{\xi\left(  \sigma_{\left\vert \sigma
\right\vert }\sigma_{\left\vert \sigma\right\vert -1}\right)  }^{\sigma
_{\left\vert \sigma\right\vert -1}^{+}}\subset...\\
&  \subset f_{-\sigma_{\left\vert \sigma\right\vert }\sigma_{\left\vert
\sigma\right\vert -1}...\sigma_{1}}s^{\xi\left(  \sigma_{\left\vert
\sigma\right\vert }\sigma_{\left\vert \sigma\right\vert -1}...\sigma
_{1}\right)  }T_{\xi\left(  \sigma_{\left\vert \sigma\right\vert }%
\sigma_{\left\vert \sigma\right\vert -1}...\sigma_{1}\right)  }^{\sigma
_{1}^{+}}%
\end{align*}
Apply the isometry $E=s^{-\xi\left(  \sigma\right)  }f_{\sigma}$ on the left
throughout to obtain%
\begin{align*}
s^{-\xi\left(  \sigma\right)  }f_{\sigma_{1}\sigma_{2}...\sigma_{\left\vert
\sigma\right\vert }}T_{0}^{\sigma_{\left\vert \sigma\right\vert }|0}  &
\subset s^{-\xi\left(  \sigma\right)  }f_{\sigma_{1}\sigma_{2}...\sigma
_{\left\vert \sigma\right\vert -1}}s^{\xi\left(  \sigma_{\left\vert
\sigma\right\vert }\right)  }T_{\xi\left(  \sigma_{\left\vert \sigma
\right\vert }\right)  }^{\sigma_{\left\vert \sigma\right\vert }^{+}}\\
&  \subset s^{-\xi\left(  \sigma\right)  }f_{\sigma_{1}\sigma_{2}%
...\sigma_{\left\vert \sigma\right\vert -2}}s^{\xi\left(  \sigma_{\left\vert
\sigma\right\vert }\sigma_{\left\vert \sigma\right\vert -1}\right)  }%
T_{\xi\left(  \sigma_{\left\vert \sigma\right\vert }\sigma_{\left\vert
\sigma\right\vert -1}\right)  }^{\sigma_{\left\vert \sigma\right\vert -1}^{+}%
}\\
&  \subset...\\
&  \subset s^{-\xi\left(  \sigma\right)  }f_{\sigma_{1}}s^{\xi\left(
\sigma_{\left\vert \sigma\right\vert }\sigma_{\left\vert \sigma\right\vert
-1}...\sigma_{2}\right)  }T_{\xi\left(  \sigma_{\left\vert \sigma\right\vert
}\sigma_{\left\vert \sigma\right\vert -1}...\sigma_{2}\right)  }^{\sigma
_{2}^{+}}\\
&  \subset T_{\xi\left(  \sigma\right)  }^{\sigma_{1}^{+}}%
\end{align*}
which is equivalent to equation \ref{hierarchy}.
\end{proof}

\subsection{Absolute addresses}

The set of \textit{absolute addresses} associated with $(\mathcal{F}%
,\mathcal{G)}$ is
\[
\mathbb{A}:=\{\theta.\sigma:\theta\in\Sigma_{\ast}^{\dag},\,\sigma^{-}%
=\theta^{+},\,\theta_{\left\vert \theta\right\vert }\neq\sigma_{1}\}.
\]
Define $\widehat{\Pi}:\mathbb{A\rightarrow\{}t\in T:T\in\mathbb{T\}}$ by
\[
\widehat{\Pi}(\theta.\omega)=f_{-\theta}.f_{\sigma}(A_{\sigma^{+}}).
\]
The condition $\theta_{\left\vert \theta\right\vert }\neq\sigma_{1}$ is
imposed. We say that $\theta.\sigma$ is an \textit{absolute address} of the
tile $f_{-\theta}.f_{\omega}(A)$. It follows from Definition \ref{defONE} that
the map $\widehat{\Pi}$ is surjective: every tile of $\mathbb{\{}t\in
T:T\in\mathbb{T\}}$ possesses at least one absolute address.

Although tiles have unique relative addresses, relative to the $T_{k}^{v}$ to
which they are being treated as belonging, they may have many different
absolute addresses. The tile $[1,1.5]$ of Example \ref{ex01} has the two
absolute addresses $1.21$ and $21.211$, and many others.

\subsection{Relationship between relative and absolute addresses}

\begin{theorem}
\label{relabsthm}If $t\in\Pi(\theta)$ with $\theta\in\Sigma_{\ast}^{\dag}$ has
relative address $\omega$ relative to $\Pi(\theta),$ then an absolute address
of $t$ is $\theta_{1}\theta_{2}...\theta_{l}.S^{\left\vert \theta\right\vert
-l}\omega$ where $l\in\mathbb{N}$ is the unique index such that
\begin{equation}
t\in\Pi(\theta_{1}\theta_{2}...\theta_{l})\text{ and }t\notin\Pi(\theta
_{1}\theta_{2}...\theta_{l-1}) \label{*eqn}%
\end{equation}

\end{theorem}

\begin{proof}
Recalling that%
\[
\Pi(\theta|0)\subset\Pi(\theta_{1})\subset\Pi(\theta_{1}\theta_{2}%
)\subset...\subset\Pi(\theta_{1}\theta_{2}...\theta_{\left\vert \theta
\right\vert -1})\subset\Pi(\theta),
\]
we have the disjoint union
\[
\Pi(\theta)=\Pi(\theta|0)\cup\left(  \Pi(\theta_{1})\backslash\Pi
(\varnothing)\right)  \cup\left(  \Pi(\theta_{1}\theta_{2})\backslash
\Pi(\theta_{1})\right)  \cup...\cup\left(  \Pi(\theta)\backslash\Pi(\theta
_{1}\theta_{2}...\theta_{\left\vert \theta\right\vert -1})\right)  .
\]
So there is a unique $l$ such that Equation (\ref{*eqn}) is true. Since
$t\in\Pi(\theta)$ has relative address $\varnothing.\sigma$ relative to
$\Pi(\theta)$ we have
\[
\varnothing.\sigma=\varnothing.\pi^{-1}f_{-\theta}^{-1}(t)
\]
and so an absolute address of $t$ is
\[
\theta.\sigma|_{cancel}=\theta.\pi^{-1}f_{-\theta}^{-1}(t)|_{cancel}%
\]
where $|_{cancel}$ means equal symbols on either side of $``."$ are removed
until there is a different symbol on either side. Since $t\in\Pi(\theta
_{1}\theta_{2}...\theta_{l})$ the terms $\theta_{l+1}\theta_{l+2}%
...\theta_{\left\vert \theta\right\vert }$ must cancel yielding the absolute
address%
\[
\theta.\sigma|_{cancel}=\theta_{1}\theta_{2}...\theta_{l}.\sigma
_{|\theta|-l+1}...\sigma_{|\sigma|}%
\]

\end{proof}

\subsection{\label{specificinflation}Inflation and deflation of $\Pi(\theta)$
when $\theta$ is known}

\begin{definition}
\label{infldef1} The \textbf{deflation operator} $\alpha$ and its inverse, the
\textbf{inflation} \textbf{operator} $\alpha^{-1}$, \textit{both restricted to
canonical tilings }$T_{k}^{v}$ where $k\in\mathbb{N}$ and $v\in\mathcal{V}$
are specified, is defined by%
\[
\alpha T_{k}^{v}=T_{k-1}^{v},\text{ }\alpha^{-1}T_{k-1}^{v}=T_{k}^{v}%
\]
for all specified $k\in\mathbb{N}$, $v\in\mathcal{V}$. The domains of $\alpha$
and $\alpha^{-1}$ are extended to include any specified isometry
$E\in\mathcal{U}$ applied to $T_{k}^{v}$, by defining
\begin{align}
\alpha ET_{k}^{v}  &  =\left(  sEs^{-1}\right)  \alpha T_{k}^{v}=\left(
sEs^{-1}\right)  T_{k-1}^{v}\label{caninflation}\\
\alpha^{-1}ET_{k-1}^{v}  &  =\left(  s^{-1}Es\right)  T_{k}^{v}\nonumber
\end{align}
for all $k\in\mathbb{N}$, $v\in\mathcal{V}$.
\end{definition}

Note that $\alpha^{m}\alpha^{n}(ET_{k}^{v})$ is well-defined and equals
$\alpha^{m+n}\left(  ET_{k}^{v}\right)  $ for all $n,m\in\mathbb{N}_{0}$ with
$n+m\geq-k$ and $n\geq-k$ where we define $\alpha^{0}$ to be an identity map.

Note that the tiling $\alpha^{-1}T_{k-1}^{v}$ may be calculated by replacing
each tile $t\in T_{k-1}^{v}$ whose relative address (relative to $T_{k-1}^{v}
$) $\varnothing.\sigma$ obeys $\xi(\sigma)=k-1$ by the set of tiles in
$T_{k}^{v}$ whose relative addresses (relative to $T_{k}^{v}$) are
$\varnothing.\sigma i$ where $i^{-}=\sigma^{+}$; and (ii) replacing each tile
$t\in T_{k-1}^{v}$ whose relative address $\varnothing.\sigma$ obeys
$\xi(\sigma)>k-1$ by $s^{-1}t$. Conversely, $\alpha T_{k}^{v}$ can be
calculated by replacing each tile in $T_{k}^{v}$ whose relative addresses
(relative to $T_{k}^{v}$) take the form $\varnothing.\sigma i$ where
$i^{-}=\sigma^{+}$ for some fixed $\sigma$ with $\xi(\sigma)=k,$ by the tile
in $T_{k-1}^{v}$ whose relative address (relative to $T_{k-1}^{v})$ is
$\varnothing.\sigma$.

\begin{definition}
\label{uniondef}The domains of $\alpha$ and $\alpha^{-1}$ are extended to
include $E\Pi(\theta),$ for any specified isometry $E\in\mathcal{U}$ and
$\theta\in\Sigma^{\dag}$, by defining:
\begin{align*}
\alpha(E\Pi(\theta|k))  &  =sEf_{-(\theta|k)}s^{\xi(\theta|k)-1}T_{\xi
(\theta|k)-1}^{(\theta|k)^{+}}\text{ for all }k\leq\left\vert \theta
\right\vert ,k\in\mathbb{N}_{0}\\
\alpha^{-1}(E\Pi(\theta|k))  &  =s^{-1}Ef_{-(\theta|k)}s^{\xi(\theta
|k)+1}T_{\xi(\theta|k)+1}^{(\theta|k)^{+}}\text{ for all }k\leq\left\vert
\theta\right\vert ,k\in\mathbb{N}_{0}\\
\alpha^{K}(E\Pi(\theta))  &  =\bigcup\limits_{k=0}^{\infty}s^{K}%
Ef_{-(\theta|k)}s^{\xi(\theta|k)-K}T_{\xi(\theta|k)-K}^{(\theta|k)^{+}}\text{
if }\left\vert \theta\right\vert =\infty,K\in\mathbb{N}_{0}\\
\alpha^{-K}(E\Pi(\theta))  &  =\bigcup\limits_{k=0}^{\infty}s^{-K}%
Ef_{-(\theta|k)}s^{\xi(\theta|k)+K}T_{\xi(\theta|k)+K}^{(\theta|k)^{+}}\text{
if }\left\vert \theta\right\vert =\infty,K\in\mathbb{N}_{0}%
\end{align*}

\end{definition}

Theorem \ref{keythm2} tells us that the unions in this definition are
increasing unions of nested sequences, and hence that the actions of $\alpha$
and $\alpha^{-1}$ are well-defined on their extended domains, provided that
the indices $K\in\mathbb{N}_{0}$, $E\in\mathcal{U},$ $\theta\in\Sigma^{\dag}$
are specified.

\begin{theorem}
\label{keythm2} Let $(\mathcal{F}$,$\mathcal{G)}$ be a tiling IFS. Then
\begin{equation}
\alpha^{K}(E\Pi(\theta|M))\subset\alpha^{K}(E\Pi(\theta|M+1))\subset...
\label{betaequation}%
\end{equation}
for all $M\in\mathbb{N},$ $K\in\mathbb{Z},$ $K<\xi(\theta|M),$ $\theta
\in\Sigma_{\infty}^{\dag}$. Then tilings produced by the actions of
$\alpha^{K}$ on $E\Pi(\theta)$ are well defined by Definition \ref{uniondef}.
Moreover, for all $\theta\in\Sigma^{\dag}$, $n\in\left[  N\right]  ,$
$k\in\mathbb{N}_{0},$ with $E_{\theta|k}:=f_{-(\theta|k)}s^{\xi(\theta|k)},$
we have the following identities%
\begin{align}
\alpha^{a_{\theta_{1}}}\Pi(\theta)  &  =s^{a_{\theta_{1}}}f_{\theta_{1}}%
^{-1}\Pi(S\theta)\label{alphaequation}\\
\alpha^{-a_{n}}\Pi(\theta)  &  =s^{-a_{n}}f_{n}\Pi(n\theta)\nonumber\\
\Pi(S^{k}\theta)  &  =\alpha^{\xi(\theta|k)}E_{\theta|k}^{-1}\Pi
(\theta)\nonumber
\end{align}
In the last equality, we require $k<|\theta|$.
\end{theorem}

\begin{proof}
The crucial point is that the unions in Definition \ref{uniondef} are
increasing unions (i.e. each successive collection of tiles contains its
predecessor). The nestedness in Equation (\ref{betaequation}) follows from the
equivalence of the following statements.
\begin{align*}
s^{K}Ef_{-\left(  \theta|k\right)  }s^{\xi(\theta|k)-K}T_{\xi(\theta
|k)-K}^{(\theta|k)^{+}}  &  \subset s^{K}Ef_{-\left(  \theta|k+1\right)
}s^{\xi(\theta|k+1)-K}T_{\xi(\theta|k+1)-K}^{(\theta|k+1)^{+}}\\
f_{-\left(  \theta|k\right)  }s^{\xi(\theta|k)-K}T_{\xi(\theta|k)-K}%
^{(\theta|k)^{+}}  &  \subset f_{-\left(  \theta|k+1\right)  }s^{\xi
(\theta|k+1)-K}T_{\xi(\theta|k+1)-K}^{(\theta|k+1)^{+}}\\
s^{+\xi(\theta|k)-K}T_{\xi(\theta|k)-K}^{(\theta|k)^{+}}  &  \subset
f_{-\theta_{k+1}}s^{\xi(\theta|k+1)-K}T_{\xi(\theta|k+1)-K}^{(\theta|k+1)^{+}%
}\\
\{f_{\sigma}(A^{(\theta|k)^{+}})  &  :\sigma\in\Omega_{\xi(\theta
|k)-K}^{(\theta|k)^{+}}\}\subset f_{-\theta_{k+1}}\{f_{\sigma}(A^{(\theta
|k+1)^{+}}):\sigma\in\Omega_{\xi(\theta|k+1)-K}^{(\theta|k+1)^{+}}\}\\
\Omega_{\xi(\theta|k+1)-K}^{(\theta|(k+1))^{+}}  &  \supset\{\theta
_{k+1}\sigma:\sigma\in\Omega_{\xi(\theta|k)-K}^{(\theta|k+1)^{+}},\theta
_{k+1}^{+}=\sigma^{-}\}
\end{align*}
Next we prove that $\alpha^{\xi(\theta|m)}\Pi(\theta)=s^{\xi(\theta
|m)}f_{-\theta|m}\Pi(S^{m}\theta)$ and in particular that $\alpha
^{a_{\theta_{1}}}\Pi(\theta)=s^{a_{\theta_{1}}}f_{\theta_{1}}^{-1}\Pi
(S\theta).$
\begin{align*}
\alpha^{\xi(\theta|m)}\Pi(\theta)  &  =\bigcup\limits_{k=K}^{\infty}%
s^{\xi(\theta|m)}f_{-(\theta|k)}s^{\xi(\theta|k)-\xi(\theta|m)}T_{\xi
(\theta|k)-\xi(\theta|m)}^{\left(  \theta|k\right)  ^{+}}\\
&  =\bigcup\limits_{k=m}^{\infty}s^{\xi(\theta|m)}f_{-(\theta|k)}s^{\xi
(\theta|k)-\xi(\theta|m)}T_{\xi(\theta|k)-\xi(\theta|m)}^{\left(
\theta|k\right)  ^{+}}\\
&  =\bigcup\limits_{k=m}^{\infty}s^{\xi(\theta|m)}f_{-(\theta|m)}%
f_{-(S^{m}\theta|k-m)}s^{\xi(\theta|m)}s^{\xi(S^{m}\theta|k-m)}T_{\xi
(S^{m}\theta|k-m)}^{\left(  \theta|k\right)  ^{+}}\\
&  =s^{\xi(\theta|m)}f_{-(\theta|m)}\bigcup\limits_{k=m}^{\infty}%
f_{-(S^{m}\theta|k-m)}s^{\xi(S^{m}\theta|k-m)}T_{\xi(S^{m}\theta
|k-m)}^{\left(  \theta|k\right)  ^{+}}\\
&  =s^{\xi(\theta|m)}f_{-\left(  \theta|m\right)  }\Pi(S^{m}\theta)
\end{align*}
Proofs of the remaining two equalities in Equation (\ref{alphaequation})
follow similarly.
\end{proof}

\begin{remark}
Notice that for $\alpha$ or $\alpha^{-1}$ to act on a tiling $\Pi(\theta)$, as
in Theorem \ref{keythm2}$,$ it is necessary that $\theta$ is known: that is,
$\alpha$ acts on the function $\Pi:\Sigma^{\dag}\rightarrow\Pi(\Sigma^{\dag})$
or equivalently on the graph $\{\left(  \Pi(\theta),\theta\right)  :\theta
\in\Sigma^{\dag})\}$. For example the statement $\Pi(\theta)=\Pi(\psi)$ does
not imply $\alpha\Pi(\theta)=\alpha\Pi(\psi)$ without more information.
\end{remark}

\section{\label{secrigid}Rigid tiling IFSs}

Call a tiling $T$ an \textbf{isometric combination of canonical tilings} if it
can be written in the form
\[
T=\cup_{i\in\mathcal{I}}E_{i}T_{k_{i}}^{v_{i}}%
\]
where $\mathcal{I}$ is a countable index set, $v_{i}\in\mathcal{V}$, $k_{i}%
\in\mathbb{N}_{0}$ for all $i\in\mathcal{I}$, and it is assumed that
$E_{i},v_{i},k_{i}$ are known for all $i\in\mathcal{I}$. For example the
tiling $\Pi(\theta)$ where $\theta$ is given is an isometric combination of
canonical tilings for all $\theta\in\Sigma^{\dag}$. Inflation and deflation of
a tiling $T$ may not be well-defined when it is represented as an isometric
combination of canonical tilings. For example it can occur that $T=T_{v}%
^{k}=\cup_{i\in\mathcal{I}}E_{i}T_{k_{i}}^{v_{i}}$ but $\alpha T\neq\cup
_{i\in\mathcal{I}}\alpha\left(  E_{i}T_{k_{i}}^{v_{i}}\right)  $ as the
following example shows.

\begin{example}
\label{ex3} In $\mathbb{R}$ let $f_{1}(x)=\frac{1}{2}x,$ $f_{2}(x)=\frac{1}%
{4}x+\frac{1}{2},$ $f_{3}(x)=\frac{1}{4}x+\frac{1}{4},$ $f_{4}(x)=\frac{1}%
{2}x+2,$ $f_{5}(x)=\frac{1}{2}x+\frac{3}{2}$ and let $Ex=x-1$. Then observe
that
\begin{align*}
A_{1}  &  =f_{1}(A_{1})\cup f_{2}(A_{1})\cup f_{3}(A_{2}),A_{2}=f_{4}%
(A_{1})\cup f_{5}(A_{2})\\
T_{0}^{1}  &  =\{[0,\frac{1}{2}],[\frac{1}{2},\frac{3}{4}],[\frac{3}%
{4},1]\},T_{0}^{2}=\{[2,\frac{5}{2}],[\frac{5}{2},3]\}\\
T_{1}^{1}  &  =T_{0}^{1}\cup ET_{0}^{2},\\
\alpha T_{1}^{1}  &  =T_{0}^{1}\neq\alpha T_{0}^{1}\cup\alpha ET_{0}^{2}%
=T_{0}^{1}\cup sE[0,1]
\end{align*}
Note that $EsT_{0}^{2}\subset T_{0}^{1}$ where $Es[2,3]=[\frac{1}{2},1]$ and
$s=\frac{1}{2}.$
\end{example}

In this Section \ref{secrigid} we define the notions of a rigid tiling IFS
$(\mathcal{F},\mathcal{G)}$ and a rigid tiling $T$. We extend the definitions
of $\alpha$ and $\alpha^{-1}$ so that they act directly on tilings, in such a
way that if $T$ is a rigid tiling and $T=\cup_{i\in\mathcal{I}}E_{i}T_{k_{i}%
}^{v_{i}}$ with $v_{i}\in\mathcal{V}$ and $k_{i}\in\mathbb{N}$ is an isometric
combination, then
\[
\alpha T=\cup_{i\in\mathcal{I}}\alpha\left(  E_{i}T_{k_{i}}^{v_{i}}\right)
=\cup_{i\in\mathcal{I}}sE_{i}s^{-1}T_{k_{i}-1}^{v_{i}}%
\]
and similarly for $\alpha^{-1}$ independently of the specific representation
of $T$ as an isometric combination.

\subsection{Definitions\label{rigidsec1}}

Let $\mathcal{U}$ be any set of isometries on $\mathbb{R}^{M}$ that contains
the set of isometries $\{s^{m}f_{-\theta}f_{\sigma}:m\in\{0,1,...,a_{\max
}-1\},\theta\in{\Sigma}_{\ast}^{\dag},\sigma\in{\Sigma}_{\ast},\theta
^{+}=\sigma^{-},$ $m+\xi(\sigma)-\xi(\theta)=0\}.$ It may be a group such as
the group of translations or the Euclidean group on $\mathbb{R}^{M}$.

\begin{definition}
\label{meetdef} If $P$ and $Q$ are sets of subsets of $\mathbb{R}^{M}$ we say
\textquotedblleft$P$ \textbf{meets} $Q$\textquotedblright, to mean that $P\cap
Q\neq\varnothing$ and $\left(  \cup P\right)  \cap$ $\left(  \cup Q\right)
=\cup\left(  P\cap Q\right)  $. We also say that \textquotedblleft$P$ is a
\textbf{copy} of $Q$\textquotedblright\ to mean that there is $E\in
\mathcal{U}$ such that $P=EQ$. For example, \textquotedblleft$T_{k}^{v}$ meets
a copy of $T_{l}^{w}$\textquotedblright\ is shorthand for \textquotedblleft
there is $E\in\mathcal{U}$ such that $T_{k}^{v}\cap ET_{l}^{w}\neq\varnothing$
and the union of the set of tiles $T_{k}^{v}\cap ET_{l}^{w}$ is $s^{-k}%
A_{v}\cap Es^{-l}A_{w}$\textquotedblright$.$
\end{definition}

\begin{definition}
\label{localdef} The tilings $\mathbb{T}:=\left\{  \Pi(\theta):\theta\in
\Sigma^{\dag}\right\}  $ and the tiling IFS $(\mathcal{F},\mathcal{G)}$ are
each said to be \textbf{rigid} (with respect to $\mathcal{U}$) when the
following three statements are true for all $E\in\mathcal{U}$, and all
$v,w\in\mathcal{V}$:

A(i) if $T_{0}^{v}$ meets $Es^{k}T_{0}^{w}$ for some $k\in\mathbb{\{}%
0,1,...,a_{\max}-1\}$ then $E=Id$, $k=0,$ and $v=w$;

A(ii) if $ET_{0}^{v}$ tiles $A_{w}$ then $E=Id$ and $v=w;$

A(iii) if $A_{w}=Es^{k}A_{v}$ for some $k\in\mathbb{N}_{0},$ then $E=Id,$
$k=0,$ and $v=w.$
\end{definition}

Definition \ref{localdef} is weaker than the definition of strongly rigid in
the case $\left\vert \mathcal{V}\right\vert =1$ in \cite{barnsleyvince}. For
tiles with non-empty interiors, if $\mathcal{U}$ is the group of translations
on $\mathbb{R}^{M},$ and $a_{\max}=1$, rigidity is largely equivalent to
recognizability \cite{anderson} and to the unique composition property
\cite{solomyak2}. Rigidity extends these concepts to tilings involving more
than one scaling factor, more general sets of transformations, and to the
context of more general fractal tilings.

\begin{lemma}
\label{keylemma} Let the family of tilings $\mathbb{T}:=\left\{  \Pi
(\theta):\theta\in\Sigma^{\dag}\right\}  $ and the tiling IFS $(\mathcal{F}%
,\mathcal{G)}$ be rigid. If $s^{k}T_{0}^{v}$ meets $ET_{l}^{w}$ for some
$k,l\in\mathbb{N}_{0},$ $v,w\in\mathcal{V}$, $E\in\mathcal{U}$, then $k=0$ and
$T_{0}^{v}\subset ET_{l}^{w}$.
\end{lemma}

\begin{proof}
If $s^{k}T_{0}^{v}$ meets $ET_{0}^{w}$ then $k=0,$ $E=Id,$ $v=w.$ In
particular, if $s^{k}T_{0}^{v}$ meets $ET_{0}^{w}$ then $k=0,$ and $T_{0}%
^{v}\subset ET_{0}^{w}.$ Suppose that if $s^{k}T_{0}^{v}$ meets $ET_{l}^{w}$
then $k=0,$ and $T_{0}^{v}\subset ET_{l}^{w},$ for all $l=0,1,2,..L$. If
$s^{k}T_{0}^{v}$ meets $ET_{L+1}^{w},$ but does not meet any copy of
$T_{0}^{x}$ contained in $ET_{L+1}^{w}$ we can apply $\alpha$ to $ET_{l}^{w}$
and at the same time shrink $s^{k}T_{0}^{v}$ without modification, yielding
that $s^{k+1}T_{0}^{v}$ meets $T_{l-1}^{w}$ where $E^{\prime}=sEs^{-1}%
\in\mathcal{U}$. This implies $k=-1$ which is false. We conclude that
$s^{k}T_{0}^{v}$ meets a copy of $T_{0}^{x}$ contained in $ET_{L+1}^{w}$ which
implies $k=0$ and $T_{0}^{v}\subset ET_{L+1}^{w}$.
\end{proof}

\begin{theorem}
\label{equivalence} If the family of tilings $\mathbb{T}:=\left\{  \Pi
(\theta):\theta\in\Sigma^{\dag}\right\}  $ and the tiling IFS $(\mathcal{F}%
,\mathcal{G)}$ are rigid then the following four statements are true.

B(i) if $E\in\mathcal{U},$ $v,w\in\mathcal{V},$ and $T_{0}^{v}\ $meets
$ET_{0}^{w}$, then $v=w$ and $E=Id$;

B(ii) if $E\in\mathcal{U}$, $v,w\in\mathcal{V},$ and $k,l\in\mathbb{N}_{0}$
are such that $T_{k}^{v}\ $meets $ET_{l}^{w},$ then%
\[
\text{either }T_{k}^{v}\subset ET_{l}^{w}\text{ or }ET_{l}^{w}\subset
T_{k}^{v}%
\]

B(iii) if $E\in\mathcal{U}$, $v,w\in\mathcal{V},$ and $ET_{0}^{v}$ tiles
$A_{w}$, then $E=Id$ and $v=w$;

B(iv) if $A_{w}=Es^{k}A_{v}$ for some $E\in\mathcal{U},v\in\mathcal{V}%
,k\in\mathbb{N}_{0},\ $then $E=Id,k=0,$ and $v=w.$

If $\left\vert \mathcal{V}\right\vert =1$ or if each $T_{0}^{v}$ possesses a
tile isometric to $s^{a_{\max}}A_{w}$, for some $w$ that may depend on $v$,
then the two sets of conditions, \{A(i),A(ii),A(iii)\} and
\{B(i),B(ii),B(iii),B(iv)\} are equivalent.
\end{theorem}

\begin{proof}
Follows from Lemma \ref{keylemma}.
\end{proof}

\begin{corollary}
\label{contain} Let the family of tilings $\mathbb{T}:=\left\{  \Pi
(\theta):\theta\in\Sigma^{\dag}\right\}  $ and the tiling IFS $(\mathcal{F}%
,\mathcal{G)}$ be rigid. If $\theta,\varphi\in\Sigma_{\ast}^{\dag},$ and
$\Pi(\theta)\ $meets $E\Pi(\varphi),$ then
\[
\text{either }\Pi(\theta)\subset E\Pi(\varphi)\text{ or }E\Pi(\varphi
)\subset\Pi(\theta)
\]

\end{corollary}

\subsection{\label{extendinflation}Inflation and deflation of rigid tilings}

Let $\mathbb{Q}$ be the set of all tilings $T$ that can be written in the form
$T=\cup_{i\in\mathcal{I}}E_{i}T_{k_{i}}^{v_{i}}$ where $i$ is a countable
index set, $E_{i}\in\mathcal{U}$, $k_{i}\in\mathbb{N}_{0},$ and $v_{i}%
\in\mathcal{V}$ for all $i\in\mathcal{I}$. Let $\mathbb{Q}^{\prime}$ be the
set of all tilings $T^{\prime}$ that can be written in the form $T^{\prime
}=\cup_{i\in\mathcal{I}}E_{i}T_{k_{i}-1}^{v_{i}}$ where $i$ is a countable
index set, $E_{i}\in\mathcal{U}$, $k_{i}\in\mathbb{N}_{0},$ and $v_{i}%
\in\mathcal{V}$ for all $i\in\mathcal{I}$.

The following definition extends the domains of $\alpha$ and $\alpha^{-1}$ to
$\mathbb{Q}$ and $\mathbb{Q}^{\prime}$ respectively, in the case of rigid
tilings. It generalizes the definition of strongly rigid in
\cite{barnsleyvince} to the graph directed case. It relies on the fact,
assured by Lemma \ref{keylemma}, that no \textquotedblleft spurious
copies\textquotedblright\ of any $T_{0}^{v}$ can occur in any tiling in
$\mathbb{Q}$.

\begin{definition}
\label{inflationdef} Let $\left(  \mathcal{F},\mathcal{G}\right)  $ be a rigid
tiling IFS. \textbf{Deflation} $\alpha:\mathbb{Q}\rightarrow\mathbb{Q}%
^{\prime}$ is defined by $\alpha(T)=\{\alpha(t):t\in T\}$ for all $t\in
T\in\mathbb{Q}$, where
\[
\alpha(t):=\left\{
\begin{array}
[c]{cc}%
sEA_{v} & \text{if }t\in ET_{0}^{v}\subset T\text{ for some }E\in
\mathcal{U}\text{, }v\in\mathcal{V}\text{,}\\
st & \text{otherwise}%
\end{array}
\right.
\]
$ET_{0}^{v}$ is called the set of \textbf{partners }of $t\in ET_{0}^{v}$. If
$t_{1}$ and $t_{2}$ are partners of $t,$ then $\alpha(t_{1})=\alpha(t_{2})$.
\textbf{Inflation} $\alpha^{-1}:\mathbb{Q}^{\prime}\mathbb{\rightarrow Q}$ is
defined by $\alpha^{-1}T=\{\alpha^{-1}(t):t\in T\}$ for all $t\in
T\in\mathbb{Q}^{\prime}$, where%
\[
\alpha^{-1}(t):=\left\{
\begin{array}
[c]{cc}%
s^{-1}t & \text{if }t\neq EsA_{v}\text{ for any }E\in\mathcal{U}\text{, }%
v\in\mathcal{V}\text{,}\\
ET_{0}^{v} & \text{if }t=EsA_{v}%
\end{array}
\right.
\]
for all $T\in\mathbb{Q}^{\prime}$.
\end{definition}

Conditions \textit{A(ii) }and \textit{A(iii)} ensure that inflation,
represented by the operator $\alpha^{-1},$ is well-defined on $\mathbb{Q}%
^{\prime}$. Call a tile in any tiling in $\mathbb{Q}^{\prime}$ which is
isometric to $sA_{v}$ for some $v\in\mathcal{V}$ a \textbf{large tile}. To
inflate a tiling $T^{\prime}$ in $\mathbb{Q}^{\prime}$, first replace each
large tile in $T^{\prime}$ by the corresponding unique (by \textit{A(ii)})
copy of $sT_{0}^{v}$ (for all $v$), yielding a set of sets $T^{\prime}$, and
then apply the similitude $s^{-1}$ to $T^{\prime}$ to yield $T\in\mathbb{Q}$.
Similarly, deflation is well-defined, because by Lemma \ref{keylemma} no
copies of $s^{k}T_{0}^{v}$ with $k>0$ can occur in any $T_{l}^{w}$.

Condition $A(iii)$ ensures that, given the canonical tiling $T_{k}^{v},$ we
can infer the values of the indices $v$ and $k$. In the case $a_{\max}=1,$ a
consequence of rigidity (with respect to the translation group) is that
canonical tilings are recognizable, as discussed in Section \ref{relationsec}.

For rigid tilings $\alpha:\mathbb{Q}\rightarrow\mathbb{Q}^{\prime}$ and
$\alpha^{-1}:\mathbb{Q}^{\prime}\mathbb{\rightarrow Q}$ are well-defined.
Every copy of $T_{0}^{w}$ in $T_{k}^{v}$ is related via $\alpha^{-1}$ to a
large tile in $T_{k-1}^{v}$. There is a one-to-one correspondence between the
large tiles in $T_{k-1}^{v}$ and copies of $T_{0}^{x}$ in $T_{k}^{v}.$ In
particular we find that $\alpha$ and $\alpha^{-1}$ in Definition
\ref{inflationdef} are consistent with the definition in Section
\ref{specificinflation}. The following theorem says that, for rigid tilings,
inflation and deflation are well defined, in particular they interact in an
unconfusing manner on isometric combinations.

\begin{theorem}
\label{welldefined} If $(\mathcal{F}$,$\mathcal{G)}$ is rigid, then the
following statements are true for all $E,E^{\prime}\in\mathcal{U}$,
$k,l\in\mathbb{N}$, $v,w\in\mathcal{V}$, and index sets $\mathcal{I}%
,\mathcal{I}^{\prime},\mathcal{J},\mathcal{J}^{\prime}$,

(i) $ET_{0}^{v}\subset T_{k}^{w}$ if and only if $sEA_{v}\in T_{k-1}^{w}$

(ii) $\alpha$ and $\alpha^{-1}$ in Definition \ref{inflationdef} are
consistent with Definition \ref{infldef1} in Section \ref{specificinflation},
that is
\[
\alpha(ET_{k}^{v})=\bigcup\limits_{t\in ET_{k}^{v}}\alpha(t)\text{ and }%
\alpha^{-1}(ET_{k}^{v})=\bigcup\limits_{t\in ET_{k}^{v}}\alpha^{-1}(t)
\]

(iii) if $ET_{k}^{v}\subset E^{\prime}T_{l}^{w}$, then
\[
\alpha\left(  ET_{k}^{v}\right)  \subset\alpha\left(  E^{\prime}T_{l}%
^{w}\right)  \text{ and }\alpha^{-1}\left(  ET_{k}^{v}\right)  \subset
\alpha^{-1}\left(  E^{\prime}T_{l}^{w}\right)
\]

(iv) if $\cup_{i\in\mathcal{I}}E_{i}T_{k_{i}}^{v_{i}}\in\mathbb{Q}$, and
$\cup_{j\in\mathcal{J}}E_{j}T_{k_{j}}^{v_{j}}\in\mathbb{Q}^{\prime}$, then
\[
\alpha(\cup_{i\in\mathcal{I}}E_{i}T_{k_{i}}^{v_{i}})=\cup_{i\in\mathcal{I}%
}\alpha(E_{i}T_{k_{i}}^{v_{i}})\text{ and }\alpha^{-1}(\cup_{j\in\mathcal{J}%
}E_{j}T_{k_{j}}^{v_{j}})=\cup_{j\in\mathcal{J}}\alpha^{-1}(E_{j}T_{k_{j}%
}^{v_{j}})
\]

(v) if $\cup_{i\in\mathcal{I}}E_{i}T_{k_{i}}^{v_{i}}\subset\cup_{i\in
\mathcal{I}^{\prime}}E_{i}^{\prime}T_{k_{i}^{\prime}}^{v_{i}^{\prime}}%
\in\mathbb{Q}$ and $\cup_{j\in\mathcal{J}}E_{j}T_{k_{j}}^{v_{j}}\subset
\cup_{j\in\mathcal{J}^{\prime}}E_{j}^{\prime}T_{k_{j}^{\prime}}^{v_{j}%
^{\prime}}\in\mathbb{Q}^{\prime},$ then
\[
\alpha(\cup_{i\in\mathcal{I}}E_{i}T_{k_{i}}^{v_{i}})\subset\alpha(\cup
_{i\in\mathcal{I}^{\prime}}E_{i}^{\prime}T_{k_{i}^{\prime}}^{v_{i}^{\prime}%
})\text{ and }\alpha^{-1}(\cup_{j\in\mathcal{J}}E_{j}T_{k_{j}}^{v_{j}}%
)\subset\alpha^{-1}(\cup_{j\in\mathcal{J}^{\prime}}E_{j}^{\prime}%
T_{k_{j}^{\prime}}^{v_{j}^{\prime}})
\]

\end{theorem}

\begin{proof}
These statements follow from Theorem \ref{equivalence}.
\end{proof}

\begin{corollary}
\label{corollaryX} Let $(\mathcal{F}$,$\mathcal{G)}$ be rigid and $\Pi
(\theta)\subset E\Pi(\psi)$, for some $\theta,\psi\in\Sigma^{\dag}$. Then
$\alpha^{k}\Pi(\theta)\subset s^{k}Es^{-k}\alpha^{k}\Pi(\psi)$ for all
$k\in\mathbb{N}$, with $k\leq\min\{\xi(\theta),\xi(\psi)\}$ when both $\theta$
and $\psi$ lie in $\Sigma_{\ast}^{\dag}$. Also $\alpha^{-k}\Pi(\theta)\subset
s^{-k}Es^{k}\alpha^{-k}\Pi(\psi)$.
\end{corollary}

\begin{proof}
This follows directly using the above identities.
\end{proof}

\section{\label{charismasec}Characterization of isometric rigid tilings}

Define for all $k\in\mathbb{N}$ and $v,w\in\mathcal{V}$
\[
\Lambda_{k}^{v,w}=\{\sigma\in\Sigma_{\ast}:\xi(\sigma)=k,\sigma^{-}%
=v,\sigma^{+}=w\}\subset\Omega_{k-1}^{v}%
\]

\begin{theorem}
\label{keythm} Let $\left(  \mathcal{F},\mathcal{G}\right)  $ be a rigid
tiling IFS. For all $k\in\mathbb{N}_{0}$ there is a bijection between
$\Lambda_{k}^{v,w}$ and the set of isometric copies of $T_{0}^{w}$ contained
in $T_{k}^{v}$. The bijection is provided by the map $H:\Lambda_{k}%
^{v,w}\rightarrow\mathcal{R}(H)\subset$ $T_{k}^{v}$ defined by
\[
H(\sigma)=s^{-k}f_{\sigma}(T_{0}^{w})
\]
where $\mathcal{R(H)}$ is the range of $H$.
\end{theorem}

\begin{proof}
(i) It is readily checked that $H(\Lambda_{k}^{v,w})\subset T_{k}^{v}$. (ii)
Suppose $H\left(  \sigma\right)  =H(\omega)$ for $\sigma,\omega\in\Lambda
_{k}^{v,w}$. Then $\xi(\sigma)=\xi(\omega)=k,$ $\sigma^{+}=\omega^{+}=w$,
$\sigma^{-}=\omega^{-}=v$ and
\[
s^{-k}f_{\sigma}(T_{0}^{w})=s^{-k}f_{\omega}(T_{0}^{w})\Rightarrow f_{\sigma
}(A_{v})=f_{\omega}(A_{v})\Rightarrow\sigma=\omega
\]
(iii) Suppose that $ET_{0}^{w}\subset T_{k}^{v}$ is an isometric copy of
$T_{0}^{w}$ that is contained in $T_{k}^{v}$. Then we need to show that
$ET_{0}^{w}$ is in $\mathcal{R}(H).$ We have
\[
\alpha ET_{0}^{w}\subset\alpha T_{k}^{v}\Rightarrow sEs^{-1}sA_{w}\in
T_{k-1}^{v}\Rightarrow sEs^{-1}sA_{w}=s^{-k+1}f_{\sigma}(A_{w})
\]
for some $\sigma$ such that $\sigma^{+}=w$, $\sigma^{-}=v$, $\xi(\sigma)=k,$
because the r.h.s. must be a tile in $T_{k-1}^{v}$ congruent to $sA_{w}$. It
follows that $E=s^{-k}f_{\sigma}$ where $\sigma\in\Lambda_{k}^{v,w}$and so
$H(\sigma)=ET_{0}^{w}\in\mathcal{R}(H)$, because any copy of $T_{0}^{w}$ in
$ET_{k}^{v}$ must equal the result of application of $\alpha^{-1}$ to a copy
of $sA_{v}$ in $T_{k-1}^{v}$.
\end{proof}

\begin{theorem}
\label{basictheorem} Let $\left(  \mathcal{F},\mathcal{G}\right)  $ be a
tiling IFS.

(i) If $\theta,\psi\in\Sigma_{\infty}^{\dag}$, $S^{p}\theta=S^{q}\psi,$
$E=f_{-\theta|p}(f_{-\psi|q})^{-1},$ $\left(  \theta|p\right)  ^{+}=\left(
\psi|q\right)  ^{+}$, and $\xi\left(  \theta|p\right)  =\xi\left(
\psi|q\right)  ,$ then $\Pi(\theta)=E\Pi(\psi)$ where $E$ is an isometry.

(ii) Let $\left(  \mathcal{F},\mathcal{G}\right)  $ be rigid, and let
$\Pi(\theta)=E\Pi(\psi)$ where $E\in\mathcal{U}$ is an isometry, for some pair
of addresses $\theta,\psi\in\Sigma_{\infty}^{\dag}.$ Then there are
$p,q\in\mathbb{N}$ such that $S^{p}\theta=S^{q}\psi,$ $E=f_{-\left(
\theta|p\right)  }(f_{-\left(  \psi|q\right)  })^{-1},\left(  \theta|p\right)
^{+}=\left(  \psi|q\right)  ^{+}$, and $\xi\left(  \theta|p\right)
=\xi\left(  \psi|q\right)  .$
\end{theorem}

\begin{proof}
Part (i) is readily checked. Proof of (ii). (A) Begin by choosing
$L\in\mathbb{N}_{0}$ such that $\Pi(\theta|0)\cap E\Pi(\psi|L)\neq\emptyset$.
Note that $\Pi(\theta|0)\subset E\Pi(\psi|L)$. (B) Let $l\in\mathbb{N}_{0}$
with $l\geq L$. Using Corollary \ref{contain} we can choose $k=k_{l}$ so that
\begin{equation}
\Pi(\theta|k)\subset E\Pi(\psi|l)\subset\Pi(\theta|k+1) \label{containment}%
\end{equation}
(C) Using Theorem \ref{keythm2} and Corollary \ref{corollaryX}, we can apply
$\alpha^{\xi(\theta|k)}$ to both sides of $\Pi(\theta|k)\subset E\Pi(\psi|l)$
to obtain
\[
\alpha^{\xi(\theta|k)}\Pi(\theta|k)\subset\alpha^{\xi(\theta|k)}E\Pi(\psi|l)
\]
Writing $w=(\theta|k)^{+}$, $v=(\psi|l)^{+}$ and using the first part of
Theorem \ref{piusingts}, we now have%
\begin{align*}
s^{\xi(\theta|k)}f_{-\left(  \theta|k\right)  }T_{0}^{w}  &  \subset
s^{\xi(\theta|k)}Ef_{-\left(  \psi|l\right)  }s^{\xi(\psi|l)-\xi(\theta
|k)}T_{\xi(\psi|l)-\xi\left(  \theta|k\right)  }^{v}\\
&  \Rightarrow s^{-\xi(\psi|l)+\xi(\theta|k)}\left(  f_{-\left(
\psi|l\right)  }\right)  ^{-1}E^{-1}f_{-\left(  \theta|k\right)  }T_{0}%
^{w}\subset T_{\xi(\psi|l)-\xi\left(  \theta|k\right)  }^{v}%
\end{align*}
Now apply the Theorem \ref{keythm} to conclude that there is $\sigma\in
\Lambda_{\xi(\psi|l)-\xi(\theta|k)}^{v,w}$ with $\sigma^{+}=v$ and $\sigma
^{-}=w$ so that
\[
s^{-\xi(\psi|l)+\xi(\theta|k-1)}\left(  f_{-\left(  \psi|l\right)  }\right)
^{-1}E^{-1}f_{-\left(  \theta|k-1\right)  }T_{0}^{w}=s^{-\xi(\psi
|l)+\xi(\theta|k-1)}f_{\sigma}T_{0}^{w}%
\]
This implies
\[
E=f_{-\left(  \theta|k\right)  }f_{\sigma}^{-1}\left(  f_{-\left(
\psi|l\right)  }\right)  ^{-1}%
\]
We also have $E\Pi(\psi|l)\subset\Pi(\theta|k+1)$ which, following the same
steps, yields%
\[
E=f_{-\left(  \theta|k+1\right)  }f_{\tilde{\sigma}}\left(  f_{-\left(
\psi|l\right)  }\right)  ^{-1}%
\]
for some $\tilde{\sigma}\in\Lambda_{\xi(\theta|k)-\xi(\psi|l)}^{x,y}$ where
$x=\theta_{k+1}^{+},y=\psi_{l}^{+}=v$. Comparing the two expression for $E$ we
conclude%
\begin{align*}
f_{-\left(  \theta|k+1\right)  }f_{\tilde{\sigma}}\left(  f_{-\left(
\psi|l\right)  }\right)  ^{-1}  &  =f_{-\left(  \theta|k\right)  }f_{\sigma
}^{-1}\left(  f_{-\left(  \psi|l\right)  }\right)  ^{-1}\\
&  \Rightarrow f_{-\theta_{k}}=f_{\sigma}^{-1}f_{\tilde{\sigma}}^{-1}%
\end{align*}
which implies either $\tilde{\sigma}=\emptyset$, $\sigma=\theta_{k},$ and
$v=w$, or $\sigma=\emptyset$ and $\tilde{\sigma}=\theta_{k}$ and $w=x$. It
follows that either $E=f_{-(\theta|k)}\left(  f_{-\psi|l}\right)  ^{-1}$ or
$f_{-(\theta|k+1)}\left(  f_{-\psi|l}\right)  ^{-1}$. That is, one or other of
the two inclusion symbols in (\ref{containment}) can be replaced by an
equality sign. It follows that either $E=f_{-(\theta|k)}\left(  f_{-\psi
|l}\right)  ^{-1}$ where $\xi(\theta|k)=\xi(\psi|l)$ or $f_{-(\theta
|k+1)}\left(  f_{-\psi|l}\right)  ^{-1}$ where $\xi(\theta+1|k)=\xi(\psi|l)$.
(D) The rest of the proof follows from the arbitrarily large size of $l$.
\end{proof}

\begin{corollary}
\label{aperithm}If $(\mathcal{F}$,$\mathcal{G)}$ is rigid (with respect to
$\mathcal{U)}$ then $\Pi(\theta)=E\Pi(\theta)$ for some $E\in\mathcal{U}$ and
$\theta\in\Sigma_{\infty}^{\dag}$, if and only if $E=Id$. In particular, if
$\mathcal{U}$ contains the group of Euclidean translations on $\mathbb{R}%
^{M},$ then $\Pi(\theta)$ is non-periodic for all $\theta\in\Sigma_{\infty
}^{\dag}$.
\end{corollary}

\section{\label{notrigidsec}Inflation and deflation of tilings which may not
be rigid}

In this section we explore consequences of $\Pi(\theta)=E\Pi(\psi)$ without
requiring rigidity. An example of what we do require is: if $\Pi(\theta
)=E\Pi(\psi),$ where $\theta,\psi\in\Sigma_{\infty}^{\dag}$ and $E\in
\mathcal{U}$ are known, then $\alpha^{n}\Pi(\theta)=\alpha^{n}(E\Pi(\psi))$
for all $n$. In this example $\alpha$ acts on graphs of functions as described
earlier, but the resulting tilings on both sides of the equation coincide.
This is always true when the system is rigid. But it occurs more commonly as
illustrated by the following examples.

\begin{example}
Let $V=1$ with $\mathcal{F=\{}\mathbb{R}^{1};f_{1}(x)=x/2,f_{2}(x)=(x+1)/2\}$.
Tilings $\Pi(\theta)$ for $\theta\in\Sigma_{\infty}^{\dag}$ take one of three
forms: either (i) $\Pi(\theta)=\{\mathbb{[}n/2,(n+1)/2]:n=...-2,-1,0,1,2...\}$
or (ii) $\Pi(\theta)$ is a translation of the tiling $\mathbb{\{[}%
n/2,(n+1)/2]:n\in\mathbb{N}_{0}\}$ or (iii) it is an integer translation of
$\mathbb{\{[}-(n+1)/2,-n/2]:n\in\mathbb{N}_{0}\}$. Also $\alpha\Pi(\theta)$
takes the form of $\Pi(\theta)$ and if $\Pi(\theta)=\Pi(\psi),$ then
$\alpha\Pi(\theta)=\alpha\Pi(\psi)$.
\end{example}

\begin{example}
\label{lastex} Let $V=1$ and $\left\vert \mathcal{E}\right\vert \mathcal{=}2$
in $\mathbb{R}^{3}$ with $f_{1}$,$f_{2}$ defined respectively by
\[%
\begin{bmatrix}
0 & -s & 0\\
s & 0 & 0\\
0 & 0 & \frac{1}{2}%
\end{bmatrix}%
\begin{bmatrix}
x\\
y\\
z
\end{bmatrix}
+%
\begin{bmatrix}
0\\
s\\
0
\end{bmatrix}
,%
\begin{bmatrix}
s^{2} & 0 & 0\\
0 & -s^{2} & 0\\
0 & 0 & \frac{1}{2}%
\end{bmatrix}%
\begin{bmatrix}
x\\
y\\
z
\end{bmatrix}
+%
\begin{bmatrix}
0\\
1\\
\frac{1}{2}%
\end{bmatrix}
\]
where $s^{2}+s^{4}=1$, $0<s.$ See Figure \ref{beeplayersa}. It is easy to see
that, if $\theta$ and $\psi$ $\in\Sigma_{\infty}^{\dag}$ and $\Pi(\theta
)=E\Pi(\psi)$ for some translation $E,$ then $\alpha^{K}\Pi(\theta)=\alpha
^{K}\left(  E\Pi(\psi)\right)  $ for all $K\in\mathbb{N}$.
\end{example}

\begin{figure}[ptb]
\centering\includegraphics[
height=2.3246in,
width=2.0216in
]{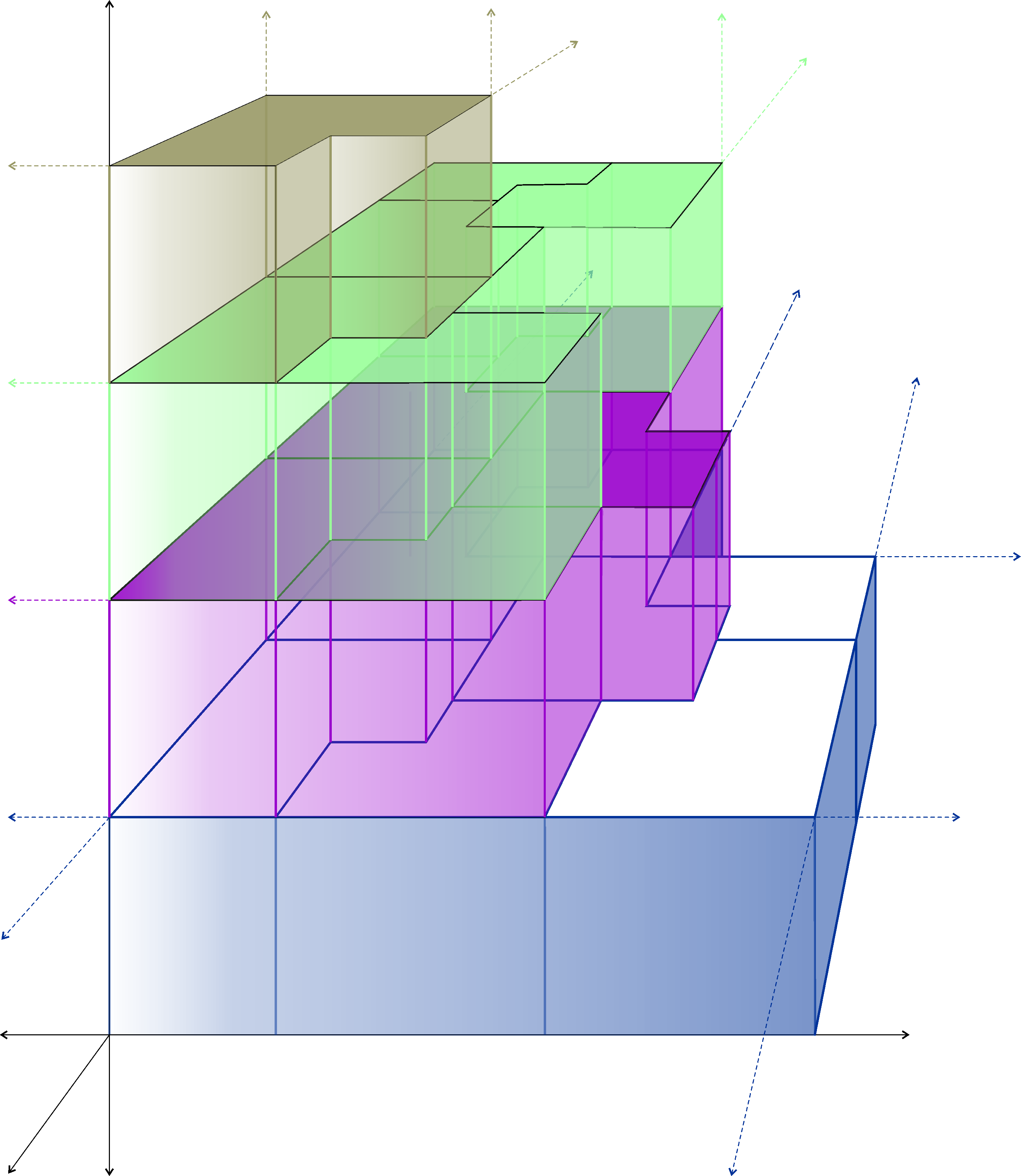}\caption{Illustration related to a 3D tiling that is golden
b in two directions and 0.5 scalings in z direction. See \cite{polygon} for
discussion of golden b tilings. See Example \ref{lastex} for the IFS.}%
\label{beeplayersa}%
\end{figure}

In the rest of this section $\mathcal{U}$ is simply a set of isometries on
$\mathbb{R}^{M}$.

\begin{definition}
\label{welldef}If $\Pi(\theta)=E\Pi(\psi)$ implies $\alpha^{k}\Pi
(\theta)=\alpha^{k}E\Pi(\psi)$ for all $k\in\mathbb{N}$, for all $\theta
,\psi\in\Sigma_{\infty}^{\dag},$ $E\in\mathcal{U}$, $i,j\in\mathbb{N}_{0},$
then we say that $(\mathcal{F}$,$\mathcal{G)}$ is \textbf{well-behaved} (with
respect to the set of isometries $\mathcal{U}$).
\end{definition}

\begin{theorem}
\label{lastthm}Let $(\mathcal{F}$,$\mathcal{G)}$ be a well-behaved tiling IFS.
If $\Pi(\theta)=E\Pi(\psi)$ for some isometry $E\in\mathcal{U}$ and
$\theta,\psi\in\Sigma_{\infty}^{\dag},$ then there are $p,q\in\mathbb{N}$,
$h,e\in\mathcal{E}$, $l\in\{0,1,...,a_{\max}-1\}$ so that
\[
E=f_{-\theta|p}s^{l}f_{h}V_{h^{+}e^{+}}f_{e}^{-1}f_{-\psi|q}^{-1}%
\]
where $h^{-}=\theta_{p+1}^{+}$, $e^{-}=\psi_{q}^{+}$ and $V_{h^{+}e^{+}}$ is a
similitude such that $V_{h^{+}e^{+}}A_{e^{+}}=A_{h^{+}}$.
\end{theorem}

\begin{proof}
Since $\Pi(\theta)=E\Pi(\psi)$ and $\alpha$ is well-behaved we have
$\alpha^{\xi(\psi|q)}\Pi(\theta)=\alpha^{\xi(\psi|q)}E\Pi(\psi)$ for all
$q\in\mathbb{N}$. Let $\xi(\psi|q)=\xi(\theta|p)+m\leq\xi(\theta|p+1)$. Note
that $m=m(p,\theta,\psi)$ and $q=q(p,\theta,\psi)$. We calculate%
\begin{align*}
\alpha^{\xi(\psi|q)}E\Pi(\psi)  &  =s^{\xi(\psi|q)}Es^{-\xi(\psi|q)}%
\alpha^{\xi(\psi|q)}\Pi(\psi)=s^{\xi(\psi|q)}Ef_{-(\psi|q)}\Pi(S^{q}\psi)\\
\alpha^{\xi(\psi|q)}\Pi(\theta)  &  =\alpha^{m}\alpha^{\xi(\theta|p)}%
\Pi(\theta)=\alpha^{m}s^{\xi(\theta|p)}f_{-(\theta|p)}\Pi(S^{p}\theta)\\
&  =s^{\xi(\theta|p)}s^{m}f_{-(\theta|p)}s^{-m}\alpha^{m}\Pi(S^{p}\theta)\\
&  =s^{\xi(\psi|q)}f_{-(\theta|p)}s^{-m}\alpha^{m}\Pi(S^{p}\theta)
\end{align*}
In the following, the sequences of unions are increasing unions.%
\begin{align*}
s^{m}\left(  f_{-(\theta|p)}\right)  ^{-1}Ef_{-(\psi|q)}\Pi(S^{q}\psi)  &
=\alpha^{m}\Pi(S^{p}\theta)\\
&  =\alpha^{m}\Pi(S^{p}\theta|j)\cup\alpha^{m}\Pi(S^{p}\theta|j+1)...\\
&  =\alpha^{m}\Pi(\theta_{p+1})\cup\alpha^{m}\Pi(\theta_{p+1}\theta_{p+2}%
)\cup...\\
&  =\alpha^{m}f_{-\theta_{p+1}}s^{\xi(\theta_{p+1})}T_{\xi(\theta_{p+1}%
)}^{\theta_{p+1}^{+}}\cup\alpha^{m}f_{-\theta_{p+1}\theta_{p+2}}s^{\xi
(\theta_{p+1}\theta_{p+2})}T_{\xi(\theta_{p+1}\theta_{p+2})}^{\theta_{p+2}%
^{+}}..\\
&  =s^{m}f_{-\theta_{p+1}}s^{\xi(\theta_{p+1})-m}T_{\xi(\theta_{p+1}%
)-m}^{\theta_{p+1}^{+}}\cup...
\end{align*}
We also have
\[
s^{m}\left(  f_{-(\theta|q)}\right)  ^{-1}Ef_{-(\psi|q)}\Pi(S^{q}\psi
)=s^{m}\left(  f_{-(\theta|q)}\right)  ^{-1}Ef_{-(\psi|q)}T_{0}^{\psi_{q}^{+}%
}\cup...
\]
By choosing $P$ sufficiently large, the following equivalent statements hold
for all $p\geq P$ :%
\begin{align*}
&  s^{m}\left(  f_{-(\theta|p)}\right)  ^{-1}Ef_{-(\psi|q)}T_{0}^{\psi_{q}%
^{+}}\text{ meets }s^{m}f_{-\theta_{p+1}}s^{\xi(\theta_{p+1})-m}T_{\xi
(\theta_{p+1})-m}^{\theta_{p+1}^{+}}\\
&  s^{m-\xi(\theta_{p+1})}\left(  f_{-(\theta|p+1)}\right)  ^{-1}%
Ef_{-(\psi|q)}T_{0}^{\psi_{q}^{+}}\text{ meets }T_{\xi(\theta_{p+1}%
)-m}^{\theta_{p+1}^{+}}%
\end{align*}
It follows that, for some $e,h\in\mathcal{E}$, $l=\xi(\theta_{q+1})-m$, we
must have:
\begin{align*}
s^{-l}\left(  f_{-(\theta|p+1)}\right)  ^{-1}Ef_{-(\psi|q)}f_{e}(A_{e^{+}})
&  =f_{h}(A_{h^{+}})\text{ where }h^{-}=\theta_{p+1}^{+}\text{, }e^{-}%
=\psi_{q}^{+}\\
f_{-h}s^{-l}\left(  f_{-(\theta|p+1)}\right)  ^{-1}Ef_{-(\psi|q)}f_{e}  &
=V_{h^{+},e^{+}}\text{ where }V_{h^{+},e^{+}}A_{e^{+}}=A_{h^{+}}\text{ is a
similitude;}\\
E  &  =f_{-(\theta|p+1)}s^{l}f_{h}V_{h^{+},e^{+}}f_{-e}\left(  f_{-(\psi
|q)}\right)  ^{-1}%
\end{align*}
Replacing $p+1$ by $p$ yields the formula for $E$ in the statement of the theorem.
\end{proof}

Notice how this result is consistent with Theorem \ref{basictheorem} because
if the system is rigid, then $s^{l}f_{h}V_{h^{+},e^{+}}f_{-e}$ must be the
identity. It has a nice interpretation in terms of the possibilities for
Example \ref{lastex}: translations by $\frac{\pm1}{2}$ map any tiling of
$\mathbb{R}$ into the same tiling, and correspond to the fact that in this
case $T_{0}$ meets $T_{0}\pm\frac{1}{2}$.

\section{\label{relationsec}Relationship with the self-similar tiling spaces
of Anderson and Putnam \cite{anderson}}

Here we construct the tiling spaces of Anderson and Putnam \cite{anderson}
(A\&P) and relate them to the tilings in this paper. We recall the terminology
of A\&P with adjustments so that their setting intersects ours. In this
Section a \textit{tile} is homeomorphic to a closed ball in $\mathbb{R}^{M}$,
a \textit{partial tiling} is a set of tiles with disjoint interiors, a
\textit{tiling }is a partial tiling with support $\mathbb{R}^{M}$, and
$\mathcal{U}$ is the set of euclidean translations on $\mathbb{R}^{M}$.

A\&P present the following general description of the kind of tilings they
consider, in the introduction to \cite{anderson}. `A tiling of $(\mathbb{R}%
^{M}=)\mathbb{R}^{d}$ is a cover of $\mathbb{R}^{d}$ by sets, each of which is
a translation of one of the prototiles ..., so that they overlap only on their
boundaries. We also assume a substitution rule: we have a constant
$(s^{-1}=)\lambda>1$ and, for each (A\&P) prototile, a rule for subdividing it
into pieces, each of which is another prototile, scaled down by a factor
$\lambda^{-1}.$

Next we report the construction of the A\&P tiling space $\mathbb{T}_{A\&P}$.
We refer to the sets that A\&P call prototiles as A\&P-prototiles. For any
partial tiling $T$, expansions $\lambda$ and translations $u$ are defined by
A\&P according to%
\begin{align*}
\lambda T  &  =\{\lambda t:t\in T\}\text{ for }\lambda\in(1,\infty)\\
u\left(  T\right)   &  =\{u\left(  t\right)  :t\in T\}\text{ for all }%
u\in\mathcal{U}%
\end{align*}
In A\&P $u\in\mathcal{U}$ is represented by $u\in\mathbb{R}^{M}$ and $u\left(
t\right)  =t+u$. The collection of tilings $\mathbb{T}_{A\&P}$ is defined as
follows. All tiles in a tiling in $\mathbb{T}_{A\&P}$ are translations of a
finite set of \textit{(A\&P-)prototiles} $\{\hat{p}_{i}:i=1,2,...,n_{pro}\}.$
Let $\widehat{\mathbb{T}}_{A\&P}$ be the collection of all partial tilings
that only contain translations of these prototiles. Assume there is a number
$\lambda>1$ and a \textit{substitution} rule\textit{ }that associates to each
$\hat{p}_{i}$ a partial tiling $P_{i}$ with support $\hat{p}_{i}$ such that
$\lambda P_{i}$ is in $\widehat{\mathbb{T}}_{A\&P}$. An inflation map
$\hat{\omega}:\widehat{\mathbb{T}}_{A\&P}\rightarrow\widehat{\mathbb{T}%
}_{A\&P}$ is defined by
\[
\hat{\omega}(T)=\lambda\bigcup\limits_{u(\hat{p}_{i})\in T}u\left(
P_{i}\right)
\]
The tiling space $\mathbb{T}_{A\&P}$ is the collection of tilings $T$ in
$\widehat{\mathbb{T}}_{A\&P}$ such that for any $P\subset T$ with bounded
support, we have $P\subset\hat{\omega}^{n}(u(\hat{p}_{i}))$ for some $n,i,$and
$u$. Let $\omega=\hat{\omega}|_{\mathbb{T}_{A\&P}}.$ A\&P point out that their
definitions of the tiling space $\mathbb{T}_{A\&P}$ and the operator $\omega,$
are adapted from the standard ones for symbolic substitution dynamical
systems, for example in \cite{mozes}, and are similar to the usage by
\cite{radin1}. It is also the same as the definition in \cite{solomyak2}. But
care must be taken with all such assertions of equivalence. For example, here
we do not consider either labelled tiles or tiles with adornments.

The partial tilings $\left\{  P_{i}\right\}  $ of the prototiles $\left\{
\hat{p}_{i}\right\}  $ define a graph IFS $\mathcal{\{F},\mathcal{G\}}$ in the
following way. The vertices of $\mathcal{G}$ correspond to the
A\&P-prototiles, one for each $\hat{p}_{v}$, $v=1,2,...,n_{pro}$. There is one
directed edge $e$ of $\mathcal{G}$ from vertex $v$ to $w$ for each distinct
$u\in\mathcal{U}$ such that $\lambda^{-1}\hat{p}_{w}+u\subset P_{v}.$ The
result is a directed graph $\mathcal{G}$ and a set of similitudes
$\mathcal{F}$ so that
\[
P_{v}=\{f_{e}(\hat{p}_{e^{+}}):e^{-}=v\}\text{, }\hat{p}_{v}=\bigcup
\limits_{e^{-}=v}f_{e}(\hat{p}_{e^{+}}),\text{ }f_{e}(x)=\lambda^{-1}%
x+u_{e},\text{ }u_{e}\in\mathbb{R}^{M}%
\]
We have $A=\cup_{v}\hat{p}_{v}$ is the attractor of $\mathcal{\{F}%
,\mathcal{G\}}$ and $\left\{  A_{v}=\hat{p}_{v}:v=1,2,...,n_{pro}\right\}  $
are its components. Using the construction following Definition \ref{defONE}
in Subsection \ref{TIFSsec}, we can assume without loss of generality that the
components $A_{v}=\hat{p}_{v}$ are disjoint. Also $\mathcal{F}$ obeys the OSC,
as can be seen by choosing the open sets $\mathcal{O}_{v}$ to be the interiors
of $\hat{p}_{v}$ for all $v\in\mathcal{V}$. Provided that the A\&P system is
primitive as defined below, see (2) below, $\mathcal{G}$ is strongly connected
and primitive as defined earlier. In this way the partial tilings $P_{i}$ of
the prototiles $\hat{p}_{i}$ define a tiling IFS $\mathcal{\{F},\mathcal{G\}}$.

A\&P require that $(\mathbb{T}_{A\&P},\omega,\mathcal{U})$ have these three properties:

\begin{enumerate}
\item $\omega:\mathbb{T}_{A\&P}\rightarrow\mathbb{T}_{A\&P}$ is bijective;

\item the substitution is \textit{primitive} (there is a fixed positive
integer $N_{0}$ such that for each pair of prototiles $\hat{p}_{i}$ and
$\hat{p}_{j}$, there exists $u\in\mathcal{U}$ so that the partial tiling
$\hat{\omega}^{N_{0}}(\{\hat{p}_{i}\})$ contains $u(\hat{p}_{j})$);

\item $\mathbb{T}_{A\&P}$ satisfies a \textit{finite pattern condition}: for
each $r>0,$ there are only finitely many partial tilings up to translation
that are subsets of tilings in $\mathbb{T}_{A\&P}$ and whose supports have
diameters less than $r$.
\end{enumerate}

Condition 1 is equivalent to \textit{recognizability }as referred to by A\&P
and as defined by \cite{solomyak2}. See also \cite{berthe} and references therein.

In other works, see \cite{solomyak2}, tilings are defined by starting from a
self-similar tiling $T$ of $\mathbb{R}^{M}$ and then taking the closure of the
set of all translations of $T$. A\&P prove that the resulting collection of
tilings is the same as $\mathbb{T}_{A\&P}$. This leads us to the following
question. How is $\mathbb{T}_{A\&P}$ related to the collection of tilings
$\mathbb{T}$ defined in this paper?

To relate the two contexts note that our prototiles are related to
A\&P-prototiles by $p_{v}=\lambda^{-1}\hat{p}_{v}$ for $v=1,2,...,\left\vert
\mathcal{V}\right\vert $ and $s=\lambda^{-1}$. In the present setting, where
tiles have nonempty interiors, note that $\Sigma_{rev}^{\dag}$ is the set of
$\theta\in\Sigma_{\infty}^{\dag}$ such that the support of $\Pi(\theta)$ is
all of $\mathbb{R}^{M}.$

\begin{theorem}
\label{AandPtheorem} Let $\left(  \mathcal{F},\mathcal{G}\right)  $ be a rigid
tiling IFS defined by the partial tilings $P_{i}$ of the sets $\hat{p}_{i}$
$\in\mathbb{T}_{A\&P}$, let $\left\vert P_{i}\right\vert >1$ for all $i$, and
let A\&P's conditions (1) and (2) hold. Then
\[
\mathbb{T}_{A\&P}=\{\lambda u(\Pi(\theta)):\theta\in\Sigma_{rev}^{\dag}%
,u\in\mathcal{U}\}
\]

\end{theorem}

We will need the following observation.

\begin{proposition}
\label{proposition1} Let $\left(  \mathcal{F},\mathcal{G}\right)  $ be a rigid
tiling IFS with $a_{\max}=1.$ Let $\Pi(\theta)\subset E\Pi(\psi)$ for some
$\theta,\psi\in\Sigma_{\ast}^{\dag},E\in\mathcal{U}$. Then $E\Pi(\psi
)=\Pi(\theta\tilde{\psi})$ for some $\tilde{\psi}\in\Sigma_{\ast}^{\dag}$ such
that $\tilde{\psi}^{-}=\theta^{+}$, $\psi^{+}=\tilde{\psi}^{+},\left\vert
\psi\right\vert =\left\vert \theta\right\vert +\left\vert \tilde{\psi
}\right\vert $ .
\end{proposition}

\begin{proof}
[Proof of Proposition \ref{proposition1}]%
\begin{align}
\Pi(\theta)  &  \subset E\Pi(\psi)\nonumber\\
&  \Longrightarrow\alpha^{\left\vert \theta\right\vert }\Pi(\theta
)\subset\alpha^{\left\vert \theta\right\vert }E\Pi(\psi)\label{middle}\\
&  \Longrightarrow s^{\left\vert \theta\right\vert }f_{-\theta}T_{0}%
^{\theta^{+}}\subset s^{\left\vert \theta\right\vert }Ef_{-\psi|\left(
\left\vert \theta\right\vert \right)  }\Pi(S^{\left\vert \theta\right\vert
}\psi)=\tilde{E}\Pi(\tilde{\psi})\nonumber
\end{align}
where $\tilde{E}:=s^{\left\vert \theta\right\vert }Ef_{-\psi|\left(
\left\vert \theta\right\vert \right)  }\in\mathcal{U}$ and $\tilde{\psi
}:=S^{\left\vert \theta\right\vert }\psi.$ But%
\[
\tilde{E}\Pi(\tilde{\psi})=\tilde{E}\{f_{-\tilde{\psi}}f_{\omega}%
(T_{0}^{\omega^{+}}):\omega^{-}=\tilde{\psi}^{+},\left\vert \omega\right\vert
=\left\vert \tilde{\psi}\right\vert \}
\]
so by rigidity there is some $\omega\in\Sigma_{\ast}^{\dag}$ with $\omega
^{-}=\tilde{\psi}^{+}$ and $\left\vert \omega\right\vert =\left\vert
\tilde{\psi}\right\vert ,$ such that
\begin{align*}
s^{\left\vert \theta\right\vert }f_{-\theta}T_{0}^{\theta^{+}}  &  =\tilde
{E}f_{-\tilde{\psi}}f_{\omega}(T_{0}^{\omega^{+}})\\
&  \Longrightarrow f_{-\theta}T_{0}^{\theta^{+}}=Ef_{-\psi|\left(  \left\vert
\theta\right\vert \right)  }f_{-\tilde{\psi}}f_{\omega}(T_{0}^{\omega^{+}})\\
&  \Longrightarrow f_{-\theta}T_{0}^{\theta^{+}}=Ef_{-\psi|\left(  \left\vert
\theta\right\vert \right)  }f_{-\tilde{\psi}}f_{\omega}(T_{0}^{\omega^{+}})\\
&  \Longrightarrow\left(  f_{-\psi|\left(  \left\vert \theta\right\vert
\right)  }f_{-\tilde{\psi}}f_{\omega}\right)  ^{-1}f_{-\theta}=E
\end{align*}
where we have again used rigidity to deduce the last implication. We now
substitute back into Equation \ref{middle} to obtain
\[
\alpha^{\left\vert \theta\right\vert }E\Pi(\psi)=\alpha^{\left\vert
\theta\right\vert }\left(  f_{-\psi|\left(  \left\vert \theta\right\vert
\right)  }f_{-\tilde{\psi}}f_{\omega}\right)  ^{-1}\Pi(\psi)
\]
and applying $\alpha^{-\left\vert \theta\right\vert }$ to both sides we get
\[
E\Pi(\psi)=\left(  f_{-\psi|\left(  \left\vert \theta\right\vert \right)
}f_{-\tilde{\psi}}f_{\omega}\right)  ^{-1}\Pi(\psi)=\Pi(\theta\tilde{\psi})
\]
as stated in the Theorem.
\end{proof}

\begin{proof}
[Proof of Theorem \ref{AandPtheorem}]Let $T\in\mathbb{T}_{A\&P}$. Let $r>0$.
Let $T_{r}$ be the partial tiling $T_{r}:=\{t\in T:t\cap B_{r}(O)\neq
\varnothing\}$ where $B_{r}(O)$ is the open ball of radius $r$. Let $r_{1}=1$.
Then $T_{1}\subset\hat{\omega}^{n_{1}}(u_{1}(\hat{p}_{i_{1}}))$ for some
$n_{1},u_{1},i_{1}$. Now choose $r_{2}>r_{1}$ so that $\hat{\omega}^{n_{1}%
}(u_{1}(\hat{p}_{i_{1}}))\subset T_{r_{2}}$ and choose $n_{2},i_{2},u_{2}$ so
that $T_{2}\subset\hat{\omega}^{n_{2}}(u_{2}(\hat{p}_{i_{2}}))$. Proceeding in
this manner we find%
\[
T_{r_{1}=1}\subset\hat{\omega}^{n_{1}}(u_{1}(\hat{p}_{i_{1}}))\subset
T_{r_{2}}\subset\hat{\omega}^{n_{2}}(u_{2}(\hat{p}_{i_{2}}))\subset T_{r_{3}%
}\subset\hat{\omega}^{n_{3}}(u_{2}(\hat{p}_{i_{3}}))\subset...
\]
Hence we can rewrite $T$ as a the union of a strictly increasing sequence of
partial tilings,%
\[
T=\bigcup\limits_{k\in\mathbb{N}}T_{r_{k}}=\bigcup\limits_{k\in\mathbb{N}}%
\hat{\omega}^{n_{k}}(u_{k}(\hat{p}_{i_{k}}))
\]
for some sequence $\left(  n_{k},u_{k},i_{k}\right)  $. Since the sequence
$\left\{  \hat{\omega}^{n_{k}}(u_{k}(\hat{p}_{i_{k}}))\right\}  $ is
increasing (nested), we can replace the sequence $\left\{
n=1,2,3,...\right\}  $ by any infinite subsequence of it. Also, let $\hat
{p}_{v}$ be such that $\hat{p}_{v}=\hat{p}_{i_{k}}$ for infinitely many values
of $k$. It follows that there is an infinite subsequence $(n_{k_{n}},u_{k_{n}%
})$ such that%
\[
T=\bigcup\limits_{n\in\mathbb{N}}\hat{\omega}^{n_{k_{n}}}(u_{k_{n}}(\hat
{p}_{v}))=\bigcup\limits_{n\in\mathbb{N}}(s^{-1}u_{k_{n}}s\hat{\omega
}^{n_{k_{n}}}\hat{p}_{v})=s^{-1}\bigcup\limits_{n\in\mathbb{N}}(u_{k_{n}%
}T_{n_{k_{n}}-1}^{v})
\]
It follows that there is a sequence $\left\{  \theta^{\left(  k_{n}\right)
}\in\Sigma_{\ast}^{\dag}:\left\vert \theta^{\left(  k_{n}\right)  }\right\vert
=k_{n},\left(  \theta^{\left(  k_{n}\right)  }\right)  ^{-}=v\right\}  $ and a
sequence of translations $\left\{  E_{k_{n}}\in\mathcal{U}\right\}  $ so that
$sT$ can be written as the increasing union%
\[
sT=\bigcup\limits_{n\in\mathbb{N}}E_{k_{n}}\Pi(\theta^{\left(  k_{n}\right)
})
\]
We now apply Proposition \ref{proposition1} repeatedly to deduce that there
are unique $E\in\mathcal{U}$ and $\theta\in\Sigma^{\dag}$ such that
$E=E_{k_{n}}$ and $\theta|k_{n}=\theta^{\left(  k_{n}\right)  }$ for all $n$
\[
sT=\bigcup\limits_{n\in\mathbb{N}}E\Pi(\theta|n)=E\Pi(\theta)
\]
This completes the proof that $\mathbb{T}_{A\&P}\subset\{\lambda u(\Pi
(\theta)):\theta\in\Sigma_{rev}^{\dag},u\in\mathcal{U}\}.$ To prove the
inclusion the other way round, suppose that $u\in\mathcal{U}$ and $\Pi
(\theta)$ with $\theta\in\Sigma_{rev}^{\dag}$ is given. Since $\theta\in
\Sigma_{rev}^{\dag},$ $\Pi(\theta)$ is supported on $\mathbb{R}^{M}%
=\mathbb{R}^{d}$. Then we need to show that there is $T\in\mathbb{T}_{A\&P}$
such that $T=u\lambda\Pi(\theta)$. We show instead that there is $T^{\prime
}\in\mathbb{T}_{A\&P}$ such that $T^{\prime}=\lambda\Pi(\theta)$ because then,
by \cite[Corollary 3.5]{anderson}, $\mathbb{T}_{A\&P}$ contains all
translations of any tiling that it contains. Let $P$ be a patch in $\Pi
(\theta)$. Then $P\subset\Pi(\theta|k)$ for some $k$. We show that $\Pi
(\theta|k)=s\hat{\omega}^{k+1}(u(\hat{p}_{v}))$ for some $u$ and $P_{v}$. But
\begin{align*}
\Pi(\theta|k)  &  =f_{-\left(  \theta|k\right)  }s^{k}T_{k}^{(\theta|k)^{+}%
}=f_{-\left(  \theta|k\right)  }s^{k}\alpha^{-k}T_{0}^{(\theta|k)^{+}}\\
&  =\alpha^{-k}f_{-\left(  \theta|k\right)  }s^{k}T_{0}^{(\theta|k)^{+}%
}=\alpha^{-k}uT_{0}^{(\theta|k)^{+}}\\
&  =s\hat{\omega}^{k+1}(u(\hat{p}_{v}))
\end{align*}
where $u=f_{-\left(  \theta|k\right)  }s^{k}\in\mathcal{U}$ and $v=(\theta
|k)^{+}$. Since the patch $P$ of $\Pi(\theta)$ is arbitrary, it follows that
$\Pi(\theta)\in s\mathbb{T}_{A\&P}$.
\end{proof}

If $(\mathcal{F}$,$\mathcal{G)}$ is rigid, then $\alpha$ and $\alpha^{-1}$ are
well-defined on tilings. For the case where $a_{\max}=1$, tiles have non-empty
interiors, and $\mathcal{U}$ is translations, this means if $(\mathcal{F}%
$,$\mathcal{G)}$ is rigid, then all tilings of $\mathbb{R}^{M}$ in
$\mathbb{T}_{\infty}$ are recognizable in the sense of $A\&P$ and
\cite{solomyak2}. But we do not know whether or not, in the same setting,
recognizability implies rigidity.

\begin{acknowledgement}
We thank Christoph Bandt for advice, discussions, and help.
\end{acknowledgement}

\end{document}